\documentclass[12pt, reqno]{amsart}

\usepackage{amsmath,amssymb,amsthm,amsfonts,verbatim}
\usepackage{microtype}
\usepackage[all,2cell]{xy}
\usepackage{mathtools}
\usepackage{graphicx}
\usepackage{pinlabel}
\usepackage{hyperref}
\usepackage{mathrsfs}
\usepackage{color}
\usepackage[dvipsnames]{xcolor}
\usepackage{enumitem}
\usepackage{cite}
\usepackage{soul}
\usepackage{MnSymbol}

\CompileMatrices

\usepackage[top=1.2in,bottom=1.2in,left=1in,right=1in]{geometry}

\usepackage{hyperref} 
\hypersetup{          
	colorlinks=true, breaklinks, linkcolor=[RGB]{51 102 204}, filecolor=Orchid, urlcolor=[RGB]{51 102 204},
	citecolor=Orchid, linktoc=all, }
\usepackage[nameinlink]{cleveref}

\theoremstyle{plain}
\newtheorem{theorem}{Theorem}[section]
\newtheorem{maintheorem}{Theorem}

\newtheorem{maincor}[maintheorem]{Corollary}
    \AddToHook{env/maincor/begin}{\crefalias{maintheorem}{maincor}}
    \crefname{maincor}{corollary}{corollaries}

    \AddToHook{env/problem/begin}{\crefalias{theorem}{problem}}
    \crefname{problem}{problem}{problems}
\newtheorem{proposition}[theorem]{Proposition}
    \AddToHook{env/proposition/begin}{\crefalias{theorem}{proposition}}
    \crefname{proposition}{proposition}{propositions}
\newtheorem{lemma}[theorem]{Lemma}
    \AddToHook{env/lemma/begin}{\crefalias{theorem}{lemma}}
    \crefname{lemma}{lemma}{lemmas}
\newtheorem{conjecture}[theorem]{Conjecture}
    \AddToHook{env/conjecture/begin}{\crefalias{theorem}{conjecture}}
    \crefname{conjecture}{conjecture}{conjectures}

    \AddToHook{env/fact/begin}{\crefalias{theorem}{fact}}
    \crefname{fact}{fact}{facts}
\newtheorem{corollary}[theorem]{Corollary}
    \AddToHook{env/corollary/begin}{\crefalias{theorem}{corollary}}
    \crefname{corollary}{corollary}{corollaries}

    \AddToHook{env/claim/begin}{\crefalias{theorem}{claim}}
    \crefname{claim}{claim}{claims}

\theoremstyle{definition}
\newtheorem{definition}[theorem]{Definition}
    \AddToHook{env/definition/begin}{\crefalias{theorem}{definition}}
    \crefname{definition}{definition}{definitions}

    \AddToHook{env/example/begin}{\crefalias{theorem}{example}}
    \crefname{example}{example}{examples}
\newtheorem{construction}[theorem]{Construction}
    \AddToHook{env/construction/begin}{\crefalias{theorem}{construction}}
    \crefname{construction}{construction}{constructions}

    \AddToHook{env/exercise/begin}{\crefalias{theorem}{exercise}}
    \crefname{exercise}{exercise}{exercises}
\newtheorem{remark}[theorem]{Remark}
    \AddToHook{env/remark/begin}{\crefalias{theorem}{remark}}
    \crefname{remark}{remark}{remarks}

    \AddToHook{env/question/begin}{\crefalias{theorem}{question}}
    \crefname{question}{question}{questions}

    \AddToHook{env/answer/begin}{\crefalias{theorem}{answer}}
    \crefname{answer}{answer}{answers}
\newtheorem{convention}[theorem]{Convention}
    \AddToHook{env/convention/begin}{\crefalias{theorem}{convention}}
    \crefname{convention}{convention}{conventions}

\newcommand{\nc}{\newcommand}
\nc{\dmo}{\DeclareMathOperator}

\nc{\OO}{\mathcal{O}}
\nc{\cA}{\mathcal{A}}
\nc{\cB}{\mathcal{B}}
\nc{\sB}{\mathscr{B}}
\nc{\C}{\mathbb{C}}
\nc{\cC}{\mathcal{C}}
\nc{\BB}{\mathbb{B}}
\nc{\LL}{\mathcal{L}}
\nc{\bd}{\mathbf{d}}
\nc{\DD}{\mathbb{D}}
\nc{\cD}{\mathcal{D}}
\nc{\bF}{\mathbb{F}}
\nc{\cF}{\mathcal{F}}
\nc{\cG}{\mathcal{G}}
\nc{\cI}{\mathcal{I}}
\nc{\cK}{\mathcal{K}}
\nc{\cL}{\mathcal{L}}
\nc{\cM}{\mathcal{M}}
\nc{\bM}{\mathbf{M}}
\nc{\N}{\mathbb{N}}
\nc{\cN}{\mathcal{N}}
\nc{\cO}{\mathcal{O}}
\nc{\bp}{\mathbf{p}}
\nc{\cP}{\mathcal{P}}
\nc{\Q}{\mathbb{Q}}
\nc{\R}{\mathbb{R}}
\nc{\cS}{\mathcal{S}}
\nc{\cT}{\mathcal T}
\nc{\cU}{\mathcal U}
\nc{\cX}{\mathcal{X}}
\nc{\Z}{\mathbb{Z}}
\nc{\disk}{\mathbb{D}}
\nc{\hyp}{\mathbb{H}}
\renewcommand{\P}{\mathbb{P}}
\renewcommand{\O}{\mathcal O}
\nc{\CP}{\mathbb{CP}}
\nc{\RP}{\mathbb{RP}}
\dmo{\Mod}{Mod}
\dmo{\PMod}{PMod}
\dmo{\LMod}{LMod}
\dmo{\Diff}{Diff}
\dmo{\Homeo}{Homeo}
\dmo{\dist}{dist}
\dmo\BDiff{BDiff}
\dmo\SO{SO}
\dmo\Hom{Hom}
\dmo\SL{SL}
\dmo\rank{rank}
\dmo\sig{sig}
\dmo\Out{Out}
\dmo\Aut{Aut}
\dmo\Inn{Inn}
\dmo\GL{GL}
\dmo\PGL{PGL}
\dmo\Gr{Gr}
\dmo\PSL{PSL}
\dmo\BHomeo{BHomeo}
\dmo\EHomeo{EHomeo}
\dmo\EDiff{EDiff}
\dmo\Disc{Disc}
\dmo\Aff{Aff}
\dmo\Spin{Spin}

\renewcommand{\hat}{\widehat}
\renewcommand{\bar}{\overline}

\dmo\Teich{Teich}
\dmo\Fix{Fix}
\nc{\pair}[1]{\ensuremath{\left\langle #1 \right\rangle}}
\nc{\abs}[1]{\ensuremath{\left| #1 \right|}}
\nc{\action}{\circlearrowright}
\nc{\norm}[1]{\ensuremath{\left | \left | #1 \right | \right |}}
\nc{\abcd}[4]{\ensuremath{\left(\begin{array}{cc} #1 & #2 \\ #3 & #4 \end{array}\right)}}
\nc{\into}\hookrightarrow
\dmo{\Isom}{Isom}
\nc{\normal}{\vartriangleleft}
\dmo{\Vol}{Vol}
\dmo{\im}{Im}
\dmo{\Push}{Push}
\dmo{\Conf}{Conf}
\dmo{\UConf}{UConf}
\dmo{\PConf}{PConf}
\dmo{\Poly}{Poly}
\dmo{\PB}{PB}
\dmo{\id}{id}
\dmo{\Jac}{Jac}
\dmo{\Pic}{Pic}
\dmo{\Stab}{Stab}
\dmo{\Arf}{Arf}
\dmo{\End}{End}
\dmo{\Gal}{Gal}
\dmo{\lcm}{lcm}
\dmo{\ab}{ab}
\dmo{\opp}{op}
\dmo{\SU}{SU}
\dmo{\OT}{\Omega \mathcal{T}}
\dmo{\OM}{\Omega \mathcal{M}}
\dmo{\PH}{\mathbb{P}\mathcal{H}}
\dmo{\spin}{spin}
\dmo{\even}{even}
\dmo{\odd}{odd}
\dmo{\comp}{\mathcal{H}}
\dmo{\Mgk}{\mathcal{M}_{g, \underline{\kappa}}}
\dmo{\orb}{orb}
\dmo{\AJ}{AJ}
\dmo{\Ck}{\mathsf{C}(\underline{\kappa})}
\dmo{\Int}{Int}
\dmo{\pr}{pr}
\dmo{\lab}{lab}
\dmo{\Sym}{Sym}
\dmo{\Ann}{Ann}
\dmo{\Rad}{Rad}
\dmo{\Ind}{Ind}
\dmo{\Div}{Div}
\dmo{\Res}{Res}
\dmo{\Hur}{Hur}
\dmo{\vcd}{vcd}

\nc{\Span}[1]{\operatorname{Span}(#1)}
\newcommand{\onto}{\twoheadrightarrow}

\renewcommand{\epsilon}{\varepsilon}
\renewcommand{\tilde}{\widetilde}
\renewcommand{\le}{\leqslant}
\renewcommand{\coloneq}{\mathrel{\mathop:}\mkern-1.2mu=}
\nc{\margin}[1]{\marginpar{\scriptsize #1}}
\nc{\para}[1]{\medskip\noindent\textbf{#1.}}
\definecolor{myblue}{RGB}{102,153, 255}
\definecolor{myred}{RGB}{204,0,0}
\definecolor{mygreen}{RGB}{0,204,0}
\definecolor{myorange}{RGB}{255,102,0}
\definecolor{mypurple}{RGB}{138,43,226}
\definecolor{myyellow}{RGB}{255,204,0}
\nc{\red}[1]{\textcolor{myred}{#1}}
\nc{\blue}[1]{\textcolor{myblue}{#1}}

\nc{\Adm}{\mathbf{Adm}}
\nc{\Env}{\mathbf{Env}}
\dmo{\ord}{ord}

\author{Ishan Banerjee}
\author{Nick Salter}

\address{Ishan Banerjee: Math Tower, the Ohio State University, 231 W 18th St Columbus OH 43210}
\email{banerjee.238@osu.edu}

\address{Nick Salter: Department of Mathematics, University of Notre Dame, 255 Hurley Building, Notre Dame, IN 46556}
\email{nsalter@nd.edu}

\date{November 10, 2026}

\title{Monodromy and vanishing cycles for complete intersection curves}

\begin{document}

\begin{abstract}
    We compute the topological monodromy of every family of complete intersection curves. Like in the case of plane curves previously treated by the second-named author, we find the answer is given by the $r$-spin mapping class group associated to the maximal root of the adjoint line bundle. Our main innovation is a suite of tools for studying the monodromy of sections of a tensor product of very ample line bundles in terms of the monodromy of sections of the factors, allowing for an induction on (multi-)degree.
\end{abstract}

\maketitle

\section{Introduction}

Let $\bd := (d_1, \dots, d_{n-1})$ be a multidegree, and let $C \subset \CP^n$ be a smooth complete intersection curve of multidegree $\bd$. The family $U_\bd$ of all such curves admits a {\em topological monodromy representation}
\[
\rho: \pi_1(U_\bd) \to \Mod(\Sigma_{g(\bd)}),
\]
where $g(\bd)$ denotes the genus of such curves (see \Cref{lemma:multidegreegenus} for a formula) and $\Mod(\Sigma_{g(\bd)})$ denotes the mapping class group of $\Sigma_g$, i.e. isotopy classes of orientation-preserving diffeomorphisms. (Note that we are not modding out by automorphisms, so $U_\bd$ is a quasiprojective complex variety, and $\pi_1$ denotes the ordinary topological fundamental group). We define the {\em topological monodromy group}
\[
\Gamma_d \coloneq \im(\rho) \le \Mod(\Sigma_{g(\bd)}). 
\]

Understanding the topological monodromy group of families of curves is important in a wide variety of settings, from symplectic geometry \cite{donaldson} to number theory \cite{LV}. We resolve this here for any family of complete intersection curves.

The answer turns out to be completely governed by an elementary principle in algebraic geometry. By the adjunction formula, $K_C \cong \cO_C(r(\bd))$ for some integer $r(\bd)$ (again, see \Cref{lemma:multidegreegenus}). Thus $\cO_C(1)$ determines a canonical $r(\bd)$-spin structure $\phi_{\bd}$ on $C$. The topological monodromy of the family of smooth complete intersection curves of multidegree $\bd$ is therefore contained in an ``$r$-spin mapping class group'' $\Mod(\Sigma_{g(\bd)})[\phi_\bd]$ - the stabilizer of $\cO_C(1)$ under an action of the mapping class group (see \Cref{section:rspin} for details). The main result of this paper shows that for families of complete intersection curves, this is a complete characterization.

\begin{maintheorem}\label{mainthm:ambientPn}
    For all multidegrees $\bd$, the monodromy $\Gamma_\bd \le \Mod(\Sigma_{g(\bd)})$ of the family of smooth complete intersection curves of multidegree $\bd$ is the associated $r(\bd)$-spin mapping class group:
    \[
    \Gamma_{\bd} = \Mod(\Sigma_{g(\bd)})[\phi_\bd].
    \]
\end{maintheorem}
\begin{remark}\label{remark:lowedegree}
    In a handful of low-degree cases (enumerated in \Cref{lemma:smallgenustable}), $r(\bd) \le 1$, and so some care must be taken in understanding what is being asserted in \Cref{mainthm:ambientPn}. The case $r(\bd)< 0$ is equivalent to the case $g(\bd) = 0$, in which case the entire mapping class group is trivial. The case $r(\bd) = 0$ is equivalent to the condition $g(\bd) = 1$. Curves of genus $1$ are canonically framed, and indeed the entire mapping class group $\Mod(\Sigma_1)$ preserves this, so that in these cases, \Cref{mainthm:ambientPn} is asserting that $\Gamma_\bd = \Mod(\Sigma_1) = \SL_2(\Z)$. In the case $r(\bd) = 1$, there is no additional structure imposed by $\phi_\bd$, so the theorem specializes here to asserting $\Gamma_\bd = \Mod(\Sigma_{g(\bd)})$. In summary, in the cases where $r(\bd) \le 1$, \Cref{mainthm:ambientPn} asserts that $\Gamma_\bd = \Mod(\Sigma_{g(\bd)})$.
\end{remark}

Now suppose $n \ge 3$, and define $\bd' := (d_1, \dots, d_{n-2})$. Let $X \subset \CP^n$ be a smooth complete intersection surface of multidegree $\bd'$. In the family of complete intersection curves of multidegree $\bd$, there is the subfamily $U_{X,\bd}$ of smooth complete intersections $C = X \cap Y$, where $Y$ is some hypersurface of degree $d_{n-1}$. Again by adjunction, $\cO(1)$ induces a canonical $r(\bd)$-spin structure on any such $C$, which we continue to denote by $\phi_\bd$.
    
\begin{maintheorem}\label{mainthm:ambientX}
In the above setting, the monodromy $\Gamma_{X,\bd} \le \Mod(\Sigma_{r(\bd)})$ of the family $U_{X,\bd}$ of complete intersection curves of multidegree $\bd$ in $X$, is the associated $r(\bd)$-spin mapping class group:
\[
\Gamma_{X,\bd} = \Mod(\Sigma_{g(\bd)})[\phi_\bd].
\]
\end{maintheorem}

\begin{remark}
    Evidently \Cref{mainthm:ambientX} implies \Cref{mainthm:ambientPn}, but in fact the two statements are equivalent. This is a relatively straightforward consequence of the Lefschetz hyperplane theorem - see \Cref{lemma:bigsmall}. 
\end{remark}

\begin{remark}
    To date, the monodromy calculations appearing in the literature \cite{CL1,CL2,saltertoric} have taken place in the setting of smooth toric surfaces, and {\em a fortiori} within the world of rational surfaces. \Cref{mainthm:ambientX} thus shows that a wide variety of algebraic surfaces have linear systems with $r$-spin monodromy- not only rational surfaces, but $K3$'s and surfaces of general type. \Cref{conj:rspinmon} below posits that this should hold in great generality.
\end{remark}

\para{Vanishing cycles} A primary application of these results is to the problem of characterizing {\em vanishing cycles}. Recall (see \Cref{subsection:nodaldegen}) that under a nodal degeneration of a smooth algebraic curve (over $\C$), there is a topological cylinder which collapses to a cone at the nodal point; the core curve of such a cylinder is called a vanishing cycle. It is of interest, e.g. in symplectic geometry, to understand precisely which simple closed curves can arise as a vanishing cycle. This problem was posed by Donaldson \cite{donaldson}, originally in the setting of families of curves in toric surfaces. Substantial progress in this direction was made by Cr\'etois--Lang \cite{CL1,CL2}, and was resolved in full for smooth toric surfaces in the work \cite{saltertoric} of the second-named author. 

The characterization of vanishing cycles turns out to hinge on a topological reformulation of an $r$-spin structure known as a {\em winding number function}. As outlined in \Cref{section:rspin}, $r$-spin structures on a smooth algebraic curve $C$ are in correspondence with topological structures known as {\em winding number functions}, which assign an element of $\Z/r\Z$ to each isotopy class of oriented simple closed curve on $C$. A simple closed curve $c \subset C$ is said to be {\em admissible} for a winding number function $\phi$ (equivalently, for an $r$-spin structure) if $c$ is nonseparating ($C \setminus c$ is topologically connected) and if $\phi(c) = 0$. 

\begin{maincor}\label{maincor:VC}
    In the settings of \Cref{mainthm:ambientPn} or \Cref{mainthm:ambientX}, let $C$ be a smooth member of the relevant family. Then a nonseparating simple closed curve $c \subset C$ is a vanishing cycle if and only if it is admissible for the winding number function associated to the distinguished $r(\bd)$-spin structure.
\end{maincor}

The passage from a monodromy group calculation to the determination of vanishing cycles is well-known, and we omit the proof. See \cite[Lemma 11.5]{saltertoric} for one point of view on this in the $r$-spin setting.

\begin{remark}
    Again, some commentary is warranted in the low-degree cases. If $r(\bd) < 0$ then \Cref{maincor:VC} is vacuous. In the cases $0 \le r(\bd) \le 1$, this asserts that every nonseparating $c$ is a vanishing cycle.
\end{remark}

\begin{remark}[A note on terminology and notation]
    As is inevitable when studying topological aspects of algebraic curves and algebraic surfaces, there is a collision of terminology between algebraic geometry and topology. Unfortunately, we will need to consider both algebraic surfaces (which are $4$-manifolds to a topologist), algebraic curves ({\em surfaces} in the parlance of topology), as well as topological surfaces and simple closed curves on them (embedded $1$-manifolds).

    To minimize confusion, we adopt the following terminological conventions. When dealing with algebro-geometric objects (curves and surfaces), we will always prepend ``algebraic''. We will reserve the term ``surface'' to mean a topological $2$-manifold. When discussing (simple closed) curves on topological surfaces, we will always prepend ``simple closed''.

    We will also reserve certain symbols for specific classes of objects. Capital letters $C,D,E$ will always denote algebraic curves, $X$ will always denote an algebraic surface, and $S$ will always denote a (topological) surface. Simple closed curves on $S$ will be denoted by lowercase letters $a,b,c$, etc., or by lowercase Greek letters $\alpha, \beta, \gamma$, etc. 
\end{remark}

\para{Idea of proof} We believe that the proof techniques developed here should be applicable to monodromy problems on a wide variety of algebraic surfaces. Here we give an overview of the main ideas.

Following the work of \cite{saltertoric,strata2,strata3}, the theory of $r$-spin mapping class groups (and their cousins the ``framed mapping class groups'' appearing below) is well-developed, and there exist simple and flexible criteria for showing that a collection of Dehn twists generates a given $r$-spin mapping class group (see, e.g. \Cref{theorem:assemblagegenset}). Dehn twists arise as the monodromy of a nodal degeneration, and so the challenge in proving a result such as \Cref{mainthm:ambientPn} lies in exhibiting enough nodal degenerations.

In the prior work \cite{CL1,saltertoric} in the setting of smooth toric surfaces, this was accomplished using the techniques of tropical geometry, as developed in \cite{CL1}. Although beautiful in its own right, the tropical method does not extend to general algebraic surfaces.

Here, we overcome this limitation by means of an inductive argument. Suppose that $L_1$ and $L_2$ are very ample line bundles on an algebraic surface $X$. One obtains a family of smooth sections of $L_1 \otimes L_2$ by taking the union of a section of $L_1$ and a section of $L_2$ (both smooth) and then smoothing by a small perturbation. The crux of our argument is to show that if $L_1$ has ``full monodromy'' (equal to some $r$-spin mapping class group), then (under suitable technical hypotheses) so does $L_1 \otimes L_2$ - note that in general, the value of $r$ will be different! We found the asymmetry in this principle to be surprising - it is {\em not} necessary to assume that both $L_1$ and $L_2$ have full monodromy, just that one of them does. 

For clarity, it is worthwhile to explain the argument in more detail. Let $C,D$ denote smooth zero loci of sections of $L_1, L_2$, respectively, and let $E$ be a smooth perturbation of $C \cup D$. The generation criterion \Cref{theorem:assemblagegenset} asserts that it is necessary to find a collection of finitely many vanishing cycles that topologically ``fill'' $E$ (subject to technical hypotheses detailed in \Cref{subsection:genusge5}). 

The construction of $E$ from $C \cup D$ endows it with a decomposition
\[
E = \tilde C \cup \tilde D,
\]
where $\tilde C$ and $\tilde D$ are surfaces with boundary components inserted at each of the points of intersection $C \cap D$; then $E$ is glued together by identifying each pair of boundary components. The inductive hypothesis provides a suitable supply of vanishing cycles on $C$; we show in \Cref{lemma:liftconfig} that these can be lifted to vanishing cycles in $\tilde C \subset E$.

It remains to produce vanishing cycles that fill out the $\tilde D$-side of $E$. For this, we vary $C$ in a pencil $C_t$, and consider the singularities that arise when a member $C_{t_i}$ becomes tangent to $D$. A bare-hands analysis (carried out in \Cref{section:tacnode}) shows that one of the vanishing cycles for this tacnodal singularity enters and exits $\tilde D$ exactly once. What is more, we can {\em precisely control} the corresponding arc on $\tilde D$, by considering the image in $\CP^1$ under the pencil map $\pi: X \dashedrightarrow  \CP^1$ obtained from $C_t$. We show in \Cref{section:tacnode} that we can obtain sufficiently many vanishing cycles of this form by taking enough tacnodal degenerations along a set of well-chosen paths.

\para{A conjecture}
The original monodromy results of this sort, obtained in \cite{saltertoric}, held only for the rather narrow class of smooth toric surfaces, but this was for artificial reasons - the use of tropical techniques to manufacture vanishing cycles was limited to this arena. The results here reinforce the idea that possessing $r$-spin monodromy should be a fairly general phenomenon for linear systems of curves on algebraic surfaces - one exception being the case when every smooth section is hyperelliptic. On the other hand, the work \cite{ishanpi1} of the first-named author shows that some restrictions are necessary - when the fundamental group of the algebraic surface $X$ is nontrivial, the monodromy group can be an infinite-index subgroup of the mapping class group! Perhaps these are the only obstructions, at least in large degree?

\begin{conjecture}\label{conj:rspinmon}
    Let $X$ be a smooth simply-connected algebraic surface, and let $L$ be an ample line bundle on $X$. Let $k \gg 0$ be sufficiently large, and suppose that a general smooth section of $L^{\otimes k}$ is not hyperelliptic. Then the monodromy of the family of smooth sections of $L^{\otimes k}$ is an $r$-spin mapping class group, where $r$ is the largest root of $K_X \otimes L^{\otimes k}$ in $\Pic(X)$.
\end{conjecture}

Coming from the other direction, it would be quite interesting to find examples of very ample line bundles $L$ whose sections are non-hyperelliptic, but for which the monodromy group is a strict subgroup of the associated $r$-spin mapping class group.

\para{Organization} The technical core of paper is divided into two parts. The first, occupying \Cref{section:rspin,section:simplebraids,section:genspinmcg}, is concerned with group theory (of the (framed or $r$-spin) mapping class group and various subgroups), and the second, from \Cref{section:topologyCI,section:exhibitsimplebraids,section:tacnode,section:basecase,section:inductionstep}, with the topology of families of complete intersection curves. The brief \Cref{section:numerology} recalls some basic facts of algebraic geometry used throughout.

Within the first part, \Cref{section:rspin} recalls the various points of view on $r$-spin structures and framings on surfaces, and the associated subgroups of the mapping class group. \Cref{section:simplebraids} establishes a technical result, giving a generating set for the ``simple braid group''. This is a subgroup of the braid group of $E$ previously studied by Dolgachev-Libgober \cite{DolgLib}, Shimada \cite{shimada}, Walker \cite{walker}, and Qiu-Zhou \cite{QZ}. Finite presentations for the simple braid group were obtained by Shimada \cite{shimadapres} and Qiu-Zhou \cite{QZ}. These make use of specific generating sets, whereas for our purposes we require more flexibility, and so we give a separate argument.
\Cref{section:genspinmcg} treats the subject of generating sets for $r$-spin mapping class groups. Following the work of \cite{saltertoric,strata2,strata3}, there is a completely satisfactory theory so long as the genus of the surface is at least five. However, our arguments will require us to go past this threshold; in \Cref{section:genspinmcg} we establish the rather bespoke results needed to deal with this.

The second part begins in \Cref{section:topologyCI} with some general discussion of the topology of a complete intersection curve near the reducible locus (the decomposition $E = \tilde C \cup \tilde D$ discussed above), as well as some other generalities. In \Cref{section:exhibitsimplebraids}, we show that the monodromy contains the simple braid group studied in \Cref{section:simplebraids}; ultimately this is a technical step that will allow us to lift vanishing cycles from $C$ up to $\tilde C \subset E$. In \Cref{section:tacnode}, we study the vanishing cycles associated to tacnodal degenerations; this will be a crucial technical tool in the inductive step. \Cref{section:basecase} establishes the base cases, carrying out the monodromy calculations for the cases $r(\bd) \le 1$ where the $r$-spin condition is vacuous; this is essentially classical. Finally, the inductive step of the argument is carried out in \Cref{section:inductionstep}.

\para{Acknowledgements} The first author would like to thank Madhav Nori for bringing this question to his attention. The first author would also like to thank Aaron Calderon for many helpful discussions. The second author acknowledges support from the National Science Foundation, grant no. DMS-2338485. Both authors gratefully acknowledge the efforts of an anonymous referee whose close reading led to numerous improvements.

\section{Numerology of complete intersection curves}\label{section:numerology}

Here we record some basic results relating the multidegree $\bd$ to other numerical invariants of the complete intersection curve, especially genus $g(\bd)$ and the index $r(\bd)$ of the distinguished $r$-spin structure induced from $\cO(1)$. Letting $\bd = (d_1, \dots, d_{n-1})$ be a multidegree, we define the numerical quantities
\[
\Pi(\bd) = \prod_{i = 1}^{n-1}d_i \qquad \mbox{and} \qquad r(\bd) = \sum_{i = 1}^{n-1} d_i - n - 1.
\]

The following is a straightforward application of the adjunction formula.

\begin{lemma}[Multidegree-genus formula]\label{lemma:multidegreegenus}
    Let $C$ be a smooth complete intersection curve of multidegree $\bd = (d_1, \dots, d_{n-1})$. Then $C$ has genus $g(\bd)$ given by
    \[
    g(\bd) = \tfrac 12(\Pi(\bd)r(\bd) + 2),
    \]
    and 
    \[
    K_C \cong \cO_C(r(\bd)),
    \]
    so that $\cO(1)$ restricts to $C$ as a distinguished $r(\bd)^{th}$ root of $K_C$.
\end{lemma}

The regimes $g(\bd) = 0, g(\bd) = 1, g(\bd) \ge 2$ each correspond to a regime for $r(\bd)$. We will freely pass back and forth between these points of view.

\begin{lemma}
    There is the following correspondence between $g(\bd)$ and $r(\bd)$:
    \begin{itemize}
        \item $g(\bd) = 0$ if and only if $r(\bd) < 0$,
        \item $g(\bd) = 1$ if and only if $r(\bd) = 0$,
        \item $g(\bd) \ge 3$ if and only if $r(\bd) \ge 1$.
    \end{itemize}
\end{lemma}

\begin{proof}
    This is a corollary of \Cref{lemma:multidegreegenus}.
    As $\Pi(\bd) > 0$, the only way for $g(\bd) = 0$ to hold is if $r(\bd) < 0$, and similarly $g(\bd) = 1$ is equivalent to $r(\bd) = 0$. The additional assertion that $g(\bd) \ne 2$ likewise follows from the multidegree-genus formula, by seeing that no $\bd$ have $\Pi(\bd)r(\bd) = 2$.
\end{proof}

A multidegree $\bd = (d_1, \dots, d_{n-1})$ is {\em reduced} if either $n = 2$ and $d_1 = 1$, or else $d_i \ge 2$ for all $i$. It will be useful to have a complete listing of reduced multidegrees leading to small genus and/or $r(\bd)$.

\begin{lemma}\label{lemma:smallgenustable}
    The table below gives a complete listing of reduced multidegrees of complete intersection curves of genus at most four and/or of $r(\bd) \le 1$.
    \[
    \begin{array}{|ccc|} \hline
       \bd    &  g(\bd) &r(\bd)\\ \hline
        1 & 0 & <0\\ 
        2 & 0 & < 0\\
        3& 1 & 0\\
        (2,2) & 1& 0\\
        4 & 3 & 1\\
        (3,2) & 4 & 1 \\ 
        (2,2,2) & 5 & 1\\ \hline
    \end{array}
    \]
\end{lemma}

\begin{proof}
By \Cref{lemma:multidegreegenus}, enumerating $\bd$ for which $g(\bd) \le 4$ amounts to enumerating $\bd$ for which $\Pi(\bd)r(\bd) \le 6$. Since $\bd$ is reduced, either $\bd = 1$ (which is accounted for in the table) or else $r(\bd) \ge (2n-2)-n-1 = n-3$ and $\Pi(\bd) \ge 2^{n-1}$. 

    If $n \ge 4$ then $\Pi(\bd) \ge 8$ and $r(\bd) \ge n-3 > 0$. Thus every complete intersection curve of genus at most four has $n \le 3$. Suppose $n = 3$ and $\bd = (d_1, d_2)$ has (without loss of generality) $d_1 \ge 4$. Then $\Pi(\bd) \ge 8$ and $r(\bd)\ge d_2 \ge 2$, so no examples of this form exist. Of the remaining possibilities, $(2,2)$ and $(3,2)$ are added to the table, while $(3,3)$ has genus $10$. The case $n = 2$ is equally easy to analyze. Returning to the inequality $r(\bd) \ge n-3$, we conclude that the only $\bd$ for which $r(\bd) = 1$ which has not yet been accounted for is $\bd = (2,2,2)$.
\end{proof}

\section{Framings and $r$-spin structures}\label{section:rspin}
Here we recall the topological theory of framings and $r$-spin structures on surfaces, following the treatment given in \cite[Section 2]{strata3}. While our ultimate interest is in $r$-spin mapping class groups, it will be necessary to also consider the closely-related notion of a {\em framed mapping class group} and the underlying theory of framings on surfaces.

\subsection{Perspectives} There is a surprising diversity of points of view on framings and $r$-spin structures.\\

\para{Framings} Recall that a {\em framing} of an $n$-manifold $M$ is a trivialization $\phi: TM \cong M \times \R^n$ of the tangent bundle, or equivalently a section of the frame bundle $F(TM)$. Two framings are said to be isotopic\footnote{\label{footnote:absfr}In our previous work \cite{strata3}, we had occasion to also consider a more restrictive notion of ``relative isotopy'' which required the sections to stay fixed on $\partial M$, but we will not need this here. What we discuss here were called ``absolute framings'' in \cite{strata3}.} if they are isotopic through sections of $F(TM)$. In the sequel we will only ever consider framings up to isotopy; for simplicity, ``framing'' should be understood to mean ``isotopy class of framing''. 

When $M = S$ is an oriented surface, a framing can be obtained from an apparently more modest collection of data.

\begin{lemma}
    Let $S$ be an oriented surface. Either the data of a non-vanishing vector field $\xi$, or a non-vanishing $1$-form $\omega$, determines a framing on $S$. An isotopy of $\xi$ or $\omega$ produces an isotopic framing.
\end{lemma}
\begin{proof}
    A framing of $S$ consists of the data of two non-vanishing vector fields $\xi_1, \xi_2$ that are everywhere linearly independent. Let $g$ be an arbitrary Riemannian metric on $S$. Then, given $\xi$, a second vector field $\eta$ can be constructed as the positively-oriented unit perpendicular to $\xi$; it is easy to see that different choices of metric lead to isotopic $\eta$ and hence isotopic framings. Similarly, given a $1$-form $\omega$, the metric $g$ identifies $\omega$ with some non-vanishing vector field $\xi$, and the construction proceeds as before. It is clear that an isotopy of $\xi$ induces an isotopy of the resulting framing.
\end{proof}

\para{$\mathbf{r}$-spin structures} A framing is an instance of the more general notion of an {\em $r$-spin structure}. Recall that a classical spin structure can be defined as a lifting of the structure group of the tangent bundle of an oriented Riemannian $n$-manifold $M$ from $\SO(n)$ to its double cover $\Spin(n)$. For $n \ge 3$, $\pi_1(\SO(n)) \cong \Z/2\Z$, but as $\SO(2) \cong S^1$, for $n = 2$ there is a unique $\Z/r\Z$-cover $\Spin(2;r)$ of $\SO(2)$ for each $r \ge 0$. An $r$-spin structure on $S$ is then a lifting of the structure group of $TS$ to $\Spin(2;r)$. As $\Spin(2;0) \cong \R$ is contractible, it follows that a framing indeed coincides with the case $r = 0$ of this notion.

\para{Equivalent formulations} In this paper we will have occasion to consider a wide variety of perspectives on $r$-spin structures. 

\begin{definition}[Winding number function]
    Let $S$ be an oriented surface. Let $\cS$ denote the set of isotopy classes of oriented simple closed curves on $S$, including the inessential curve $\delta$ given as the boundary of an embedded disk $D \subset S$. A {\em $\Z/r\Z$-winding number function} is a function $\phi: \cS \to \Z/r\Z$ that obeys the following conditions:
    \begin{enumerate}[label=(\alph*)]
        \item (Reversibility) If $c \in \cS$ and $\bar c$ denotes $c$ with the opposite orientation, then $\phi(\bar c) = - \phi(c)$,
        \item (Twist-linearity) Given $c, d \in \cS$,
        \[
        \phi(T_c(d)) = \phi(d) + \pair{[c],[d]}\phi(c),
        \]
        where $T_c$ denotes the Dehn twist about $c$ and $\pair{\cdot, \cdot}$ denotes the algebraic intersection pairing,
        \item (Homological coherence) If $c_1, \dots, c_k$ bound a subsurface $S'$ lying to the left of each $c_i$, then
        \[
        \sum \phi(c_i) = \chi(S'),
        \]
        where $\chi(S')$ denotes the Euler characteristic of $S'$.
    \end{enumerate}
\end{definition}

\begin{remark}[Orientation conventions]\label{remark:convention}
    The reversibility axiom implies that a statement of the form ``$\phi(c) = k$'' is only unambiguous either when $k = -k \pmod r$ or if an orientation for the simple closed curve $c$ is specified. An important case we will frequently encounter is when $c \subset S$ is a boundary component. Here, we adopt the convention that $c$ will always be oriented with $S$ to the left, in keeping with the homological coherence condition. Otherwise, statements of the form $\phi(c) = k$ will tacitly mean that such an equality holds for some unspecified choice of orientation.
\end{remark}

\begin{definition}[Signature of $r$-spin structure]\label{def:signature}
    Let $(S, \phi)$ be a surface equipped with an $r$-spin structure. Suppose $S$ has boundary components $d_1, \dots, d_k$. The {\em signature} of $\phi$ is the $k$-tuple $(\phi(d_1), \dots, \phi(d_k)) \in (\Z/r\Z)^k$, where each $d_i$ is oriented (in accordance with \Cref{remark:convention}) with $S$ to the left. The $r$-spin structure $\phi$ is said to have {\em constant signature $w$} if $\phi(d_i) = w$ for all $d_i \in \partial S$.
\end{definition}

\begin{remark}
    The term ``signature'' in the context of $r$-spin structures was introduced in \cite{saltertoric}. It is unrelated to the meaning of the term in the context of orbifolds (or $4k$-manifolds, etc.). 
\end{remark}

Winding number functions admit a cohomological reformulation, by work of Humphries-Johnson \cite{HJ}. Given $c \in \cS$, there is a well-defined {\em Johnson lift} of $c$ to an isotopy class of simple closed curve $\hat c$ in $UTS$, the unit tangent bundle of $S$, by equipping $c$ with its forward-pointing unit tangent vector. Thus a cohomology class $\phi \in H^1(UTS; \Z/r\Z)$ determines a function $\phi: \cS \to \Z/r\Z$ via the assignment $c \mapsto \phi(\hat c)$. 

It turns out that all of these notions are different aspects of the same theory.

\begin{proposition}\label{prop:rspindefs}
    Let $S$ be an oriented surface, and let $r \ge 0$ be given. Then the following sets are in natural bijective correspondence:
    \begin{enumerate}
        \item The set of $r$-spin structures on $S$,
        \item The set of isotopy classes of vector fields on $S$ all of whose zeroes have order divisible by $r$,
        \item The set of isotopy classes of $1$-forms on $S$ all of whose zeroes have order divisible by $r$,
        \item The set of $\Z/r\Z$-winding number functions on $S$,
        \item The affine subset 
        \[
        \{\phi \in H^1(UTS; \Z/r\Z)\mid \phi(\hat \delta) = 1\},
        \]
        where $\delta$ as above is the boundary of an embedded disk $D$, oriented with $D$ to the left,
        \item When $S = C$ is an algebraic curve and $r > 0$, the set of line bundles $\cL \in \Pic(C)$ for which $\cL^{\otimes r} \cong K_C$.
    \end{enumerate}
\end{proposition}
\begin{proof}
    The equivalence of (1)-(4) is discussed in \cite[Section 2.1]{strata3} (cf. Lemma 2.2 and the discussion ca. Definition 2.6), and the equivalence of (4)-(6) is discussed in \cite[Sections 2,3]{salterquintic} (cf. Proposition 2.2 and Theorem 3.2). 
\end{proof}

In the case $r = 0$ above, ``order divisible by $r = 0$'' means that there are no zeroes at all.
We remark that when $S$ is a closed surface of genus $g$, $r$-spin structures exist on $S$ only for $r \mid 2g-2$, as can be seen easily from characterizations (2) or (3) of \Cref{prop:rspindefs}.

\subsection{$r$-spin mapping class groups} The mapping class group $\Mod(S)$ of $S$ acts on the set of $r$-spin structures. This action is perhaps most transparent from the point of view of winding number functions: given such a function $\phi$ and a mapping class $f$, define $f \cdot \phi$ via the formula
\begin{equation}
    \label{eqn:rspinaction}
    (f \cdot \phi)(c) = \phi(f^{-1}(c)).
\end{equation}

\begin{definition}[$r$-spin/framed mapping class group]
    Let $\phi$ be an $r$-spin structure on $S$. The associated {\em $r$-spin mapping class group}, written $\Mod(S)[\phi]$, is the stabilizer of $\phi$ under the action \eqref{eqn:rspinaction}. When $r = 0$ we will call $\Mod(S)[\phi]$ a {\em framed mapping class group}.
\end{definition}

It will be important to understand the amount of data necessary to completely specify an $r$-spin structure. Equivalently, this gives a simple criterion to check that a given $f \in \Mod(S)$ is contained in $\Mod(S)[\phi]$.

\begin{lemma}\label{lemma:specifyrss}
    Let $S$ be a surface and let $c_1, \dots, c_k$ be a set of simple closed curves such that $[c_1], \dots, [c_k]$ forms a basis for $H_1(S;\Z)$. Let $r$ be a nonnegative integer, further supposing $r \mid 2g-2$ in the case $S \cong \Sigma_g$ is a closed surface of genus $g$. Then there is a one-to-one correspondence between $r$-spin structures on $S$ and $k$-tuples in $\Z/r\Z$, given by 
    \[
    \phi \leftrightarrow (\phi(c_1), \dots, \phi(c_k)).
    \]
    In particular, $f \in \Mod(S)$ lies in $\Mod(S)[\phi]$ if and only if $\phi(f(c_i)) = \phi(c_i)$ for $i = 1, \dots, k$.
\end{lemma}
\begin{proof}
    This is best understood from point of view (5) in \Cref{prop:rspindefs}, viewing $\phi$ as a cohomology class in $H^1(UTS; \Z/r\Z)$ for which $\phi(\hat \delta) =1$.
     One computes
    \[
    H_1(UTS; \Z) \cong H_1(S;\Z) \oplus \Z/\epsilon\Z,
    \]
    with $\epsilon = 2g-2$ if $S$ is closed of genus $g$ and $\epsilon = 0$ otherwise; in either case this latter factor is generated by the class of $\hat \delta$. The former factor $H_1(S;\Z)$ embeds (non-canonically) in $H_1(UTS; \Z)$ via the Johnson lift. By the universal coefficients theorem, to specify an $r$-spin structure, it is necessary and sufficient to specify its values on any basis for $H_1(S; \Z)$, giving the result.
\end{proof}

\para{$r$-spin monodromy constraints} Let $X$ be a smooth projective algebraic surface, let $\cL$ be a line bundle on $X$, and let $U_{\cL}$ be a family of smooth curves in the linear system determined by $\cL$. According to the adjunction formula, the line bundle $\cL \otimes K_X$ restricts to the canonical bundle on any member of $U_\cL$. Thus if $\cL \otimes K_X$ admits an $r^{th}$ root, this equips the curves in $U_\cL$ with a distinguished $r$-spin structure. 

\begin{remark}
    If $\Pic_0(X)$ has $r$-torsion, then the root $\cL' \in \Pic(X)$ is not unique. In this case, {\em each} such root is monodromy-invariant, and the monodromy group $\Gamma_\cL$ is necessarily contained in the intersection of the corresponding $r$-spin mapping class groups. Note, however, that this only occurs when the irregularity $q$ of $X$ is positive; in the complete intersection setting, the $r^{th}$ root $\cL'$ is unique. See \cite{ishanpi1} for an investigation of monodromy problems in the $q > 0$ regime.
\end{remark}

\begin{lemma}
    In the above setting, the monodromy $\Gamma_\cL \le \Mod(C)$ (where $C$ is the topological surface underlying some member of $U_\cL$) is contained in the $r$-spin mapping class group associated to the $r^{th}$ root $\cL'$ of $\cL \otimes K_X$.
\end{lemma}
\begin{proof}
    See \cite[Proposition 2.7]{CL1} or \cite[Theorem 1.1]{salterquintic}; both of these results are formulated in somewhat limited settings, but only make use of the adjunction formula and so apply in broad generality.
\end{proof}

As discussed below in \Cref{subsection:nodaldegen}, the Picard-Lefschetz formula states that the monodromy associated to a nodal singularity is the Dehn twist about the vanishing cycle. This leads to the following constraint on the set of simple closed curves on $C$ that can be vanishing cycles.

\begin{lemma}\label{lemma:VCadmiss}
    In the above setting, let $c \subset C$ be a nonseparating simple closed curve, and suppose that $c$ is the vanishing cycle for some nodal degeneration in $U_\cL$. Then
    \[
    \phi(c) = 0,
    \]
    where $\phi$ is the $\Z/r\Z$-winding number function associated to the distinguished $r^{th}$ root of $K_C$.
\end{lemma}
\begin{proof}
    By the Picard-Lefschetz formula, $T_c \in \Gamma_{\cL}$, and hence preserves the winding numbers of all simple closed curves on $C$. Since $c$ is nonseparating, there is some simple closed curve $d \subset C$ for which $\pair{[c],[d]} = 1$. Since $\phi(T_c(d)) = \phi(d)$, the twist-linearity property of winding number functions implies that $\phi(c) = 0$.
\end{proof}

In light of this fact, we introduce the following terminology.

\begin{definition}[Admissible curve]
    Let $(S, \phi)$ be a surface equipped with an $r$-spin structure $\phi$. A simple closed curve $c \subset S$ is said to be {\em admissible} if it is nonseparating and if $\phi(c) = 0$.
\end{definition}

Thus \Cref{maincor:VC} shows that for complete intersection curves, admissibility gives a {\em complete} characterization of nonseparating vanishing cycles. 

\subsection{The framed change-of-coordinates principle} In the theory of mapping class groups, the {\em change-of-coordinates principle} is a body of results for working with simple closed curves on surfaces, e.g. guaranteeing the existence of a configuration of simple closed curves with specified combinatorics that extends some particular subconfiguration of pre-specified simple closed curves. See \cite[Section 1.3]{FM} for an exposition. On a technical level, change-of-coordinates amounts to a body of transitivity results for the action of the mapping class group on whatever type of configuration is being studied.

In the setting of $r$-spin or framed mapping class groups, we will be interested in similar such results concerning the orbits of $r$-spin mapping class groups on topological configurations, e.g. individual non-separating simple closed curves. Here there is at least one constraint not present in the classical setting: if $c, d \subset S$ are simple closed curves and $\phi(c) \ne \phi(d)$, then no element of $\Mod(S)[\phi]$ can take $c$ to $d$. The upshot of the framed change-of-coordinates principle is that this is {\em almost} the only new constraint (the other obstruction being the ``Arf invariant'' discussed below).

To give a precise formulation of the framed change-of-coordinates principle, we will restrict ourselves to a precise notion of ``configuration'' of simple closed curves.

\begin{definition}[Simple configuration, arboreal, of type $E$, filling]\label{def:simpleconfig}
    Let $S$ be a surface. A {\em simple configuration} on $S$ is a set $\cC = \{c_1, \dots, c_k\}$ of simple closed curves, subject to the condition that $i(c_i,c_j) \le 1$ for all pairs of curves $c_i, c_j \in \cC$ (here and throughout, $i(\cdot, \cdot)$ denotes the geometric intersection number).

    A simple configuration $\cC$ has an associated {\em intersection graph} $\Lambda_\cC$ with vertex set $\cC$ and with $c_i, c_j$ joined by an edge if and only if $i(c_i,c_j) = 1$. The {\em embedded type} of $\cC$ is the data of $\Lambda_\cC$ together with the homeomorphism type of the complement $S \setminus \cC$.
    
    A simple configuration is {\em arboreal} if $\Lambda_\cC$ is a tree, and is {\em $E$-arboreal} if moreover $\Lambda_\cC$ contains the $E_6$ Dynkin diagram as a full subgraph. A simple configuration $\cC$ is {\em filling} if $S \setminus \cC$ is a union of disks and boundary-parallel annuli.
\end{definition}

As usual, a {\em $k$-chain} is a simple configuration $c_1, \dots, c_k$ for which $i(c_i, c_{i+1}) = 1$ and $i(c_i, c_j) = 0$ otherwise.

\para{Arf invariants} In this paper, we will not encounter Arf invariants in any serious way, and so in the interest of concision we will keep this discussion rather terse. See \cite[Section 2.2]{strata3} for a fuller treatment. Suffice it to say that when $r$ is even and $\phi(c)$ is odd for every component $c$ of $\partial S$, there is a $\Z/2\Z$-valued invariant of $r$-spin structures known as the {\em Arf invariant} that classifies orbits of $r$-spin structures under the action of the mapping class group. As the framed change-of-coordinates principle (appearing below as \Cref{prop:ccp}) asserts, the existence of certain ``especially large'' configurations is obstructed by the Arf invariant, but otherwise does not play a role in our arguments.

\begin{proposition}[Framed change-of-coordinates principle]
    \label{prop:ccp}
    Let $(S, \phi)$ be a surface of genus $g \ge 2$ equipped with an $r$-spin structure $\phi$. A simple configuration $c_1, \dots, c_k$ of curves of prescribed embedded type and winding numbers $\phi(c_i) = w_i$ exists if and only if
    \begin{enumerate}[label=(\alph*)]
        \item A simple configuration $\{c_1', \dots, c_k'\}$ of the prescribed embedded type exists in the ``unframed'' setting where the values $\phi(c_i')$ are allowed to be arbitrary,
        \item There exists some $r$-spin structure $\psi$ with $\psi(c_i') = w_i$ for all $i$,
        \item If $\Arf(\psi)$ exists and is constrained by (b), then $\Arf(\phi) = \Arf(\psi)$.
    \end{enumerate}
\end{proposition}
\begin{proof}
    See \cite[Proposition 2.15]{strata3}.
\end{proof}

\begin{remark}[When is $\Arf$ constrained?]
    For the readers' convenience, we comment here on some circumstances under which $\Arf(\phi)$ is or isn't constrained by the simple configuration, as mentioned in \Cref{prop:ccp}.c. First, as remarked above, $\Arf(\phi)$ is only defined when $r$ is even and $\phi(c)$ is odd for every boundary component $c$ of $\partial S$. In such a setting, a simple configuration $\cC$ {\em does not} constrain $\Arf(\phi)$ if there is a pair of simple closed curves $a,b \subset S$ such that $i(a,b) = 1$ and $a,b$ are disjoint from every $c_i \in \cC$. Thus, only ``large'' configurations $\cC$ for which it is not possible to find such a pair, possibly induce a constraint on $\Arf$.
\end{remark}

The result below is an existence result for ``large'' configurations.

\begin{lemma}\label{lemma:Earbexists}
    Let $S$ be a closed surface of genus $g(S) \ge 3$ and let $\phi$ be an $r$-spin structure on $S$. Then there is an $E$-arboreal simple configuration $\cC = \{c_1, \dots, c_{2g}\}$ of admissible curves on $S$, such that the associated homology classes $\{[c_i]\} \in H_1(S;\Z)$ form a basis. Moreover, $\cC$ can be chosen to take one of the two forms shown in \Cref{fig:liftingvcs} (p. \pageref{fig:liftingvcs}), depending on the value of the Arf invariant (if present).
\end{lemma}
\begin{proof}
    The reader familiar with computing the Arf invariant (see \cite[Section 2.2]{strata3}) can check that the two simple configurations shown in \Cref{fig:liftingvcs} (p. \pageref{fig:liftingvcs}) have distinct Arf invariants. The result then follows from the framed change-of-coordinates principle (\Cref{prop:ccp}). 
\end{proof}

\subsection{Functoriality}
We will have occasion to consider various ``functorial'' aspects of the theory of $r$-spin structures. While one could imagine formulating a category of surfaces equipped with $r$-spin structures with morphisms given by inclusion, and a functor taking such a surface to its associated $r$-spin mapping class group, we do not attempt to pursue this rigorously; ``functoriality'' is meant here as a term of art to evoke aspects of the theory that involve passing from one surface to another, or from one $r$ to another. To prove these results, we will first need to develop some preliminary notions (\Cref{lemma:dualexists,lemma:pushaffectsphi}).

\begin{definition}[Dual configuration]
    Let $S$ be a surface with nonempty boundary, and let $\cC = \{c_1, \dots, c_k\}$ be a simple collection of simple closed curves. A {\em dual configuration based at $*$} for $\cC$ is a collection $\beta_1, \dots, \beta_k$ of simple closed curves on $S$, all {\em based} at some common point $*\in \partial S$, such that $i(c_i, \beta_j) = \delta_{i,j}$.
\end{definition}

Certainly not every set of simple closed curves admits a dual configuration. The next lemma gives a condition under which this is ensured.

\begin{lemma}\label{lemma:dualexists}
    Let $S$ be a connected surface with nonempty boundary, and let $\bar S$ be the closed surface obtained by capping all boundary components of $S$ with disks. Let $\cC = \{c_1, \dots, c_k\}$ be an arboreal simple configuration on $S$ for which the set of induced homology classes $\{[\bar{c_i}]\} \subset H_1(\bar{S};\Z)$ is linearly independent. Then, for any $* \in \partial S$, a dual collection for $\cC$ based at $*$ exists.
\end{lemma}
\begin{proof}
    Choose representative curves for $c_1, \dots, c_k$ in minimal position, and cut $S$ along each such curve to obtain a surface $S^\circ$. We claim that $S^\circ$ is connected. If not, $\cC$ separates $S$ into nonempty proper subsurfaces $S_1, S_2$, and each $S_i$ has one or more boundary components consisting of a sequence of oriented subsegments of curves in $\cC$, possibly along with additional boundary components in $\partial S$. 
    
    View each boundary component of the first kind (segments of curves in $\cC$) as a singular $1$-cycle on $S$, and cancel any segments that appear twice, once with each orientation. By definition this is a homologous cycle, but the number of path components may have increased. Each such component is of one of two types: either it consists of a single $c_i \in \cC$ in its entirety, or else consists of a union of two or more proper subsegments of curves in $\cC$, cyclically ordered and with successive segments taken from distinct elements of $\cC$.
    
    Suppose that any such boundary component is of this second type. Since each segment is proper, the beginning and end points of each segment are distinct, and since $\cC$ is a simple configuration, the intersection points at beginning and end are with distinct elements of $\cC$. Then the sequence of curves seen when running through such a segment induces a nontrivial cycle in the intersection graph $\Lambda_\cC$ with no backtracking, contradicting the assumption that $\cC$ is arboreal.

    Thus, every boundary component of $S_1$ must consist of a single curve from $\cC$ in its entirety, or else be a boundary component of $S$ itself. Passing to $\bar S$, this shows that some subset of curves in $\cC$ bound a subsurface in $\bar S$, contrary to the hypothesis that they are linearly independent in $H_1(\bar S; \Z)$. 

    We thus conclude that $S^\circ$ is connected. For each $c_i \in \cC$, choose a pair of points $p_i, p_i'$ on $\partial S^\circ$ identified with the same point on $c_i$. Since $S^\circ$ is connected, it is then straightforward to construct simple arcs connecting each of $p_i, p_i'$ to $*$, disjoint except at the common endpoint $*$. Such a pair descends to a simple closed curve $\beta_i$ with $i(\beta_i, c_j) = \delta_{i,j}$ as required. Note that there are no constraints on intersections between distinct $\beta_i, \beta_j$, so that such $\beta_i$ can be constructed independently, giving the required construction.
\end{proof}

Recall that when a surface $S$ has nonempty boundary, a loop $\beta \in S$ based at $* \in \partial S$ determines a {\em disk-push map} $P_\beta \in \Mod(S)$. This is given as the composite of two Dehn twists
\[
    P_\beta = T_{\beta_R} T_{\beta_L}^{-1},
    \]
    where $\beta_R, \beta_L$ are the curves in the interior of $S$ lying to the right (resp. left) of $\beta$ in its direction of travel. Our next preparatory result explains the effect of a disk-push on winding numbers.

    \begin{lemma}
    \label{lemma:pushaffectsphi}
    Let $(S, \phi)$ be a surface equipped with an $r$-spin structure. Let $d \subseteq \partial S$ denote a fixed boundary circle. Let $c \subset S$ be a simple closed curve, and let $\beta \subset S$ be an oriented curve based at some point $* \in d \subset \partial S$. Then the disk-push map $P_\beta$ affects the winding number of $c$ as follows:
    \[
    \phi(P_\beta(c)) = \phi(c) - (\phi(d) + 1) \pair{[c],[\beta]}.
    \]
\end{lemma}

\begin{proof}
    Applying twist-linearity to $P_\beta = T_{\beta_R} T_{\beta_L}^{-1}$,
    \[
    \phi(P_\beta(c)) = \phi(c) + \phi(\beta_R)\pair{[c], [\beta_R]} - \phi(\beta_L) \pair{[c],[\beta_L]},
    \]
    with $\beta_R, \beta_L$ oriented so as to run in the same direction as $\beta$. Note that
    \[
    \pair{[c], [\beta_R]} = \pair{[c],[\beta_L]} = \pair{[c],[\beta]},
    \]
    so that the above simplifies to
    \[
    \phi(P_\beta(c)) = \phi(c) + (\phi(\beta_R)-\phi(\beta_L))\pair{[c], [\beta]}.
    \]
    The curves $\beta_L, \beta_R$, and $d$ cobound a pair of pants. Under our specified orientations, this lies to the left of $\beta_R$ and $d$, but to the right of $\beta_L$. Consequently, by homological coherence,
    \[
    \phi(\beta_R) - \phi(\beta_L) = -(\phi(d) + 1),
    \]
    from which the claim follows.   
\end{proof}

We can now present our main results on functoriality (\Cref{lemma:framingextend,lemma:framedontospin,,lemma:rspincontainment}).

\begin{lemma}\label{lemma:framingextend}
    Let $(S, \phi)$ be a framed surface, and let $S^+$ be obtained from $S$ by attaching a $1$-handle along $\partial S$.
    Let $c \subset S^+$ be a simple closed curve for which $c \cap S$ is a single arc. 
    For any $w \in \Z$, there is a unique extension of $\phi$ to a framing $\phi^+$ of $S^+$ for which $\phi^+(c) = w$.
\end{lemma}
\begin{proof}
    According to \Cref{lemma:specifyrss}, a framing $\phi^+$ of $S^+$ is specified uniquely by the values of $\phi^+$ on any basis for $H_1(S^+;\Z)$. 
    Let $c_1, \dots, c_k \subset S$ be a set of curves whose homology classes form a basis for $H_1(S;\Z)$; then $[c_1], \dots, [c_k], [c]$ forms a basis for $S^+$. 
    Specify $\phi^+$ by setting $\phi^+(c_i) = \phi(c_i)$ and $\phi^+(c) = w$. 
    It remains to see that $\phi^+$ is indeed an extension of $\phi$, but this is straightforward from the cohomological perspective: $\phi^+ \in H^1(UTS^+; \Z)$ is evidently sent to $\phi \in H^1(UTS; \Z)$ under the pullback of the inclusion map $UTS \into UTS^+$.
\end{proof}

\begin{lemma}\label{lemma:framedontospin}
    Let $S \into \bar S$ be an inclusion of surfaces, given by capping off boundary components $d_1, \dots, d_N$ of $S$ with closed disks (note that $\bar S$ may still have boundary components and/or punctures). Let $\phi$ be a framing on $S$, and define
    \[
    \rho = \gcd(\phi(d_1)+1, \dots, \phi(d_N) +1),
    \]
    where $d_1, \dots, d_N$ are the boundary components of $S$ capped off in $\bar S$, oriented with $S$ to the left. Then the inclusion $S \into \bar S$ induces a surjection
    \[
    \Mod(S)[\phi] \onto \Mod(\bar S)[\bar \phi],
    \]
    where $\bar \phi$ is the $\rho$-spin structure on $\bar S$ obtained by reducing $\phi$ mod $\rho$.
\end{lemma}
\begin{proof}
    The framing $\phi$ corresponds to an isotopy class of non-vanishing vector field on $S$. By basic differential topology, this can be extended to a vector field on $\bar S$ with a zero inside $d_i$ of order $-1-\phi(d_i)$. Let $\bar \phi$ be the $\rho$-spin structure associated to this vector field; it is then clear that the image of $\Mod(S)[\phi]$ is contained in $\Mod(\bar S)[\bar \phi]$.

    It remains to show that this is a surjection. To that end, let $\cC = \{c_1, \dots, c_{2g}\}$ be an arboreal simple configuration of simple closed curves on $S$ satisfying $i(c_i, c_j) \le 1$, whose homology classes generate $H_1(\bar S;\Z)$. Such a collection extends to a basis for $H_1(S;\Z)$ by appending the classes of $d_1, \dots, d_{N-1}$. By \Cref{lemma:specifyrss}, $f \in \Mod(S)$  preserves $\phi$ if and only if $\phi(f(c_i)) = \phi(c_i)$ for all $i$, since each $d_i$ is fixed by definition. Given $\bar f \in \Mod(\bar S)[\bar \phi]$, let $f' \in \Mod(S)$ be an arbitrary lift. Since $f'$ lifts $\bar f \in \Mod(\bar S)[\bar \phi]$, each change in winding number $\phi(f'(c_i)) - \phi(c_i)$ is divisible by $\rho$.
    
    By \Cref{lemma:dualexists}, there exist dual systems $\{\beta_1^j, \dots, \beta_k^j\}$ for $\cC$ based at each boundary component $d_j$ of $S$. By \Cref{lemma:pushaffectsphi}, the application of the push map $P_{\beta_i^j}$ alters $c_i$ by $\pm(\phi(d_j)+1)$, while leaving the winding numbers of all other $c_j$ unchanged. Thus by some appropriate combination of such pushes, the winding numbers of each $f'(c_i)$ can be adjusted so that they equal $\phi(c_i)$. The composite $f$ of $f'$ with these pushes still maps onto $\bar f$, but now preserves each $\phi(c_i)$, so by \Cref{lemma:specifyrss}, $f \in \Mod(S)[\phi]$.  
\end{proof}

\begin{lemma}\label{lemma:rspincontainment}
    Let $S$ be a closed surface of genus $g(S) \ge 2$, and let $\psi$ (resp. $\phi$) be $r_\phi$- (resp. $r_\psi$-) spin structures for integers $r_\phi, r_\psi$. Suppose there is a containment
    \[
    \Mod(S)[\phi] \le \Mod(S)[\psi].
    \]
    Then $r_\psi$ divides $r_\phi$, and $\psi$ is given as the mod-$r_\psi$ reduction of $\phi$.
\end{lemma}

\begin{proof}
    As discussed in \Cref{lemma:VCadmiss}, a nonseparating simple closed curve $c \subset S$ is $\phi$-admissible (i.e. $\phi(c) = 0 \pmod{r_\phi}$) if and only if $T_c \in \Mod(S)[\phi]$; the corresponding statement holds for $\psi$ as well. Thus every $\phi$-admissible curve is also $\psi$-admissible.

By \Cref{lemma:Earbexists}, there is a collection $c_1, \dots, c_{2g}$ of simple closed curves on $S$ such that $[c_1], \dots, [c_{2g}]$ forms a basis for $H_1(S;\Z)$, and such that each $c_i$ is $\phi$-admissible, and hence also $\psi$-admissible. By \Cref{lemma:specifyrss}, this set of admissible curves uniquely specifies both $\phi$ and $\psi$. Viewing $r$-spin structures as classes in $H^1(UTS; \Z/r\Z)$, it follows that $\phi$ and $\psi$ admit a common refinement to an $r = 2g-2$-spin structure $\xi$ again uniquely specified by the condition that each $c_i$ be $\xi$-admissible. Thus $\phi$ and $\psi$ are the mod-$r_\phi$ (resp. mod-$r_\psi$) reductions of $\xi$.

It remains to show that $r_\psi$ divides $r_\phi$. Let $c \subset S$ be a simple closed curve with $\xi(c) = r_\phi$. Thus $c$ is $\phi$-admissible, and by the above is also $\psi$-admissible. Since $\psi$ is the reduction of $\xi$ mod $r_\psi$, it follows that $r_\phi \equiv 0 \pmod{r_\psi}$, establishing the required divisibility.
\end{proof}

\section{The simple braid group}\label{section:simplebraids}

The purpose of this section is to prove \Cref{prop:generatesbr}, which gives a generating set for a certain subgroup of the surface braid group. This will be used in the following \Cref{section:genspinmcg} as a technical tool for generating framed mapping class groups, and the relationship between the simple braid group and monodromy of complete intersection curves is explored in \Cref{section:exhibitsimplebraids} as an important ingredient in the induction argument.

In \Cref{subsection:braidsetup}, we establish these notions and explain the connection with framed mapping class groups, and in \Cref{subsection:braidproof}, we proceed with the proof of \Cref{prop:generatesbr}. In \Cref{subsection:braidsdisks}, we briefly discuss a variant of this theory suitable for surfaces with boundary.

The simple braid group has been encountered and studied by various authors - here we list the references in which it has appeared, to the best of our knowledge. It is mentioned by Dolgachev-Libgober \cite[p. 9]{DolgLib} as ``$\ker \mu_*$'' and is shown by Shimada \cite{shimada} to give the fundamental group of the complement of the dual to a smooth projective curve embedded by a suitably general line bundle of sufficiently high degree. Walker encountered a more general version of the simple braid group in her study of strata of quadratic differentials and gave a set of generators \cite{walker}. Most recently, the simple braid group was shown to be finitely presented by Qiu-Zhou \cite{QZ}. As discussed briefly in the introduction, this presentation is with respect to one fixed model for the surface, while we will require a more ``coordinate-free'' treatment, and so we give a separate proof of the finite generation of the simple braid group suitable for our purposes.

\subsection{Surface braids, simple braids, and the framed mapping class group}\label{subsection:braidsetup}

Let $S$ be a closed surface, and let $\bp = \{p_1, \dots, p_N\}$ be a configuration of $N \ge 2$ distinct points. Assuming $g(S) \ge 2$ (which will be satisfied in our context), the Birman exact sequence for the inclusion $S \setminus \bp \into S$ then takes the following form:
\[
1 \to \pi_1(\UConf_N(S)) \to \Mod(S\setminus \bp) \to \Mod(S) \to 1,
\]
where $\UConf_N(S)$ is the space of unordered $N$-tuples of distinct points in $S$. The subgroup $\pi_1(\UConf_N(S)) \le \Mod(S \setminus \bp)$ is known as the {\em surface braid group} and is denoted $Br_N(S)$.

An especially simple type of element of $Br_N(S)$ is a {\em half-twist}. Let $\alpha \subset S$ be an arc with endpoints at distinct points of $\bp$, and with interior disjoint from $\bp$. The {\em half-twist along $\alpha$}, written $P_\alpha$, is the element of $Br_N(S)$ obtained by enlarging $\alpha$ to a regular neighborhood $A$, and exchanging the endpoints $p_i, p_j$ by pushing them halfway around counterclockwise along the boundary of $A$.

Thinking of $\beta \in Br_N(S)$ as a collection of $N$ maps $[0,1] \to S$, and hence a singular $1$-chain, one sees that in fact the sum of these is a $1$-cycle, and that moreover there is a well-defined {\em cycle map}
\[
\eta: Br_N(S) \to H_1(S; \Z).
\]
More generally, suppose that the points of $\bp$ are assigned {\em weights} given by a function $\lambda: \bp \to \Z$. Let $Br_{N,\lambda}(S)$ denote the subgroup of $Br_N(S)$ consisting of braids whose associated permutations preserve the weight function. Then there is a {\em weighted cycle map}
\[
\eta_\lambda: Br_{N,\lambda}(S) \to H_1(S;\Z),
\]
where the strand $[0,1] \to S$ based at $p_i \in \bp$ is weighted by $\lambda(p_i)$ as a singular chain. Observe that $\eta_\lambda(P_\alpha) = 0$ for any half-twist along an arc $\alpha$ connecting points of equal weight.

\para{Framings and surface braids} In the presence of a framing, the weighted cycle map can be used to describe the effect of a surface braid on winding numbers. Let $S$ be a closed surface, and let $\bp = (p_1,\dots, p_N)$ be a configuration of $N \ge 1$ distinct points. Let $\phi$ be a framing of the punctured surface $S \setminus \bp$, and let $\lambda: \bp \to \Z$ be the weight function sending $p_i$ to $\phi(d_i)+1$, for $d_i \subset S$ a simple closed curve encircling $p_i$, oriented with $p_i$ to the right. 

\begin{lemma}\label{lemma:pushformula}
    In this setting, let $c \subset S \setminus \bp$ be a simple closed curve, and let $\beta \in Br_{N,\lambda}(S)$ be arbitrary. Then
    \[
    \phi(\beta(c)) = \phi(c) + \pair{[c], \eta_\lambda(\beta)}.
    \]
\end{lemma}
\begin{proof}
    This is a consequence of the homological coherence property of winding number functions. The curves $c$ and $\beta(c)$ are isotopic after the inclusion $S \setminus \bp \into S$. Consider some such isotopy from $c$ to $\beta(c)$. Each time $c$ crosses some point $p_i \in \bp$, the curves before and after bound a pair of pants along with the corresponding $d_i$. Homological coherence implies that the winding number changes by $\pm (\phi(d_i)+1)$, with the sign determined by the sign of the corresponding intersection between $c$ and the cycle representing $\eta_\lambda(\beta)$. Summing over all such points of intersection, the total change in winding number is seen to be $\pair{[c],\eta_\lambda(\beta)}$ as claimed.
\end{proof}

\begin{corollary}\label{cor:42}
      In this setting, $\Mod(S\setminus \bp)[\phi]\cap Br_N(S) = \ker(\eta_\lambda)$.
\end{corollary}

An important special case occurs when $\phi$ has constant signature. 
\begin{definition}[Simple braid group]
    Let $S$ be a closed surface, $\bp \subset S$ be a set of $N$ distinct points. The {\em simple braid group} $SBr_N(S)$ is defined as the kernel of the cycle class map $\eta$. Elements of $SBr_N(S)$ are called {\em simple braids}. 
\end{definition}

Thus $\Mod(S \setminus \bp)[\phi] \cap Br_N(S) = SBr_N(S)$ whenever $\phi$ has constant signature.

\subsection{Generating the simple braid group}\label{subsection:braidproof}

In the following, $(S, \bp)$ will denote a topological surface equipped with a configuration $\bp = \{p_1, \dots, p_N\}$ of $N \ge 2$ distinct points. It will be slightly more convenient to shift our perspective and consider $\bp$ as a set of distinguished points on $S$, instead of working with the punctured surface $S \setminus \bp$. By an {\em arc} $\alpha$ on $S$, we will mean an arc whose endpoints lie at distinct points of $\bp$ and whose interior is disjoint from $\bp$. We are free to adjust $\alpha$ by an isotopy through arcs of this same form; in particular, isotopies are always taken rel $\bp$.

\begin{proposition}\label{prop:generatesbr}
    Let $(S, \bp)$ be a surface equipped with a configuration $\bp = \{p_1, \dots, p_N\}$ of $N \ge 3$ distinct points. Let $\alpha_1, \dots, \alpha_k$ be a sequence of arcs on $S$ that satisfy the following properties:
    \begin{enumerate}
        \item The interiors of $\alpha_1, \dots, \alpha_{N-1}$ are pairwise disjoint, 
        \item A neighborhood of $\alpha_1 \cup \dots \cup \alpha_{N-1}$ is a disk containing every point of $\bp$,
        \item Defining $S_i \subset S$ as a neighborhood of $\alpha_1 \cup \dots \cup  \alpha_i$, for $i \ge N$, $\alpha_i$ exits and re-enters $S_{i-1}$ exactly once,
        \item $S \setminus S_k$ is a union of disks and annuli, and for every annular component, one boundary component is a component of $\partial S$.
    \end{enumerate}
    Then the simple braid group $SBr_N(S)$ is generated by the half-twists $P_{\alpha_1}, \dots, P_{\alpha_k}$.
\end{proposition}

This relies on the following lemma, the proof of which will occupy the bulk of the remainder of the section.

\begin{lemma}\label{lemma:alltwists}
    Let $(S, \bp)$ be a pointed surface as above, and let $\alpha_1, \dots, \alpha_k$ be a sequence of arcs satisfying the hypotheses of \Cref{prop:generatesbr}. Let $\alpha\subset S$ be an arc connecting distinct points of $\bp$ and otherwise disjoint from $\bp$. Then the half-twist $P_\alpha$ is contained in the subgroup of $Br_N(S)$ generated by $P_{\alpha_1}, \dots, P_{\alpha_k}$.
\end{lemma}

\begin{proof}[Proof of \Cref{prop:generatesbr} assuming \Cref{lemma:alltwists}]
Let $\Gamma$ be the subgroup of $Br_N(S)$ generated by the half-twists $P_{\alpha_1}, \dots, P_{\alpha_k}$. By \Cref{lemma:alltwists}, $\Gamma$ can be redefined as the subgroup of $Br_N(S)$ generated by {\em all} half-twists. Thus $\Gamma$ is a normal subgroup of $Br_N(S)$, since the set of half-twists is closed under conjugation. Note that by construction, $\Gamma$ is in fact a subgroup of the simple braid group $SBr_N(S)$. 

Let $S' \subset S$ be a subsurface for which $S \setminus S'$ is a single disk that contains no point of $\bp$. Since every braid in $S$ can be pushed into $S'$, the homomorphism $Br_N(S') \to Br_N(S)$ induced by the inclusion $S' \into S$ is a surjection. According to \cite[Theorem 2.1]{bellingerigodelle}, for a surface $S'$ with one boundary component, $Br_N(S')$ is generated by two types of elements $\sigma_1, \dots, \sigma_{N-1}$ and $\delta_1, \dots, \delta_{2g}$, with $\sigma_i$ a half-twist; it follows that $Br_N(S)$ is likewise generated by such elements. Moreover, $Br_N(S)$ admits a presentation obtained from the presentation of $Br_N(S')$ given in \cite[Theorem 2.1]{bellingerigodelle} by adding unspecified additional relations.

Since $Br_N(S)$ acts transitively by conjugation on the set of half-twists, it follows that there is an isomorphism
\[
Br_N(S)/\Gamma \cong Br_N(S)/\pair{\pair{\sigma_i}},
\]
where $\pair{\pair{\sigma_i}}$ denotes the normal closure of the set of elements $\sigma_i$.

We claim that $Br_N(S)/\Gamma$ is an abelian group generated by at most $2g$ elements. In light of the surjection $Br_N(S') \to Br_N(S)$ discussed above and the fact that a half-twist on $S'$ is sent to a half-twist on $S$ under the inclusion $S' \into S$, it suffices to show this same claim for $Br_N(S')/\pair{\pair{\sigma_i}}$. This follows from an inspection of the six families of relations (BR1)-(SCR1) appearing in \cite[Theorem 2.1]{bellingerigodelle}. Upon modding out by the $\sigma_i$, relations (BR1), (BR2), (CR1), and (CR2) become trivial, while (CR3) and (SCR1) become equivalent to the relations that the remaining $2g$ generators $\delta_1, \dots, \delta_{2g}$ commute. 

By construction, $Br_N(S)/SBr_N(S) \cong H_1(S;\Z)$ is a free abelian group generated by $2g$ elements. Since $\Gamma \le SBr_N(S)$, there is a surjection $Br_N(S)/\Gamma \onto Br_N(S)/SBr_N(S)$. The source is an abelian group generated by at most $2g$ elements, and the target is a free abelian group on $2g$ generators. It follows that this surjection must be an isomorphism, showing the desired equality $\Gamma = SBr_N(S)$.
\end{proof}

\begin{proof}[Proof of \Cref{lemma:alltwists}]
    For $1 \le i \le k$, let $\Gamma_i \le Br_N(S)$ denote the subgroup generated by the half-twists about $\alpha_1, \dots, \alpha_i$. We proceed by induction, showing that for $i = N-1, \dots, k$, the group $\Gamma_i$ contains all half-twists about arcs $\beta$ supported on $S_i$. By property (4), any arc on $S$ is isotopic to one supported on $S_k$, from which the result will follow.

    We first consider the base case $i = N-1$. This requires its own inductive step. Let the arcs $\alpha_1, \dots, \alpha_{N-1}$ be ordered in such a way that a regular neighborhood $D_j$ of $\alpha_1 \cup \dots \cup \alpha_j$ is a disk containing $j+1$ points of $\bp$. We claim that for $1 \le j \le N-1$, the half-twists about $\alpha_1, \dots, \alpha_j$ generate the $j+1$-strand braid group $B_{j+1}$ on $D_j$. This is immediate for $j = 1$; assuming the result for $j$, one can append to $\alpha_{j+1}$ a sequence $\alpha_1', \dots, \alpha_j'$ of arcs supported on $D_j$, such that $\alpha_1', \dots, \alpha_j', \alpha_{j+1}$ is the standard configuration of arcs whose half-twists generate $B_{j+2}$. By the inductive hypothesis, each of the twists about $\alpha_1', \dots, \alpha_j'$ are contained in $\Gamma_j$, from which the claim follows.

    We now assume that for some $i \ge N-1$, every half-twist supported on $S_i$ lies in $\Gamma_i$. Note that $S_{i+1}$ is obtained from $S_i$ by attaching a $1$-handle with core given by the segment of $\alpha_{i+1}$ on $S_{i+1} \setminus S_i$. Consider an arc $\beta$ supported on $S_{i+1}$. To exhibit $P_\beta$ as an element of $\Gamma_{i+1}$, we define $k(\beta)$ as the number of times $\beta$ crosses through the $1$-handle along $\alpha_{i+1}$, and we proceed by induction on $k(\beta)$.

    The case $k(\beta) = 0$ is immediate, as in this case, $\beta \subset S_i$. The case $k(\beta) = 1$ will require special consideration. Define $i(\beta,\alpha_{i+1})$ as the minimal number of intersection points on the {\em interiors} of arcs in the isotopy class of $\beta, \alpha_{i+1}$. We will obtain the case $k(\beta) = 1$ by induction on $i(\beta,\alpha_{i+1})$.

\begin{figure}
\centering
		\labellist
        \small
		\endlabellist
\includegraphics[width=\textwidth]{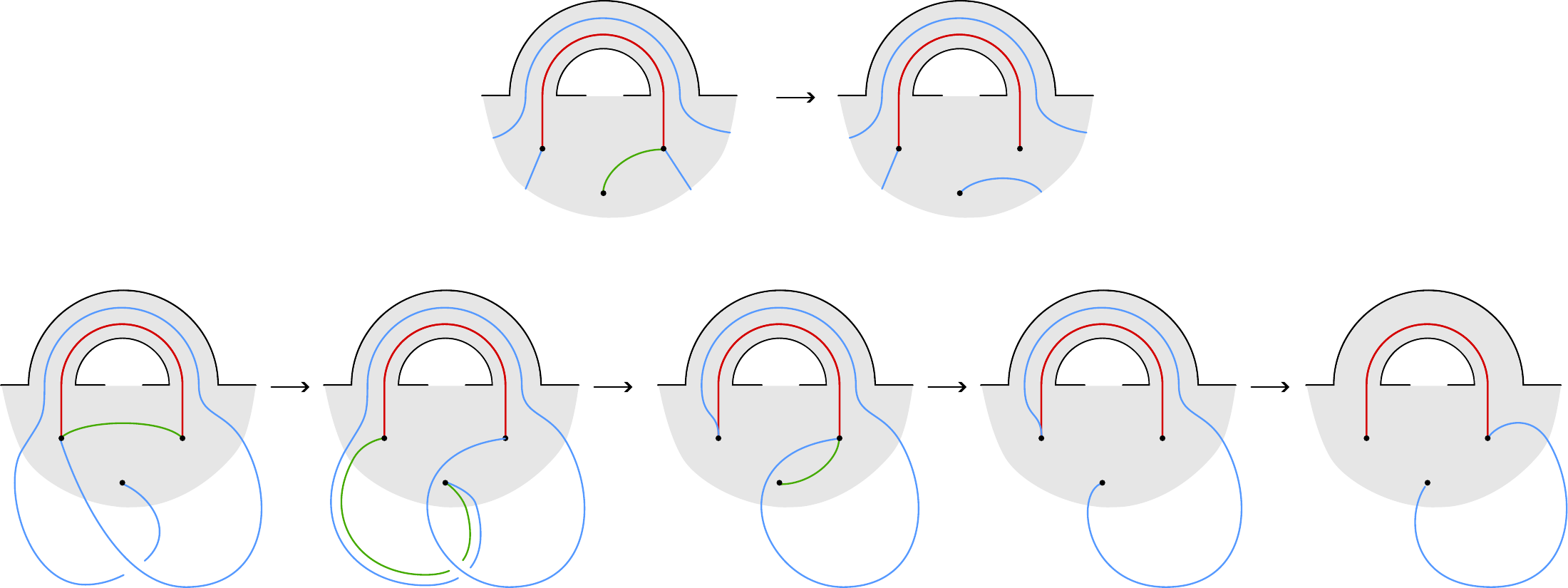}
\caption{Exhibiting $\beta$ with $i(\beta,\alpha_{i+1}) = 0$ in $\Gamma_{i+1}$. The arc $\alpha_{i+1}$ is shown in its entirety in red; $\beta$ is depicted schematically in blue. Top row: if necessary, apply a half-twist supported in $S_i$ so that $\alpha_{i+1}$ and $\beta$ share exactly one endpoint. Bottom row: it may be necessary to switch the common endpoint, so that both arcs cross the handle in the same direction when starting at their unique common endpoint. Then a further sequence of half-twists in $\Gamma_i$, followed by a half-twist about $\alpha_{i+1}$, take $\beta$ to an arc supported on $S_i$.}
\label{fig:0int}
\end{figure}

    The base case $i(\beta,\alpha_{i+1})=0$ is treated in \Cref{fig:0int}. There we see that by the hypothesis $N \ge 3$, $\beta$ can be adjusted by the application of half-twists on $S_i$ so that $\beta$ and $\alpha_{i+1}$ share exactly one endpoint, and the segments of $\beta$ and $\alpha_{i+1}$ running from this endpoint through the handle are parallel. Applying $P_{\alpha_{i+1}}$ to $\beta$ then produces an arc $P_{\alpha_{i+1}}(\beta)$ supported on $S_i$, which by hypothesis is contained in $\Gamma_i$. We conclude that $P_\beta = P_{\alpha_{i+1}}^{-1} P_{P_{\alpha_{i+1}}(\beta)} P_{\alpha_{i+1}}$ is contained in $\Gamma_{i+1}$.

\begin{figure}
\centering
		\labellist
        \small
		\endlabellist
\includegraphics[scale=0.7]{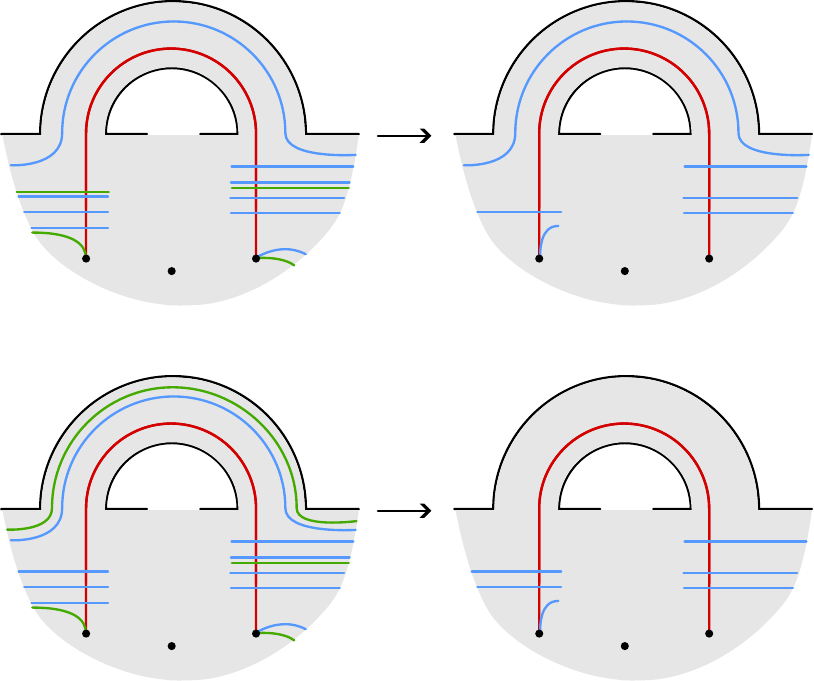}
\caption{Inductively exhibiting $\beta$ with $k(\beta) = 1$ in $\Gamma_{i+1}$. As before, $\alpha_{i+1}$ is shown in red and $\beta$ is shown schematically in blue. The arc $\gamma$ is shown schematically in green. The top and bottom rows depict the two possibilities: either $\gamma$ goes through the new handle, or it does not.}
\label{fig:kint}
\end{figure}

    We proceed to the inductive step: we assume that $\Gamma_{i+1}$ contains all half-twists about arcs $\gamma \subset S_{i+1}$ for which $i(\gamma, \alpha_{i+1}) \le p$, and consider $\beta \subset S_{i+1}$ with $i(\beta,\alpha_{i+1}) = p+1$. \Cref{fig:kint} shows how to proceed: we construct an arc $\gamma$ by following $\beta$ from its initial point to the point of crossing with $\alpha_{i+1}$ closest to an endpoint of $\alpha_{i+1}$ not shared by $\beta$. If $\gamma$ does not leave $S_i$, then $P_\gamma \in \Gamma_i$, and 
    \[
    i(\alpha_{i+1},P_\gamma(\beta)) < i(\alpha_{i+1},\beta),
    \]
    showing inductively that $P_{P_\gamma(\beta)}$, and hence $P_\beta$ itself, is contained in $\Gamma_{i+1}$. If $\gamma$ does leave $S_i$, then note that 
    \[
    i(\alpha_{i+1},\gamma) < i(\alpha_{i+1}, \beta)
    \]
    by construction, so that in this case as well, $P_\gamma \in \Gamma_{i+1}$. Then $P_\gamma(\beta) \subset S_i$, so that we conclude as above that $P_\beta \in \Gamma_{i+1}$.

\begin{figure}
\centering
		\labellist
        \small
		\endlabellist
\includegraphics[width=\textwidth]{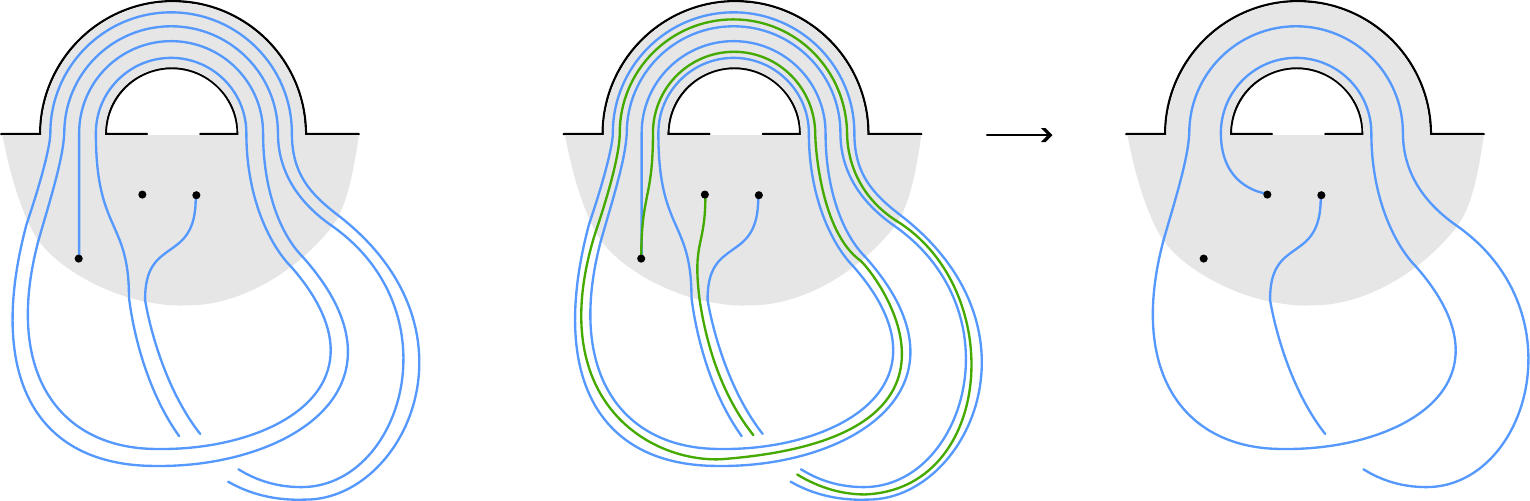}
\caption{The inductive step: decreasing $k(\beta)$. As before, $\beta$ is shown in blue and $\gamma$ is shown in green.}
\label{fig:kcrossings}
\end{figure}

    This completes the base case $k(\beta) = 1$. In general, assume $\Gamma_{i+1}$ contains all half-twists about arcs $\gamma \subset S_{i+1}$ with $k(\gamma) \le q$, and consider $\beta \subset S_{i+1}$ with $k(\beta) = q+1$. \Cref{fig:kcrossings} shows how to proceed: we construct $\gamma$ by running parallel to $\beta$ from its initial point through the handle at least once, and subsequently until the first time $\gamma$ emerges from the handle into a component of $S_i \setminus \beta$ containing some marked point disjoint from $\beta$ (such a point always exists, due to the assumption $N \ge 3$). Necessarily, $k(\gamma) < k(\beta)$, since the region of $S_i \setminus \beta$ containing the marked point must be bounded by some segment of $\beta$ other than the one making the final handle crossing, and so is accessible by following $\beta$ through some strict subset of its handle crossings. By induction, $P_\gamma \in \Gamma_{i+1}$, and so adjusting $\beta$ by $P_{\gamma} \in \Gamma_{i+1}$, we find $k(P_{\gamma}(\beta)) < k(\beta)$, completing the inductive step.  
\end{proof}

\subsection{From marked points to boundary components: disk pushing}\label{subsection:braidsdisks} Thus far, we have been considering framings on punctured surfaces $S \setminus \bp$. For the applications of the results in this section, we need to understand what happens when the points $p_i \in \bp$ are replaced with boundary components.

We briefly recall the classical theory. Let $S$ be a surface with $N$ boundary components, and let $\bar S$ be the closed surface obtained by capping each of these with a disk. Then the Birman exact sequence (again assuming $g(S) \ge 2$) for the inclusion $S \into \bar S$ has the form
\[
1 \to \tilde{Br}_N(\bar S) \to \Mod(S) \to \Mod(\bar S) \to 1,
\]
where $\tilde{Br}_N(\bar S)$ admits a description as a central extension
\[
1 \to \Z^N \to \tilde{Br}_N(\bar S) \to Br_N(\bar S) \to 1,
\]
the subgroup $\Z^N$ being generated by the twists around each of the boundary components. In our setting of ``absolute'' framings (see \Cref{footnote:absfr}), the twist about each boundary component preserves $\phi$, so that this restricts to give an extension
\[
1 \to \Z^N \to \Mod(S)[\phi] \cap \tilde{Br}_N(\bar S) \to \ker(\eta_\lambda) \to 1. 
\]
Here $\lambda: \bp \to \Z$ is the weight function associated to our boundary components, given a component $p_i$, $\lambda(p_i)$ is the winding number around $p_i$.
In particular, the preimage $\tilde{SBr}_N(S)$ of $SBr_N(S) = \ker(\eta_\lambda)$ in $\tilde{Br}_N(\bar S)$ is generated by lifts of the same set of half-twists as in \Cref{prop:generatesbr}, along with twists about all boundary components. By abuse of terminology, we will also call $\tilde{SBr}_N(S)$ a ``simple braid group''.

\section{Generating $r$-spin mapping class groups}\label{section:genspinmcg}

In this section we state and recall or prove the results on generating sets for framed and $r$-spin mapping class groups that we will employ in the proof of \Cref{mainthm:ambientPn}. A completely satisfactory theory exists for genus $g \ge 5$, as recalled in \Cref{theorem:assemblagegenset} below. Unfortunately, our inductive approach will require us to consider framed mapping class groups on surfaces of genus as low as three, and so it will be necessary to develop some additional results suitable for this regime.

\subsection{Assemblages and generation in genus at least five}\label{subsection:genusge5}
For $(S, \phi)$ a framed surface of genus $g(S) \ge 5$, a criterion for generation of $\Mod(S)[\phi]$ was obtained in \cite{strata3}. We recall that here. 

The generating sets consist of finitely many Dehn twists about admissible curves. They are ``coordinate-free'' in the sense that we do not just give one specific configuration of curves that generate (like the classical Humphries generating set for $\Mod(S)$), but rather provide a criterion under which some collection generates. There is a great deal of flexibility in the allowable intersection patterns, as captured in the notion of an {\em assemblage}, defined below.

\begin{construction}[Handle attachment]
    Let $S$ be a surface and $S' \subset S$ a subsurface. Let $c \subset S$ be a simple closed curve such that $c \cap S'$ is a single essential arc. Let $\nu(S' \cup c)$ be a regular neighborhood of the union $S' \cup c$. Then $\nu(S' \cup c)$ is said to be obtained from $S'$ by {\em attaching a handle along $c$}.
\end{construction}

\begin{definition}[Assemblage]\label{def:assemblage}
    Let $S$ be a surface and let $S' \subset S$ be a subsurface. Let $\cC = (c_1, \dots, c_k)$ be an ordered collection of simple closed curves on $S$. Define $S_0 = S'$. For $i \ge 1$ suppose that $c_i \cap S_{i-1}$ is a single essential arc, and define $S_i$ by attaching a handle to $S_{i-1}$ along $c_i$. We say that $S_i$ is {\em assembled from} $S'$ and the union of curves $c_1 \cup \dots \cup c_i$.

    An ordered configuration $\cC$ as above is said to be an {\em assemblage extending $S'$} if each $S_i$ for $1 \le i \le k$ is obtained by attaching a handle to $S_{i-1}$ along $c_i$, and if $S \setminus S_k$ is a union of disks and boundary-parallel annuli. 
\end{definition}

\begin{definition}[$h$-assemblage, type $E$]
    Let $\cC = (c_1, \dots, c_k)$ be an ordered collection of simple closed curves on $S$. Suppose that for some $\ell \le k$, the subcollection $\{c_1, \dots, c_\ell\}$ forms a simple configuration in the sense of \Cref{def:simpleconfig}, and that a regular neighborhood of their union is a subsurface $S'$ of genus $h$. Suppose moreover that $\cC$ is an assemblage extending $S'$. Then $\cC$ is said to be an {\em $h$-assemblage}. If moreover this subcollection is $E$-arboreal, then $\cC$ is said to be an $h$-assemblage {\em of type $E$}.
\end{definition}

Having established the theory of assemblages, we can now state the generating sets for framed mapping class groups obtained in \cite{strata3}.

    \begin{theorem}[Theorem B.II of \cite{strata3}]\label{theorem:assemblagegenset}
        Let $(S, \phi)$ be a framed surface. Let $\cC = (c_1, \dots, c_\ell)$ be an $h$-assemblage of type $E$ on $S$ for some $h \ge 5$. If $\phi(c) = 0$ for all $c \in \cC$, then
        \[
        \Mod(S)[\phi] = \pair{T_c\mid c \in \cC}.
        \]
    \end{theorem}

\subsection{Low genus}
\Cref{theorem:assemblagegenset} requires the existence of a simple configuration of genus at least $5$. In the course of our work, we will encounter situations where the total genus of the surface is at least $5$, but for which we can only produce a simple configuration on a surface of smaller genus. This will be handled by \Cref{prop:gencriterion}, to be proved in this subsection.

In preparation, we recall the notion of the {\em admissible subgroup} of a framed or $r$-spin mapping class group.

\begin{definition}[Admissible subgroup]
    Let $(S, \phi)$ be a surface equipped with an $r$-spin structure. The {\em admissible subgroup} $\cT_S \le \Mod(S)[\phi]$ is the group generated by all admissible twists on $S$:
    \[
    \cT_S = \pair{T_c \mid c \subset S,\ \mbox{$c$ nonseparating, }\phi(c) = 0}.
    \]
\end{definition}

\begin{remark}
    In prior work, the admissible subgroup was notated $\cT_\phi$, not $\cT_S$. Here we will have occasion to consider admissible subgroups on subsurfaces, for which the present notation is better suited.
\end{remark}

\begin{proposition}\label{prop:gencriterion}
    Let $(S,\phi)$ be a framed surface of genus $g \ge 5$, and let $S' \subset S$ be a subsurface of genus $g(S') \ge 2$ and of constant signature $-2$ (recall \Cref{def:signature}). 
    Let $c_1, \dots, c_{k}$ be a set of simple closed curves on $S$ such that $(c_1, \dots, c_k)$ forms an assemblage extending $S'$, with the special property that each $c_i$ enters and exits $S'$ exactly once. 
    Suppose $\Gamma \le \Mod(S)[\phi]$ contains the admissible subgroup $\cT_{S'}$ for $(S', \phi\left |_{S'}\right.)$, as well as the twists $T_{c_1}, \dots, T_{c_{k}}$. 
    Then $\Gamma = \Mod(S)[\phi]$.
\end{proposition}

This will mostly follow from the techniques of \cite{strata3}. Let us recall the relevant results therein.

\begin{proposition}[Cf. Proposition 5.11 of \cite{strata3}]\label{prop5.11}
    Let $(S, \phi)$ be a framed surface of genus $g \ge 5$. Then there is an equality
    \[
    \cT_{S} = \Mod(S)[\phi].
    \]
\end{proposition}

This will allow us to reduce the problem of exhibiting $\Mod(S)[\phi]$ to instead exhibiting $\cT_S$. To accomplish this latter task, we appeal to another result of \cite{strata3}. This is formulated in terms of a ``framed subsurface push subgroup''.

\begin{definition}[Framed subsurface push subgroup]
    Let $S' \subset S$ be a subsurface, and let $\Delta \subset \partial S'$ be a boundary component, not necessarily a boundary component of $S$. Let $\bar S'$ be the surface obtained from $S'$ by capping $\Delta$ with a disk. Then there is a {\em disk-pushing homomorphism} 
    \[
    \cP: \pi_1(UT \bar S') \to \Mod(S').
    \]
    The inclusion $S' \into S$ induces a homomorphism $i: \Mod(S') \to \Mod(S)$, and the {\em subsurface push subgroup} is defined as the image of $i \circ \cP$. When $S$ is equipped with a framing $\phi$, the {\em framed subsurface push subgroup}, notated $\tilde \Pi(S')$, is the intersection with $\Mod(S)[\phi]$:
    \[
    \tilde \Pi(S') = \im(i \circ \cP) \cap \Mod(S)[\phi].
    \]
\end{definition}

There is an important special case of this construction. Suppose $b \subset S$ is an oriented simple closed curve with $\phi(b) = -1$. Then $S' = S \setminus \{b\}$ has a distinguished boundary component $\Delta$ corresponding to the left side of $b$, satisfying $\phi(\Delta) = -1$. For this choice of $(S', \Delta)$, we streamline notation, setting
\[
\tilde \Pi(b) := \tilde \Pi(S \setminus \{b\}).
\]

\begin{proposition}[Cf. Proposition 3.10 of \cite{strata3}]\label{prop3.10}
    Let $(S, \phi)$ be a framed surface of genus $g \ge 5$. Let $(a_0,a_1,b)$ be an ordered $3$-chain of simple closed curves with $\phi(a_0) = \phi(a_1) = 0$ and $\phi(b) = -1$. Let $H \le \Mod(S)$ be a subgroup containing $T_{a_0}, T_{a_1}$, and the framed subsurface push subgroup $\tilde \Pi(b)$. Then $H$ contains $\cT_{S}$.
\end{proposition}

This will reduce the problem of generating $\cT_S$ to generating the subgroup $\tilde \Pi(b)$ for some suitable $b$. To accomplish this, we will need a final pair of results from \cite{strata3}.

\begin{lemma}[Cf. Lemma 3.3 of \cite{strata3}]\label{lemma3.3}
    Let $S' \subset S$ be a subsurface and let $\Delta$ be a boundary component of $S'$ such that $\phi(\Delta) = -1$, giving rise to the framed subsurface push subgroup $\tilde \Pi(S')$. Let $a \subset S$ be an admissible curve disjoint from $\Delta$ such that $a \cap S'$ is a single essential arc. Let $a' \subset S'$ be an admissible curve satisfying $i(a,a') = 1$. Let $S'^+$ be the subsurface given by attaching a handle to $S'$ along $a$. Then $\tilde \Pi(S'^+) \le \pair{T_a, T_{a'}, \tilde \Pi(S')}$.
\end{lemma}

\begin{lemma}[Cf. Lemma 3.5 of \cite{strata3}]\label{lemma3.5}
Let $S' \subset S$ be a subsurface with a boundary component $\Delta$ satisfying $\phi(\Delta) = -1$. Let $\cC$ be a filling arboreal simple configuration of admissible curves on $S'$, and suppose there exist $a, a' \in \cC$ for which $a \cup a' \cup \Delta$ form a pair of pants. Then $\tilde \Pi(S')$ is contained in the group $\cT_{\cC}$ generated by the twists about curves in $\cC$. 
\end{lemma}

\begin{remark}
    Our statement of \Cref{lemma3.5} does not match \cite[Lemma 3.5]{strata3} (nor the original result \cite[Lemma 9.4]{saltertoric} it quotes) exactly; we comment here on how to reconcile the two. The statement there is in terms of {\em networks} of simple closed curves, a notion we are now trying to deprecate in favor of simple configurations. In the language of the present paper, a network $\cN$ can be defined as a set of simple closed curves ({\em not} merely isotopy classes) whose isotopy classes form a simple configuration, subject to the condition that there are no triple intersections of curves in $\cN$. The statement of \cite[Lemma 3.5]{strata3} asserts that if $\cN$ is a network satisfying the hypotheses of \Cref{lemma3.5}, then the conclusion of \Cref{lemma3.5} holds. But by hypothesis, the simple configuration in \Cref{lemma3.5} is required to be arboreal; in particular, there can be no triple intersections of curves, as this would appear as a triangle in the intersection graph.
\end{remark}

To prove \Cref{prop:gencriterion}, we will also require a few additional new results (\Cref{lemma:admisstrans,lemma:0con1con,lemma:c0con}).

\begin{definition}
    Let $(S, \phi)$ be a framed surface, and let $w,k$ be integers, with $k = 0$ or $k = 1$. The graph $\cG_{w,k}(S)$ is defined as follows:
    \begin{itemize}
        \item Vertices of $\cG_{w,k}(S)$ consist of nonseparating simple closed curves $c \subset S$ with winding number $\phi(c) = w$,
        \item There is an edge between vertices $c, c'$ whenever $i(c,c') = k$.
    \end{itemize}
\end{definition}

\begin{lemma}\label{lemma:admisstrans}
    Let $(S, \phi)$ be a framed surface for which $\cG_{-1,1}(S)$ is connected. Then $\cT_S$ acts transitively on the vertices of $\cG_{-1,1}(S)$.
\end{lemma}

\begin{proof}
    Let $b, b' \in \cG_{-1,1}(S)$ be given. By hypothesis, there is a path $b = b_0, \dots, b_n = b'$ in $\cG_{-1,1}(S)$. It therefore suffices to exhibit $T_{a_i} \in \cT_S$ such that $T_{a_i}(b_i) = b_{i+1}$.
    By twist-linearity, $T_{b_i}^{\pm 1}(b_{i+1})$ is admissible for an appropriate choice of sign; set $a_i$ to whichever of these is. Then $b_{i+1} = T_{a_i}^{\pm 1}(b_i)$ for an appropriate choice of sign.
\end{proof}

We will require some results that allow us to establish the connectivity of $\cG_{-1,1}(S)$.

\begin{lemma}\label{lemma:0con1con}
    Let $(S,\phi)$ be a framed surface for which $\cG_{-1,0}(S)$ is connected. Then $\cG_{-1,1}(S)$ is likewise connected.
\end{lemma}
\begin{proof}
    It suffices to exhibit, for adjacent vertices $c, c' \in \cG_{-1,0}(S)$, a third curve $c''$ such that $c, c'', c'$ determines a path in $\cG_{-1,1}(S)$. To do this, let $d$ be an arbitrary curve with $i(c,d) = i(c',d) = 1$; set $w = \phi(d)$. By twist-linearity, $c'' = T_c^{\pm(w-1)}(d)$ satisfies $\phi(c'') = -1$ for appropriate choice of sign, and also satisfies $i(c,c'') = i(c,c') = 1$. 
\end{proof}

\begin{lemma}\label{lemma:c0con}
    Let $(S,\phi)$ be a framed surface of constant signature $-2$, with $g(S) \ge 2$. Then $\cG_{-1,0}$ is connected.
\end{lemma}
\begin{proof}
     By \Cref{lemma:framedontospin}, since $(S,\phi)$ has constant signature $-2$, the inclusion $S \into \bar S$ (where $\bar S$ is obtained from $S$ by capping all boundary components of $S$ with disks) induces a surjection $\Mod(S)[\phi] \onto \Mod(\bar{S})$, the full mapping class group of $S$. The Birman exact sequence for the pair $S, \bar S$ therefore restricts to the following short exact sequence (cf. \Cref{cor:42}, which takes place in the punctured setting):
     \[
     1 \to \tilde{SBr}_N(S) \to \Mod(S)[\phi] \to \Mod(\bar{S}) \to 1.
     \]
    It follows that $\Mod(S)[\phi]$ admits a generating set consisting of the following elements:
    \begin{enumerate}
        \item A set $T_{a_1}, \dots, T_{a_m}$ of admissible twists whose images generate $\Mod(\bar{S})$,
        \item Any generating set for $\tilde{SBr}_N(S)$, e.g. of the form given in \Cref{prop:generatesbr} along with the boundary twists.
    \end{enumerate}

    By the framed change-of-coordinates principle (\Cref{prop:ccp}), $\Mod(S)[\phi]$ acts transitively on the vertices of $\cG_{-1,0}(S)$. Let $b \in \cG_{-1,0}(S)$ be an arbitrary vertex. According to the {\em Putman trick} \cite[Theorem XX]{putmantrick}, to show $\cG_{-1,0}(S)$ is connected, it suffices to exhibit paths in $\cG_{-1,0}(S)$ connecting $b$ and $g_i(b)$, for some set $g_1, \dots, g_n$ of generators for $\Mod(S)[\phi]$.

    To proceed, we choose a maximally convenient generating set of the type described above. Choosing $b \in \cG_{-1,0}(S)$ arbitrarily, the framed change-of-coordinates principle ensures that there is a set of curves $a_1, \dots, a_m$ such that the images of $T_{a_1}, \dots, T_{a_m}$ generate $\Mod(\bar{S})$, and such that $i(a_i, b) = 0$ for $i \le m-1$ and $i(a_m,b) = 1$. 
    
    $T_{a_i}(b) = b$ for $i \le m-1$, so there is nothing to show for these generators. Since $g(S) \ge 2$, there exists $b' \in \cG_{-1,0}(S)$ disjoint from both $b, a_m$, and then $b,b',T^{\pm 1}_{a_m}(b)$ is a path in $\cG_{-1,0}(S)$.

    It remains to consider generators of the second type, for $\tilde{SBr}_N(S)$. Certainly the boundary twists fix $b$, so there is nothing to show here. Let $\alpha_1 \subset S$ connect distinct boundary components of $S$ and satisfy $i(\alpha_1, b) = 1$, and then choose arcs $\alpha_2, \dots, \alpha_k$ satisfying the hypotheses of \Cref{prop:generatesbr} and moreover satisfying $i(\alpha_i, b) = 0$ for $i \ge 2$. Then the half-twists about $\alpha_i$ for $i > 1$ fix $b$, so there is again nothing to show. Finally, by the framed change-of-coordinates principle, there is $b' \in \cG_{-1,0}(S)$ disjoint from both $b$ and $\alpha_1$, and $b, b', P_{\alpha_1}(b)$ forms a path in $\cG_{-1,0}(S)$.
\end{proof}

We are now in position to prove \Cref{prop:gencriterion}.

\begin{proof}[Proof of \Cref{prop:gencriterion}]
Let $b \subset S'$ be a nonseparating curve with $\phi(b) = -1$. By the framed change-of-coordinates principle (\Cref{prop:ccp}), there is a $3$-chain $a_0,a_1, a_2$ of admissible curves such that $b \cup a_0 \cup a_2$ bounds a pair of pants and such that $i(b,a_1) = 0$. Let $S'' \subset S'$ be the subsurface given as a neighborhood of $a_0 \cup a_1 \cup a_2$ (noting that one boundary component of $S''$ is $b$). Then $a_0, a_1, a_2$ satisfy the hypotheses of \Cref{lemma3.5}, so that $\tilde \Pi(S'') \le \Gamma$. 

Continuing to apply the framed change-of-coordinates principle, extend $a_0, a_1, a_2$ to an assemblage $a_0, \dots, a_p$ of admissible curves for $S'$. By repeated applications of \Cref{lemma3.3} (invoking change-of-coordinates and the hypothesis (1) that $\cT_{S'} \le \Gamma$ as necessary), it follows that the framed subsurface push subgroup $\tilde \Pi(S' \setminus \{b\})$ is contained in $\Gamma$.

Next apply \Cref{lemma3.3} to the surface $S_1$ obtained by attaching a handle to $S'$ along $c_1$, choosing $b\subset S'$ to be any nonseparating curve satisfying $\phi(b) = -1$ and $i(c_1, b) = 0$; we conclude $\tilde \Pi(S_1 \setminus \{b\}) \le \Gamma$. Since $S'$ has constant signature $-2$ by hypothesis, \Cref{lemma:admisstrans,lemma:0con1con,lemma:c0con} together show that $\cT_{S'}$ acts transitively on nonseparating curves $b \subset S'$ with $\phi(b) = -1$. It follows that $\Gamma$ contains $\tilde \Pi(S_1 \setminus \{b'\})$ for {\em any} such $b' \subset S'$.

As $c_2$ enters and exits $S'$ once, there is some nonseparating $b_2 \subset S'$ disjoint from $c_2$ with $\phi(b_2) = -1$. Applying \Cref{lemma3.3}, it follows that $\tilde \Pi(S_2 \setminus \{b_2\}) \le \Gamma$.

This argument can be continued with $c_3, \dots, c_k$, leading to the conclusion that $\Gamma$ contains $\tilde \Pi (S \setminus \{b\})$ for some $b \subset S'$. By the framed change-of-coordinates principle, $b$ can be completed to a chain $a_0, a_1, b$ with $a_0, a_1 \subset S'$ admissible. By \Cref{prop3.10}, $\Gamma$ contains the admissible subgroup $\cT_S$, and since $g(S) \ge 5$, by \Cref{prop5.11}, we conclude that $\Mod(S)[\phi] \le \Gamma$; the reverse containment holds by hypothesis.
\end{proof}

\subsection{A miscellaneous result}

The lemma below provides a useful criterion for generating $r$-spin mapping class groups in the closed setting; it will be used in the final stages of the proof of \Cref{mainthm:ambientPn}.

\begin{lemma}\label{lemma:rspingencrit}
    Let $S$ be a closed surface of genus $g(S) \ge 5$, and let $r \mid \rho$ be integers. Let $\phi$ be a $\rho$-spin structure, and let $\bar \phi$ be the mod-$r$ reduction of $\phi$. Then $\Mod(S)[\bar \phi]$ is generated by $\Mod(S)[\phi]$ along with any Dehn twist $T_c$ for $c$ nonseparating satisfying $\phi(c) = r$.
\end{lemma}

\begin{proof}
    By the framed change-of-coordinates principle (in the guise of \Cref{lemma:Earbexists}), $c$ can be extended to an $E$-arboreal filling simple configuration on $S$ of $\phi$-admissible curves. By \cite[Corollary 3.11]{strata3}, the set of such twists generates some higher spin mapping class group that contains $\Mod(S)[\phi]$, and by construction is contained in $\Mod(S)[\bar \phi]$. By \Cref{lemma:rspincontainment}, this must therefore be some reduction of $\phi$ mod some $\rho'$ dividing $\rho$ but divisible by $r$. Since it contains $T_c$ with $\phi(c) = r$ we must have $\rho' = r$.
\end{proof}

\section{Topology of complete intersection curves}\label{section:topologyCI}

Having established the necessary preliminaries about framed/$r$-spin mapping class groups, we turn now to the study of families of complete intersection curves. The proof of \Cref{mainthm:ambientPn} is ultimately by induction, proceeding by deforming to reducible curves. In this section, we gather some facts about the topology and geometry of such curves. 

\subsection{Nodal degenerations}\label{subsection:nodaldegen}

In the course of this paper will often deal with nodal degenerations of a curve $C$ embedded in a surface $X$. Here we establish some basic facts, notation, and conventions. 

\begin{definition}[Nodal degeneration]
   Let $\cL$ be a line bundle on a smooth projective surface $X$. Let $f \in H^0(X, \LL)$, with zero locus $Z(f) = C$, assumed to be smooth. 
   
    Let $\disk \subset \C$ denote the unit disk, and let $\varphi: \disk \to H^0(X, \LL) $ determine a family, with $C_t$ the zero locus of $\varphi(t)$. 
We say that $\varphi$ is a {\em nodal degeneration} of $C$ if it satisfies the following conditions:
   
   \begin{enumerate}
       \item  $C_1 = C$.
       \item  The curves $C_t$ are smooth for $t \neq 0$.
       \item  The curve $C_0$ has a single node at some point $p \in C_0$.
       \item $\frac{\partial }{\partial  \bar t} (\varphi) (p)|_{t =0} =0$ and $\frac{\partial}{\partial t} (\varphi) (p)|_{t =0}  \neq 0 .$
   \end{enumerate}
\end{definition}

A nodal degeneration as above determines a family of smooth curves over $\disk^ \setminus \{0\}$ arising from pulling back the universal family along $\varphi$. 

The {\em Picard-Lefschetz formula} states that the monodromy associated to this family is given by a right-handed Dehn twist $T_\alpha$ about some simple closed curve $\alpha \subseteq C$. If $C_0$ is irreducible, $\alpha$ is nonseparating.

\subsection{Smoothing reducible curves}\label{subsection:CDEsetup}
Let $\bd = (d_1, \dots, d_{n-1})$ be a multidegree. We consider a smooth complete intersection surface $X$ of multidegree $(d_1, \dots, d_{n-2})$; in case $n = 2$ we take $X = \CP^2$. We let $C \subset X$ be a smooth complete intersection curve of multidegree $\bd$, and we take $D \subset X$ to be a smooth complete intersection curve of multidegree $\bd':= (d_1, \dots, d_{n-2},1)$ that intersects $C$ transversely in $N:= \deg(C) = \Pi(\bd)$ points.

For $1 \le i \le n-1$, let $f_i \in H^0(\CP^n; \cO(d_i))$ be equations that together define $C \subset \CP^n$, and such that $f_1, \dots, f_{n-2}$ defines $X$. Let $h \in H^0(\CP^n; \cO(1))$ define the hyperplane $H$ such that $D = X \cap H$. Then $C \cup D$ is defined by the equations $f_1, \dots, f_{n-2}, f_{n-1}h$.

Let 
\[
U \subset H^0(\CP^n; \cO(d_{n-1}+1))
\]
be the Zariski-open subset of elements defining smooth curves and let the complement of $U$ be $\Sigma$; note $f_{n-1}h \in \Sigma$. However, $f_{n-1} h$ is not a smooth point of $\Sigma$. We shall now study the local topology of $\Sigma$ at this point.

    Let $\tilde \Sigma = \{(x,f) \in X \times \Sigma\, |\, x \in \textrm{Sing}(f)\}.$ The variety  $\tilde \Sigma$ is well known to be smooth and there is a surjective birational map  $\pi: \tilde \Sigma \to \Sigma$ (cf. \cite[p. 30]{GKZ} -  note that they consider projectivized sections as opposed to sections, but the proofs go through {\em mutatis mutandis}). 

\begin{lemma}\label{injder}
    Let $(f,x) \in \tilde \Sigma$ be such that $x$ is a nodal singularity of $Z(f).$ Then $\pi_* : T_{(f,p)} \tilde \Sigma \to T_f \Sigma$ is injective.
\end{lemma}
\begin{proof}
    This is essentially a local computation. Let $x_1, x_2$ be analytic parameters near $x$, such that locally $f= x_1^2 + x_2^2$. Then $T_{(f,x)}\tilde \Sigma \subseteq H^0(X,\LL ) \times T_x X$ can be identified with the set of $(g,v)$ satisfying the following linear constraints:
\begin{enumerate}
    \item $g(x)=0$.
    \item $\partial_i g(x) = - 2 v_i$, where $v_i$ is the $i^{th}$ coordinate of $v$.
\end{enumerate}

The map $\pi_*$ is then just giving by forgetting $v$. But this map is injective since $-2v_i = \partial_i g(x).$
\end{proof}
\begin{proposition}\label{prop:smallB}
Let $\{x_1, \dots, x_N\}$ denote the intersection points of $C \cap D$. Let $V_i$ denote small balls (in $X$) around the $x_i $. Let $B$
denote a sufficiently small ball around $f_{n-1}h$ in $H^0(\CP^n; \cO(d_{n-1} +1))$. Then the following assertions hold. 
\begin{enumerate}
    \item $B \cap \Sigma = \cup \Sigma_i$, where the $\Sigma_i$ are certain smooth closed (in $B$) analytic (not algebraic) hypersurfaces. The $\Sigma_i$ are distinct as subsets, but not necessarily disjoint. 
    \item Any $g \in \Sigma_i \setminus \cup_{j \neq i} \Sigma_j$ has singularities only in $V_i$.
\end{enumerate}
\end{proposition}
\begin{proof}
    Let $\pi: \tilde \Sigma \to \Sigma$ be as above, and let $\pi_X: \tilde \Sigma \to X$ denote the other projection. Let us note that $\pi^{-1}(f_{n-1}h)$ consists of the points $(x_i, f_{n-1}h)$. Let $U_i$ denote balls in $\tilde \Sigma$ around each of these points, taken to be small enough so that $\pi_{X}(U_i)$ are disjoint sets contained in the corresponding $V_i$. Let $ B$ be a ball in $H^0(\CP^n ,\OO(d_{n-1}+1))$ around $f_{n-1}h$, small enough that $\pi^{-1}(B)$ is contained in the union of  the $U_i.$
    
    Define $\tilde \Sigma_i = \pi^{-1}(B \cap \Sigma) \cap U_i$. By Lemma \ref{injder}, $\pi_*$ is injective at each point in  $\pi^{-1}(f_{n-1}h)$ and so by the inverse function theorem, $\tilde \Sigma_i$ is mapped isomorphically onto its image, which we define to be $\Sigma_i$. Each such $\Sigma_i$ is a smooth closed analytic hypersurface in $B$.
    By construction, any $f \in B \cap \Sigma$ has singularities only in $\bigcup_i V_i$, and further if $\textrm{Sing}(f) \cap V_i$ is nonempty, $f \in \Sigma_i$.

     To show that $\Sigma_i$ are \emph{distinct} analytic hypersurfaces, consider $g \in H^0(\CP^n ,\OO(d_{n-1}+1))$ such that $Z(g)$ is singular at $x_i$ and $Z(g)$ does not contain any of the other point. For a generic such $g$, the curve $f_{n-1}h + tg$ lies on $\Sigma_i$ but not on the others.  
\end{proof}

For any $g \in U \cap B$ , the equations $f_1, \dots, f_{n-2}, g$ define a smooth curve $E \subset X$ of multidegree $\bd^+:= (d_1, \dots,d_{n-2}, d_{n-1}+1)$. Now if $B$ is sufficiently small, such an $E$ naturally decomposes as 
\[
E = \tilde C \cup \tilde D,
\]
where $\tilde C$ (resp. $\tilde D$) is homeomorphic to $C$ (resp. $D$) with $N = \Pi(\bd)$ punctures blown up into boundary components $\beta^{C/D}_1, \dots, \beta^{C/D}_N$. Then $E$ is realized by gluing $\tilde C$ to $\tilde D$ via identifying $\beta_i^C$ to $\beta_i^D$ for all $i$. Denote by $\beta_i$ the curve in $E$ obtained by the identification of $\beta_i^C$ and $\beta_i^D$.

More intrinsically, $\tilde C$ (resp. $\tilde D$) is obtained as the {\em real oriented blowup} of $C$ (resp. $D$) at the points of intersection $C \cap D$, so that there is a natural identification of $\beta_i^{C/D}$ with the unit tangent space to $C$ or $D$ at the corresponding point of $C \cap D$.

The next lemma shows that $\tilde C$ is endowed with an enhancement of the $r(\bd)$-spin structure naturally present on a complete intersection curve. To proceed, recall the notion of {\em signature} from \Cref{def:signature}, as the tuple of $\phi$-values of boundary components. 

\begin{lemma}\label{lemma:tildeDframed}
    Suppose $r(\bd) \ge 0$. Then $\tilde C$ is equipped with a canonical framing $\phi_{C,D}$ that descends to the $r$-spin structure $\phi_\bd$ on $C$. The framing $\phi_{C,D}$ has constant signature $-r(\bd)-1$.
\end{lemma}

\begin{proof}
   $D$ induces a section $\sigma_D$ of $\cO_C(1)$ with divisor $C \cap D$, giving the $r(\bd)$-spin structure $\phi_\bd$. By adjunction, $\sigma_D^{\otimes r(\bd)}$ gives a section of $K_C$. Up to a $\C^*$-ambiguity, this can be interpreted as a $1$-form $\omega_D \in H^0(C; \Omega^1)$ with zeroes of order $r(\bd)$ at each point of $C \cap D$. As discussed in greater detail in \cite[Section 7.2]{strata3}, the real oriented blowup procedure associates to $\omega_D$ a framing on $\tilde C$. A zero of order $k$ of $\omega_D$ is transformed into a boundary component of winding number $-k-1$. Since each zero of $\omega_D$ is of order $r(\bd)$, the signature of the associated framing of $\tilde C$ is seen to be constant, equal to $-r(\bd) - 1$ as claimed.
\end{proof}

\para{Basepoint conventions} The results of \Cref{mainthm:ambientPn} and \ref{mainthm:ambientX} concern the {\em global} monodromy representation, the definition of which requires a choice of basepoint in the family of curves under study that is fixed once and for all. Here we discuss the conventions that we adopt.

We establish the following notation. Given $g \in B$, define
\[
E(g) := Z(f_1, \dots, f_{n-2},  g) \subset \CP^n.
\]

\begin{convention}\label{convention:basepoint}
    Recall from the Introduction that for $X$ a smooth complete intersection surface of multidegree $\bd' = (d_1,\dots, d_{n-2})$ and $(n-1)$-tuple $\bd =  (d_1, \dots, d_{n-2}, d_{n-1})$, the space $U_{X, \bd}$ is the parameter space of smooth complete intersection curves $C \subset X$ of multidegree $\bd$; recall also the notation $\bd^+ = (d_1, \dots, d_{n-2}, d_{n-1}+1)$ established just above, and set $\bd'_1:=(d_1, \dots, d_{n-2},1)$. The basepoint $E \in U_{X,\bd^+}$ is specified as follows: choose curves $C \in U_{X,\bd}$ and $D \in U_{X,\bd'_1}$ meeting transversely. Then $E$ is chosen to be $E(g)$ for some in $g \in B$, where $B$ is a sufficiently small ball around $f_{n-1}h$ as in \Cref{prop:smallB}.
\end{convention}

Of course, this does not specify $E$ completely, leading to a mild ambiguity in the monodromy representation. To formulate this, choose a Riemannian metric on $X$, and let $V_\eta \subset X$ be a union of $N$ disjoint open metric balls centered at each of the points of $C \cap D$, each of radius $\eta>0$. 

\begin{lemma}
    With $C,D$ as above, there are choices of $B, \eta$ such that the family of curves $E(g) \setminus V_\eta$ for $g \in B \cap U$ is topologically trivial.
\end{lemma}
\begin{proof}
    Fix $\eta > 0$ small enough so that $V_\eta$ indeed consists of $N$ disjoint open balls. Since smoothness is an open condition and since $C, D$ are smooth, for $B$ a small enough ball the complements $E(g)\setminus  V_\eta$ is a surface with two components homeomorphic to $\tilde C$ and $\tilde D$, for all $g \in B \subset H^0(\CP^n; \cO(d_{n-1}+1))$, not just for $g \in B \cap U$. Being a family over a contractible base, it is topologically trivial as claimed.
\end{proof}

It follows that there is a canonical way to identify the curves $E(g)$ for $g \in B$ away from neighborhoods of the curves $\beta_i$. Consequently, any ambiguity in the monodromy introduced by a change of basepoint within $B$ is supported on such a neighborhood, and thus consists of a product of Dehn twists about the curves $\beta_i$. The next lemma shows that the monodromy of the family of curves $E(g)$ for $g \in B$ is as large as possible, ultimately removing this ambiguity.

\begin{lemma}\label{lemma:boundaryVC}
    The monodromy 
    \[
    \rho_{loc}: \pi_1(B\cap U)   \to \Mod(\Sigma_{g(\bd^+)})
    \]
    of the family of curves $E(g)$ for $g \in B$ is the subgroup generated by the Dehn twists
    \[
    T_{\beta_1}, \dots, T_{\beta_N}.
    \]
    In particular, each curve $\beta_i \subset E$ is a vanishing cycle.
\end{lemma}
\begin{proof}
Following \Cref{prop:smallB}, $B \cap U$ has the structure of a complement of analytic hypersurfaces $\Sigma_i$ intersected with a small ball. Recall that the $\Sigma_i$ are distinct. Let $g \in B \cap U$.

As is well-known, the fundamental group $\pi_1(B \cap U)$ is therefore generated by the $N$ conjugacy classes of meridians around each of the components $\Sigma_i$. Let $V_i$ denote a small ball in $X$ centred at $x_i$. We then note that any meridian around $\Sigma_i$ has to be sent to a Dehn twist about a simple closed curve in $E(g)$ that is contained in $E(g) \cap V_i$. But for $B$ sufficiently small $E(g) \cap V_i$ is a cylinder around $\beta_i$. Thus the meridianal loop is sent to   $T_{\beta_i}$. Since the $\beta_i$ are pairwise disjoint, the twists $T_{\beta_i}$ pairwise commute. Thus, the entire conjugacy class of a meridian is sent under $\rho_{loc}$ to the twist $T_{\beta_i}$, showing that the image of $\rho_{loc}$ is the free Abelian group generated by $T_{\beta_1}, \dots, T_{\beta_N}$ as claimed. 
\end{proof}

\subsection{Pencils}
In the sequel, we will probe the topology of the family of complete intersection curves in $X$ by holding $D$ fixed and letting $C$ vary. Here we prove that a family of such $C$ can be chosen that avoids pathologies. 

\begin{definition}[Maximally generic pencil]\label{def:generic}
    A pencil $C_t$ of curves in $X$ is {\em maximally generic rel $D$} if the following conditions hold:
    \begin{enumerate}
        \item $D$ is disjoint from the base locus of $C_t$,
        \item $D$ intersects every singular fiber $C_{t_s}$ transversely and in the smooth locus,
        \item Each $C_t$ has at most one simple tangency with $D$.
        \item Let $C_{t_i}$ be a curve that intersects $D$ at $p$ with a simple tangency. Then there is a local coordinate $s$ for $\CP^1$ centered at $t_i$ and a local coordinate $z$ for $D$ centered at $p$ such that the equation of $C_t$ restricted to a neighborhood of $p$ is $z^2 - s$.
    \end{enumerate}
\end{definition}

\begin{lemma}\label{lemma:maxgenexists}
    Let $X$ be a smooth complete intersection surface, and let $C,D$ be smooth complete intersection curves in $X$ intersecting transversely. Then $C$ admits an extension to a pencil $C_t$ that is maximally generic rel $D$.
\end{lemma}
\begin{proof}
   The linear system $|C|$ is very ample on $D$, and furthermore the dual variety $D^{\vee} \subseteq |C| $ is a proper closed subvariety. A general line $L$ in $|C|$ passing through $C$ meets $D^{\vee}$ transversely at finitely many points $H_1 , \dots H_k$, corresponding to curves $C_1, \dots, C_k$ in $\abs{C}$. Each $C_i$ intersects $D$ transversely except at one non-reduced point $p$ at which $C_i$ and $D$ have a simple tangency. Transversality of the intersection also implies that if $H \in D^{\vee} \cap L$, we may pick analytic local coordinates for $L$ centered at $H$, denoted $s$, and analytic local coordinates for  $D$ centered at $p$, denoted $z$, such that the section $H_s$ is in coordinates  $H_s(z) = z^2 -s$.
\end{proof}

\section{Simple braids in the monodromy}\label{section:exhibitsimplebraids}

The purpose of this section is twofold. In \Cref{subsection:monobraids}, we prove that the monodromy group $\Gamma_{\bd^+} \le \Mod(E)$ contains a subgroup that restricts on $\tilde C$ to the simple braid group studied in \Cref{section:simplebraids}; it likewise contains some other subgroup restricting to the simple braid group on $\tilde D$. In \Cref{subsection:lifting}, we leverage this to lift configurations of vanishing cycles on $C$ to configurations on $\tilde C \subset E$.

\subsection{Simple braids in the monodromy}\label{subsection:monobraids}
Following \Cref{lemma:maxgenexists}, let $C_t$ be maximally generic rel $D$. Enumerate the fibers with simple tangencies to $D$ as $C_{t_1}, \dots, C_{t_k}$. Then there is a map 
\begin{align*}
\beta: \CP^1 \setminus \{t_1, \dots, t_k\} &\to \UConf_N(D)\\
t &\mapsto C_t \cap D.
\end{align*}
Note that by (2) of \Cref{def:generic}, the map $\beta$ extends over $t_s \in \CP^1$ with $C_{t_s}$ singular. 
\begin{proposition}\label{prop:simplebraidmonodromy}
In the above setting, for $N = \abs{C_t \cap D} \ge 3$, the induced map 
\[
\beta_*: \pi_1(\CP^1 \setminus \{t_1, \dots, t_k\}) \to Br_N(D)
\]
has image $\im(\beta_*) = SBr_N(D)$, the simple braid group. The same is true with the roles of $C$ and $D$ exchanged.
\end{proposition}

\begin{proof}
    \Cref{lemma:maxgenexists} holds regardless of the multidegrees of $C$ and $D$, so we are free take either one to play either role. For simplicity, in the argument below we will hold $D$ fixed and let $C$ vary.
    
    For ease of notation, set
    \[
    B:= \CP^1 \setminus \{t_1, \dots, t_k\},
    \]
    and choose a basepoint $t_0 \in B$. 
    For the purposes of this proof, a {\em path} shall mean an embedding $\tau: [0,1] \to \CP^1$ such that $\tau(0) = t_0$, $\tau(1) = t_i$ for some $1 \le i \le k$, and $\tau\left |_{(0,1)}\right. \subset B$.
    
    The pencil map $\pi: X \dashedrightarrow \CP^1$ restricts to a realization $\pi: D \to \CP^1$ of $D$ as a simple branched cover of degree $N$. Fix an identification of $\pi^{-1}(t_0)$ with the integers $1, \dots, N$. We then obtain a monodromy homomorphism $\mu: \pi_1(B) \to S_N$, where $S_N$ denotes the symmetric group on $N$ letters. In particular, each based meridian (i.e. a loop in $B$ obtained by a small circle around some $t_i$, based at $t_0$ by some choice of path) is assigned a well-defined transposition in $S_N$, called the {\em local monodromy} of the meridian.
    
    For each point $t \in \CP^1$, the preimage $\pi^{-1}(t) \subset D$ is given by the intersection $C_t \cap D$. It follows that we can study the map $\beta: B \to \UConf_N(D)$ by understanding the effect of dragging $t$ around some loop in $B$. In particular, let $\gamma \in \pi_1(B)$ be a meridian based via the path $\tau \subset \CP^1$, with local monodromy $(ij) \in S_N$. Then $\beta(\gamma) \in Br_N(D)$ is the half-twist in $D$ along the arc comprised of the lifts of $\tau$ based at $i,j \in \pi^{-1}(t_0)$. In this way, we can associate to each path $\tau$ in $\CP^1$ an arc $\tilde \tau$ in $D$ and the associated half-twist.

    To prove the claim, we will appeal to \Cref{prop:generatesbr}, exhibiting a set $\alpha_1, \dots, \alpha_k$ of arcs on $D$ satisfying the necessary hypotheses. Let $\tau_1, \dots, \tau_k$ be a system of paths with disjoint interiors, with $\tau_i$ terminating at $t_i$. We claim that the corresponding lifts $\tilde{\tau_i}$ satisfy properties (1)-(4) of \Cref{prop:generatesbr}. An example illustrating our argument is shown in \Cref{fig:example}.

    \begin{figure}
\centering
		\labellist
        \small
		\endlabellist
\includegraphics[width=\textwidth]{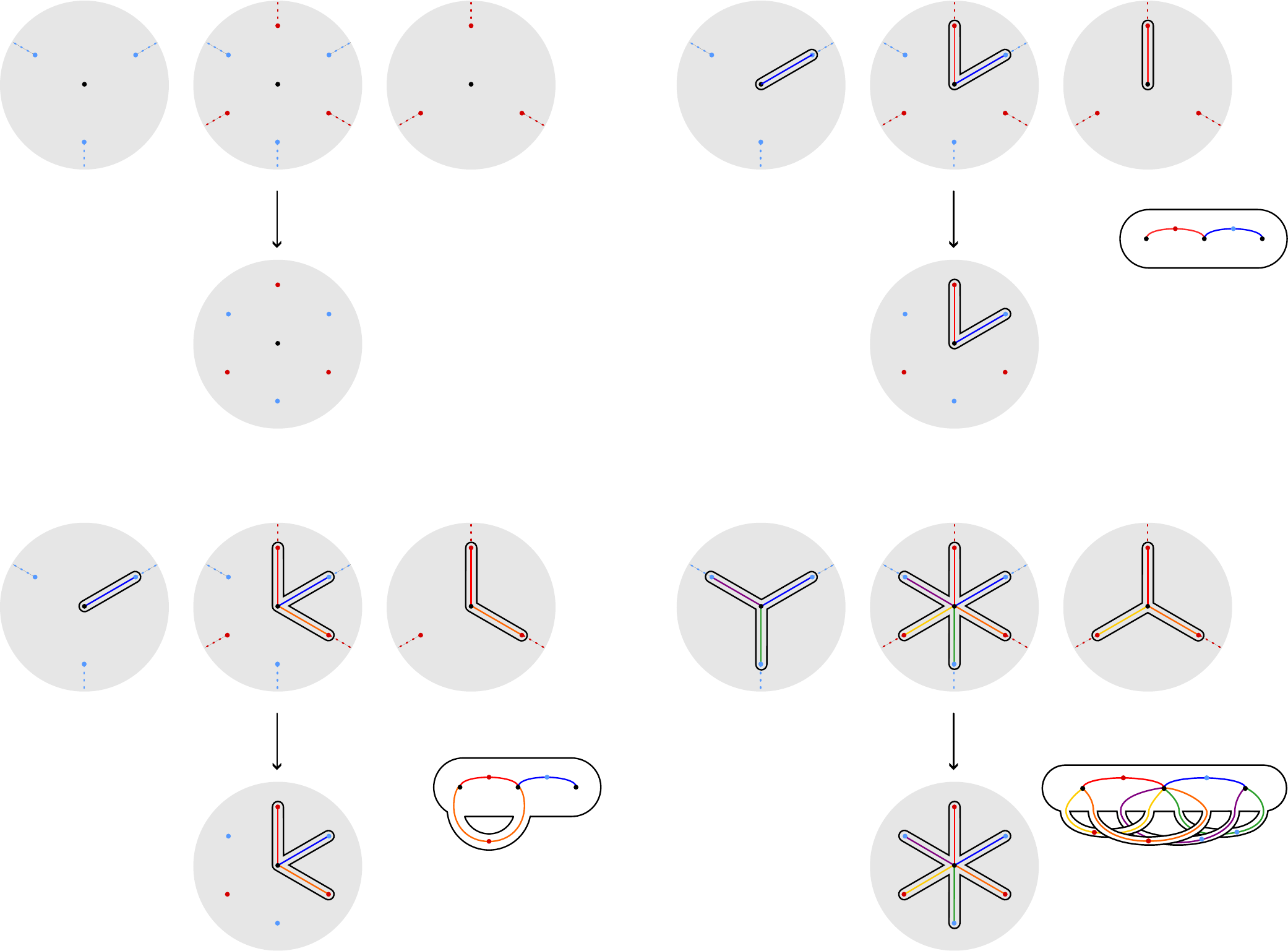}
\caption{Illustrating the construction in the case of $D$ an elliptic curve represented as a $3$-sheeted simple branched covering. First panel: a schematic picture of the branched covering, with branch cuts indicated as dashed line segments. Second panel: selecting the first $N-1 = 2$ arcs so as to span a disk. Third panel: each additional arc attaches a $1$-handle to the existing surface. Fourth panel: When all arcs have been added, the neighborhood fills $D$ away from a union of disks. A simple rendering of the neighborhood is included in the lower-right corner of panels 2-4.}
\label{fig:example}
\end{figure}
    
    First note that the disjointness property (1) will be satisfied for {\em any} collection of arcs of the form $\tilde \tau$ (lifted from paths on $\CP^1$), so long as the corresponding paths are disjoint on their interiors. We next determine a subcollection of $N-1$ such arcs that satisfy property (2). To start with, set $\alpha_1 = \tilde{\tau_1}$. Since the identification between $\pi^{-1}(t_0)$ and the set $\{1, \dots, N\}$ is arbitrary, we may assume that the local monodromy around $\tau_1$ is the transposition $(12)$. Since $D$ is connected and since the meridians associated to $\tau_1, \dots, \tau_k$ generate $\pi_1(B)$, there is some path $\tau_m$ with local monodromy $(1i)$ or $(2i)$ for some $i \ge 3$. Reordering the paths and adjusting the labeling of $\pi^{-1}(t_0)$, we may assume that the local monodromy around $\tau_2$ is $(23)$, and we set $\alpha_2 = \tilde{\tau_2}$. We proceed in this way: assuming $\alpha_1,\dots,\alpha_i$ connect the points $1, \dots, i$, connectedness of $D$ implies that there is some $\tau_m$ (without loss of generality, $\tau_{i+1}$) with local monodromy $(j\,i+1)$ (possibly after relabeling the points of $\pi^{-1}(t_0)$ assigned to $\ell > i$) for some $j \le i$, and we take $\alpha_{i+1} = \tilde{\tau_{i+1}}$. The system of arcs $\alpha_1, \dots, \alpha_{N-1}$ then have disjoint interiors and connect all points of $\pi^{-1}(t_0)$. An Euler characteristic calculation shows that the union of such $\alpha_i$'s forms an embedded tree in $D$, a regular neighborhood of which is a disk containing $\pi^{-1}(t_0)$ as required.

    We claim that {\em any} ordering of the remaining arcs $p_N, \dots, p_k$ will then satisfy properties (3) and (4). Let $D_i$ denote a regular neighborhood of $\tilde \tau_1, \dots, \tilde \tau_i$. Since the interiors of $\{\tau_j\}$ are pairwise disjoint, the lift $\tilde{\tau_{i+1}}$ is seen to lie outside of $D_i$ except for small segments containing each of the endpoints, establishing (3). For (4), we observe that the complement of a neighborhood of $\tau_1 \cup \dots \cup \tau_k$ on $B$ is a disk containing no branch points, which therefore lifts along $\pi$ to a union of $N$ disjoint disks, showing that $D \setminus D_k$ is a union of disks as required.
\end{proof}

\begin{corollary}\label{corollary:braidmonodromy}
    Let $\bd$ be a multidegree for which $r(\bd) \ge 1$, and let $C,D,E, \tilde C, \tilde D$ be as defined in \Cref{subsection:CDEsetup}. Then the monodromy group $\Gamma_{\bd^+} \le \Mod(E)$ contains a subgroup $\cB_C$ that fixes $\tilde C$ and restricts there to the simple braid group $\tilde{SBr}_N(\tilde C)$ (cf. \Cref{subsection:braidsdisks}), and likewise contains a subgroup $\cB_D$ fixing $\tilde D$ and restricting to $\tilde{SBr}_N(\tilde D)$.
\end{corollary}
\begin{proof}
    The hypothesis $r(\bd) \ge 1$ ensures that $N = \Pi(\bd)$ is at least $3$. The result now follows from \Cref{prop:simplebraidmonodromy}: given some loop in $B \subset \CP^1$ realizing a given simple braid, one obtains a family of reducible curves over $S^1$, and for an appropriate choice of smoothing, this induces a family of smooth curves based at $E$ inducing the same (preimage of a) simple braid on $\tilde C$ or $\tilde D$ as appropriate. The twists about the boundary curves of $\tilde C$ needed to complete a generating set for $\tilde{SBr}_N(\tilde C)$ or $\tilde{SBr}_N(\tilde D)$ (cf. \Cref{subsection:braidsdisks}) are supplied by \Cref{lemma:boundaryVC}.
\end{proof}

\subsection{Lifting configurations}\label{subsection:lifting}
To leverage the inductive hypothesis, we will need to understand how configurations of vanishing cycles on $C$ transfer to $\tilde C$. We accomplish this in this subsection, establishing \Cref{lemma:liftconfig}.

We first establish one piece of terminology. Suppose $S$ is a surface either with nonempty boundary or else with a collection $\bp = (p_1, \dots, p_N)$ of distinguished points, and let $c, c' \subset S$ be simple closed curves. We say that $c$ and $c'$ are in the same {\em unpointed isotopy class} if $c$ and $c'$ are isotopic after filling in all boundary components and forgetting all marked points.
 
\begin{lemma}\label{lemma:simplebraidtrans}
    Let $(S, \bp)$ be a surface with distinguished points $\bp = (p_1, \dots, p_N)$, and let $c, c' \subset S$ be nonseparating simple closed curves in the same unpointed isotopy class. Let $\phi$ be a framing of $S \setminus \bp$ of constant signature $w \ne -1$, and suppose that $\phi(c) = \phi(c')$. Then there is $\beta \in SBr_N(S)$ such that $\beta(c) = c'$. The same result holds when the points $\bp$ are converted into boundary components via the real oriented blowup.
\end{lemma}

\begin{proof}
    Since $c$ and $c'$ are in the same unpointed isotopy class, there is some braid $\beta' \in Br_N(S)$, not necessarily simple, such that $\beta'(c) = c'$. By \Cref{lemma:pushformula}, 
    \[
    \phi(c') = \phi(\beta'(c)) = \phi(c) - (w + 1) \pair{[c],\eta(\beta')},
    \]
    showing that $\pair{[c],\eta(\beta')} = 0$. Extend $c'$ to a set of curves $c', c_2, \dots, c_{2g}$ forming a basis for $H_1(S, \Z)$. Append to $\beta'$ a surface braid $\beta'' \in Br_N(S)$ disjoint from $c'$ and such that $\pair{[c_i], \eta(\beta'')} = - \pair{[c_i],\eta(\beta')}$ for $i \ge 2$. The resulting $\beta = \beta' \beta''$ has $\eta(\beta) = 0$ by construction (and so $\beta \in SBr_N(S)$), and still satisfies $\beta(c) = \beta'(c) = c'$.
    Following the discussion of \Cref{subsection:braidsdisks}, it is clear that this argument goes through unchanged in the setting of boundary components as opposed to punctures.
\end{proof}

\begin{lemma}[Lifting vanishing cycles]\label{lemma:liftVC}
     Let $\alpha \subset \tilde C$ be a nonseparating simple closed curve with $\phi_{C,D}(\alpha) = 0$, and suppose that the image $\bar \alpha \subset C$ is the vanishing cycle for some nodal degeneration of $C$. If $r(\bd) \ge 1$, then $\alpha$ is a vanishing cycle for some nodal degeneration of $E$.
\end{lemma}

\begin{proof}
    Let $\varphi : [0,1]  \to H^0(X, \OO(d_{n-1} +1))$ be the polynodal degeneration of $E$ into $C \cup D$, with $\varphi(1) = E$ and $\varphi(0) = C \cup D$.  Let $\psi: \disk \to H^0(X, \OO(d_{n-1}))$ denote the given nodal degeneration of $C$. For $\epsilon > 0$ sufficiently small, let $\disk_{\epsilon} = \{z \in \C \mid |z| \le \epsilon\}$, and set $S_{\epsilon} = \disk_{\epsilon} \cup [0,1]$. Let $\psi: S_{\epsilon} \to H^0(X, \OO(d_{n-1}))$ denote the given nodal degeneration of $C$ with vanishing cycle $\bar \alpha \subseteq C$. We denote the curve defined by $\psi(z)$ by $C_z$. 
    Let $\psi': S_{\epsilon} \to H^0(X, \OO(d_{n-1}+1))$ denote the map $\psi'(z) = \psi(z) h$, where $h \in H^0(X, \OO(1))$ is the equation for $D$. Then $\psi'(z)$ defines the family of reducible curves $C_z \cup D$.
    
We wish to simultaneously extend the maps $\varphi$ and $\psi'$ to obtain a map 
\[
\Phi: [0,1] \times S_{\epsilon}\to H^0(X, \OO(d_{n-1}+1))
\]

satisfying :

\begin{enumerate}
    \item $\Phi(t, 1)$  is homotopic to $ \varphi(t),$ via a homotopy that fixes the boundary of $[0,1]$ and avoids the discriminant locus except at $0$.
    \item  $\Phi(0, z) = \psi'(z)$,
    \item $\Phi(t,z)$ defines a smooth curve if both $z ,t$ are nonzero,
    \item For $t \neq 0$, $\Phi|_{\{t\} \times S_{\epsilon}}$ defines a  nodal degeneration of $\Phi(t,1).$
\end{enumerate}

There are many ways to construct such a $\Phi$; what follows is one reasonably explicit way of doing so. For some $0< R <1$, we claim that there is a $C^{\infty}$ function $\eta: \disk_\epsilon \to H^0(X, \OO(d_{n-1}+1))$ such that:
\begin{enumerate}[label=(\alph*)]
    \item For $z \neq 0$, $\psi'(z) +t \eta(z) $ is smooth for $t \in (0,R]$,
    \item $\psi'(0) +  t\eta (0)$ has a single node at the nodal point of $C_0$, and $\eta(0)$ is nodal at the nodal point of $C_0$, for $0<t\le R$,
    \item For $t \in (0,R]$ fixed, the family $\psi'(z) +t \eta(z) $ is a nodal degeneration.
\end{enumerate}

To see that such an $\eta$ exists, we first note that (a) and (c) are satisfied if for $z  \neq 0,$ $\eta(z)$ viewed as a tangent vector at the point $\psi'(z)$ lies outside the tangent cone of the discriminant locus. We first pick $\eta_0 \in H^0(X, \OO(d_{n-1} + 1))$ such that $\eta_0$ agrees with $\psi'(0)$ up to first order at $p$ and for  $\psi'(0) +t\eta_0$ lies outside the discriminant if $0<|t| \le R,$ for some choice of $R$. A generic $\eta_0$ agreeing with  $\psi'(0)$ up to first order at $p$ will satisfy our criteria.

 Now at the cost of reducing $\epsilon,$  defining $\eta(z) = \eta_0$  for $z \in \disk_{\epsilon}$ will result in a function $\eta$ satisfying all our conditions for $z \in \disk_{\epsilon}.$ Now we extend $\eta$ to $[0,1]$ as follows. Let $V$ be the set of vectors in $H^0(X, \OO(d_{n-1}+1))$ that lies in $T _{\psi'(t)}\Sigma$ for some $t \in [0,1]$, a subset of real codimension one. If $\eta_0$ does not lie in this set then we simply define $\eta(z) = \eta_0$ for all $z \in S_{\epsilon}.$

 If $\eta_0$ does lie in this set, there is a small ball $B$ centered at $\eta_0$ and a small interval $I =[\epsilon- r ,\epsilon +r],$ such that for $z \in I$ and $\eta' \in B,$ $\eta'$ does not lie in $T_{\Sigma} \psi'(z).$ Now select $\eta_1 \in B \setminus V$. We then define $\eta(z)$ to be  a smoothing of the function taking the constant value $\eta_1$ on $[\epsilon +r,1]$  and linearly interpolating between $\eta_0$ and $\eta_1$ between $[\epsilon, \epsilon +r]$. This $\eta$ satisfies our requirements.

Define $\Phi_0$ on $[0,R] \times S_{\epsilon} \subset [0,1] \times S_{\epsilon}$ by 
\[
\Phi_0(t,z) = \psi'(z) +t \eta(z).
\]

Such $\Phi_0$ satisfies (2) on its domain. We then extend $\Phi_0$ to a function $\Phi_1$ on  $[0,R] \times S_{\epsilon} \cup [0,1] \times \{1\}$  in such a way that (1) holds. To see that this is possible, note that the two arcs $\varphi|_{[0,R]}$ and $\Phi_{0}| _{[0,R] \times \{1\}}$ are homotopic through  a homotopy $H:[0,R] \times [0,1]$ such that $H(0, s)$ is fixed and $H(t,s)$ lies in the discriminant complement if $t \neq 0.$ This is immediate from the fact that $\Sigma$ is of real codimension $2$ in the manifold $H^0(X, \OO(d_{n-1}+1)),$ and any two arcs emanating out from a point in $\Sigma$  can be homotoped in such a way. We then use the homotopy to extend $\Phi_0$ to the required $\Phi_1.$

It is straightforward to construct a deformation retraction
\[
f: [0,1] \times S_\epsilon \to [0,R] \times S_\epsilon \cup [0,1] \times \{1\}
\]
with the property that $f(t,z) = f(R,z)$ for $t \in [R,1]$ and $z \in \disk_{\epsilon/2}$. 

We then define 
\[
\Phi \coloneq \Phi_1 \circ f.
\]

The retraction $f$ is constructed so that $f(t,z) \in [0,1] \times S_\epsilon$ vanishes in one of the coordinates if and only if at least one of $t,z$ itself vanishes. Furthermore, $\Phi|_{[0,R] \times \S_\epsilon \cup [0,1] \times \{1\}} (t,z)$ defines a smooth curve if both $t,z$ are nonzero. These facts show that $\Phi$ as constructed satisfies (3). The construction of $f$ ensures that $\Phi|_{\{t\} \times \disk_{\epsilon/3}} = \Phi|_{R \times \disk_{\epsilon/3}}$ for $t \in [R,1]$. The latter is a nodal degeneration by construction. So $\Phi$ satisfies (4).

Then $\Phi$ extends the nodal degeneration of  $C\cup D$ via the family $C_z \cup D$ to a family of nodal degenerations of smooth curves. In particular $\Phi(t,1)$ defines  a nodal degeneration  of $E$ which we will denote $E'_t$. By construction the vanishing cycle $\alpha'$ associated to $E'_t$ lies in $\tilde C$ under the identification of $E$ with $\tilde C \cup \tilde D$, and furthermore $\alpha'$ is in the unpointed isotopy class of the image $\bar{\alpha} \subset C$.  Consequently, $\phi_{C,D}(\alpha') = \phi_{C,D}(\alpha) = 0$. By \Cref{lemma:tildeDframed}, $\phi_{C,D}$ assigns the value $w = -1 - r(\bd) \ne 0$ to each boundary component of $\tilde C$. From the hypothesis $g(\bd) > 1$, we conclude that $r(\bd) > 0$, and hence $w \ne -1$. The result now follows from \Cref{lemma:simplebraidtrans} and \Cref{corollary:braidmonodromy}.
\end{proof}

\begin{lemma}[Configuration lifting lemma]\label{lemma:liftconfig}
        Let $C,D \subset X$ be smooth curves intersecting transversely with $g(C)>1$, and let $E$ be a smoothing of $C \cup D$.
        Let $C_{t_1}, \dots C_{t_k}$ be a system of nodal degenerations on $C$ with associated vanishing cycles $\alpha_1, \dots, \alpha_k \subset C$ forming an arboreal simple configuration on $C$ with $\{[\alpha_i]\} \subset H_1(C;\Z)$ linearly independent.
        
        Then there is a system $E_{t_1}, \dots, E_{t_k}$ of nodal degenerations of $E$ with vanishing cycles $\{\tilde \alpha_i\}$ all contained in $\tilde{C}$ such that $\tilde \alpha_i$ is in the unpointed isotopy class of $\alpha_i$ for all $i$, and $i(\tilde \alpha_i, \tilde \alpha_j) = i(\alpha_i,\alpha_j)$ for all indices $i,j$.
    \end{lemma}

\begin{proof}
    Let $\tilde \alpha_1', \dots, \tilde \alpha_k'$ be a lifting of the configuration $\{\alpha_i\}$ to $\tilde C$, with the same intersection pattern $i(\alpha_i',\alpha_j') = i(\alpha_i,\alpha_j)$, but with $\phi_{C,D}(\alpha_i')$ unconstrained. 
    Since $\phi_{C,D}$ descends to the $r(\bd)$-spin structure on $C$ associated to $\cO(1)$ and each $\alpha_i$ is a vanishing cycle on $C$, in fact $\phi_{C,D}(\tilde \alpha_i) \equiv 0 \pmod{r(\bd)}$.

    Since $\{\alpha_i\}$ is an arboreal simple configuration, so is $\{\tilde \alpha_i'\}$, and $\{[\alpha]\} \subset H_1(C;\Z)$ is linearly independent by hypothesis.
    By \Cref{lemma:dualexists}, there exists a dual configuration $\{\beta_i\} \subset \tilde C$ based at some $* \in \partial \tilde C$.
    Recall from \Cref{lemma:tildeDframed} that $\tilde C$ has constant signature $-r(\bd) - 1$.
    According to \Cref{lemma:pushaffectsphi}, applying the push map $P_{\beta_i}$ to the configuration $\{\tilde \alpha_i'\}$ changes $\phi_{C,D}(\tilde \alpha_i'\}$ by $\pm r(\bd)$ and leaves the winding numbers of the remaining $\tilde \alpha_j'$ unchanged. 
    Thus by applying the appropriate power of each $P_{\beta_i}$ in some arbitrary order, the configuration $\{\tilde \alpha_i'\}$ is sent to a configuration $\{\tilde \alpha_i\}$ with the same intersection pattern but with each $\tilde \alpha_i$ admissible. 
    By \Cref{lemma:liftVC}, each $\tilde \alpha_i$ is a vanishing cycle.
\end{proof}

\section{Tacnodal degenerations}\label{section:tacnode}
In this section, we establish the ``Tacnode construction lemma'' (\Cref{lemma:tacnodelocal}), which provides us with a large supply of vanishing cycles that arise from deforming $C \cup D$ to a curve with a single tacnode singularity inherited from a point of simple tangency between $C$ and $D$. 

We continue with the working environment of the previous section. 
Choose an arbitrary ordering $p_1, \dots, p_N$ of the points $C \cap D$, and let $\beta_1, \dots, \beta_N \subset E$ denote the corresponding vanishing cycles on $E$, i.e. the boundary components of $\tilde C, \tilde D$.

We adopt the same conventions about the meaning of ``path'' as in the proof of \Cref{prop:simplebraidmonodromy}. Let $\tau \subset \CP^1$ be such a path, connecting the basepoint $t_0$ to some simple branch point $t_i$ for the branched covering $D \to \CP^1$ induced by the pencil $C_t$. 
Thus the preimage of $\tau$ in $D$ has a distinguished component covering its image two-to-one (all other components cover with degree $1$). We define the {\em distinguished lift} of $\tau$ to $\tilde D$ as the proper transform of this distinguished component under the real oriented blowup map $\tilde D \to D$. See the bottom row of \Cref{fig:tacnodes}.

\begin{lemma}[Tacnode construction]\label{lemma:tacnodelocal}
    In the above setting, let $\tau \subset \CP^1$ be a path connecting $t_0$ to some point $t_i$ for which $C_{t_i}$ has a simple tangency with $D$. 
    Then $\tau$ induces a degeneration $E_s$ of $E$ to a tacnodal curve for which there is a vanishing cycle $a$ such that $a \cap \tilde D$ is a single arc given as the distinguished lift of $\tau$ to $\tilde D$.
\end{lemma}

\para{Local behavior} To prove \Cref{lemma:tacnodelocal}, we first analyze the local situation. Without loss of generality, the tangency occurs for $t_i = 0$. In a neighborhood of the tangency $q$ between $C_0$ and $D$, there are analytic local coordinates $(z,w)$ on $X$ for which $D = w$ and $C_t = w-z^2 + t$. We define the $2$-parameter family of curves on a neighborhood of $q$
\[
E^{loc}_{s,t} = Z(w(w-z^2+t) + s)
\]
and take $E^{loc}_{s_0,t_0}$ for $s_0,t_0$ positive, real, and suitably small, as basepoint. For later use it will be convenient to further assume $t_0 > 2 \sqrt{s_0}$. This is depicted in the top row of \Cref{fig:tacnodes}. 

\begin{figure}
\centering
		\labellist
        \small
        \pinlabel $a$ at 160 850
        \pinlabel $a$ at 460 600
        \pinlabel $\tilde{C_{t_0}}$ at 360 540
        \pinlabel $\tilde{D}$ at 780 540
        \pinlabel $\beta_i$ at 560 660
        \pinlabel $\beta_j$ at 560 495
        \pinlabel $\tau$ at 835 240
        \pinlabel $\CP^1$ at 900 150
        \pinlabel $\tilde{D}$ at 250 220
        \pinlabel $a$ at 340 110
        \pinlabel $a$ at 340 350
        \pinlabel $\pi$ at 560 250
        \tiny  
        \pinlabel $0$ at 790 210
        \pinlabel $t_0$ at 860 210
        \pinlabel $\beta_i$ at 310 320
        \pinlabel $\beta_j$ at 390 320
        \pinlabel $\beta_i$ at 310 80
        \pinlabel $\beta_j$ at 390 80
		\endlabellist
\includegraphics[width=\textwidth]{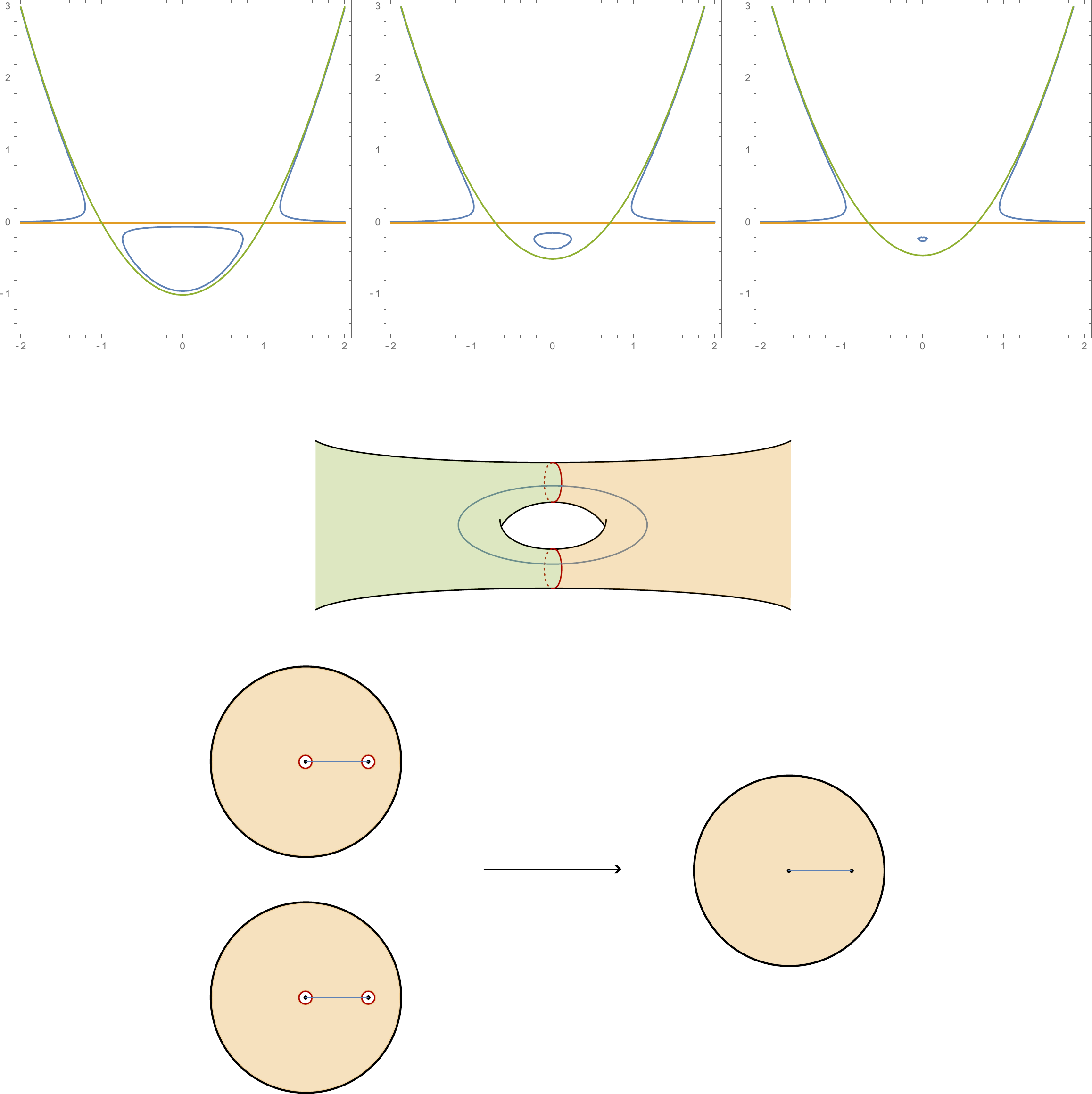}
\caption{Three points of view on the tacnode singularity. Top: A vanishing cycle associated to the tacnodal singularity $w(w-z^2)$. In the plot, $s_0 = 0.05$, making the critical fiber occur for $t_0 = 2 \sqrt{s_0} \approx 0.447$. The figure shows the real points of $E^{loc}_{s_0,t}$ for $t = 1, 0.5, 0.451$ from left to right. $D$ appears as the line (orange), and $C_t$ is the parabola (green) of varying heights. Middle: a topological picture of $E^{loc}_{s_0,t_0}$. The left (green) half is $\tilde C_{t_0}$, the right is $\tilde D$, and these are joined along the curves $\beta_i, \beta_j$. $a$ appears as the curve shown in the middle. Bottom: representing $\tilde D$ as a branched covering of $\CP^1$ via the pencil map for $C_t$. Observe that the path $\tau \subset \CP^1$ has distinguished lift $a \cap \tilde D$.}
\label{fig:tacnodes}
\end{figure}
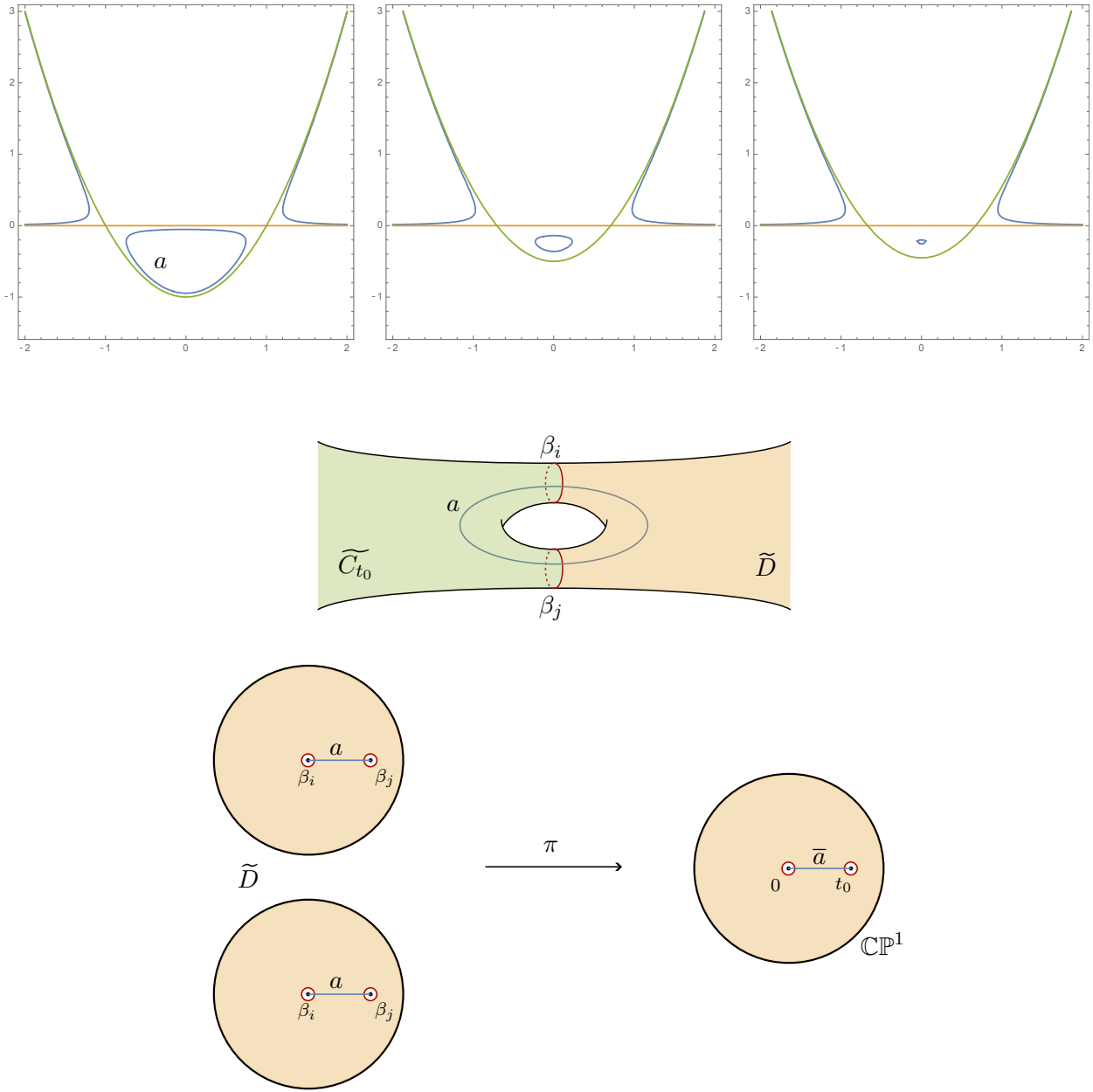

\begin{lemma}\label{lemma:tacnodeverylocal}
    The curve $a$ in \Cref{fig:tacnodes} is a vanishing cycle. Moreover, under the projection $\pi: X \dashedrightarrow \CP^1$ given by the pencil $C_t$, the intersection $a \cap \tilde D$ is isotopic to the distinguished lift of the straight line segment connecting $t_0$ to $0$.
\end{lemma}

\begin{proof}
    For $s_0$ fixed, the family of curves $E^{loc}_{s_0,t}$ centered at $t = 2 \sqrt{s_0}$ is a nodal degeneration, and it is easy to see that as $t \to 2 \sqrt{s_0}$, the vanishing cycle is $a$. In the limit $s_0 \to 0$, the curve $a$ decomposes into a union of two arcs, one on $C_{t_0}$ and one on $D$. In these coordinates, the pencil map $\pi$ sends $(z,0) \in D$ to $z^2$, so that the image of the segment on $D$ is given as the straight line segment along the real axis from $t_0$ to $0$.
\end{proof}

\begin{proof}[Proof of \Cref{lemma:tacnodelocal}]
    \Cref{lemma:tacnodeverylocal} shows that tacnodal degenerations locally have the structure asserted in \Cref{lemma:tacnodelocal}. To complete the proof, it is necessary to exhibit a family $E_{s,t}$ of curves in $H^0(X; \cO(d_{n-1}+1))$ that is analytically equivalent to $E^{loc}_{s,t}$ near the point of tangency $q$ and smooth away from $q$; we must also understand the effect of changing basepoints.

     A choice of trivialization of $\cO(1)$ in a neighborhood of $q$ gives an identification of sections of any $\cO(d)$ with function germs. It will be convenient to work with coordinates giving a different local model of the tacnode, as represented by the function germ $w^2 - z^4$. In these coordinates, $D$ is represented by $f(z,w) = w-z^2$ and $C_0$ by $g(z,w) = w + z^2$. We can identify the completed local ring of $X$ at $q$ with $\C[[z,w]]$, and under this identification, the Jacobian ideal $J$ is generated by $w$ and $z^3$. Thus the versal deformation space $\C[[z,w]]/J$ is isomorphic to $\C[z]/(z^3)$. 
    
    We will exhibit global sections of $\cO(d_{n-1}+1)$ whose restriction to a neighborhood of $q$ give the required germs, and which do not vanish near any of the other points of intersection $C_0 \cap D$. To that end, let $h_1 \in H^0(\CP^n; \cO(d_{n-1}))$ be nonvanishing at all points of $C_0 \cap D$ (including, of course, $q$). Such $h_1$ exists because $\cO(d_{n-1})$ is very ample. For this same reason, there exists $h_2 \in H^0(\CP^n; \cO(d_{n-1}+1))$ with $h_2(q) = 0$ and $\partial_z h_2(q) \ne 0$, as well as $h_3 \in H^0(\CP^n; \cO(d_{n-1}+1))$ with $h_3(q) \ne 0$; moreover such $h_2$ and $h_3$ can be selected so as to be nonvanishing at the nodes of $C_0 \cup D$.

    We claim that the sections $fh_1, h_2, h_3$ generate $\C[[z,w]]/J$ as a vector space. To see this, note that $f \equiv z^2 \pmod J$, and the $\C$-linear span of $z^2$ in $\C[[z,w]]/J$ equals the ideal $(z^2)$. Furthermore, $h_1$ is a unit in $\C[[z,w]]/ J$ and $fh_1 = a f$ for some $a \in \C^*$. Thus to establish the proposition it suffices to establish that $h_2, h_3$ generate $\C[z,w]/ \pair{J, z^2} \cong \C[z]/ (z^2)$. But by hypothesis, $h_2 \equiv bz \pmod J$ in $\C[z]/(z^2)$ for some $b \in \C^*$, and thus it suffices to establish that $h_3$ generates $\C[z]/(z)$, which is equivalent to $h_3$ not vanishing at $q$, which holds by assumption.

    From the claim, we see that the versal deformation space of the tacnode embeds into $H^0(X;\cO(d_{n-1}+1))$, and since $h_2, h_3$ were constructed so as to be nonvanishing at the nodal points of $C_0 \cup D$, a sufficiently small deformation of $C_0 \cup D$ realizing the nodal degeneration with vanishing cycle $a$ will be globally smooth.

    We have seen that when the basepoint $t_0 \in \CP^1$ is close to a point of tangency between $D$ and the pencil $C_t$, there is a vanishing cycle $a \subset E$ for which $a \cap \tilde D$ is the distinguished lift of a path on $\CP^1$. All that remains to be done is analyze the general situation when the basepoint is  arbitrary. Recall from \Cref{subsection:CDEsetup} that the basepoint curve $E$ is obtained from $C_{t_0} \cup D$ by a small perturbation, realizing $E$ topologically as $\tilde C_{t_0} \cup \tilde D$, with $\tilde C_{t_0}$, resp. $\tilde D$, given as the real oriented blowup of $C_{t_0}$, resp. $D$, at $C_{t_0} \cap D$. 
    
    The pencil map $\pi: X \dashedrightarrow \CP^1$ gives a representation of $D$ for which $C_t \cap D$ is given as the fiber $\pi^{-1}(t)$ for any $t \in \CP^1$. In this model, $\tilde D$ is realized as the $\pi$-preimage of the complement of a neighborhood of $t$. Thus, as $t$ traces along the path $\tau \subset \CP^1$, there is a canonical isotopy class of trivialization of the subsurfaces $\tilde D$ obtained by smoothing the various $C_t \cup D$, given by lifting the disk-pushing map along $\tau$ . When $t_0'$ is sufficiently close to the tangency point, the local analysis performed above shows that the vanishing cycle intersects $\tilde D$ as the distinguished lift of the path from $t_0'$ to the tangency. Parallel transporting back from $t_0'$ to the original basepoint $t_0$ along the specified trivialization over $\tau$ takes the distinguished lift of this short path to the distinguished lift of $\tau$ itself, as required. 
\end{proof}

For later convenience, we package together a collection of ``sufficiently many'' tacnodal degenerations. 

    \begin{lemma}\label{lemma:lotsoftacnodeVCs}
        In the above setting, there is a set $E_{s_1}, \dots,  E_{s_m}$ of degenerations of $E$ to tacnodal curves via the construction of \Cref{lemma:tacnodelocal}, such that the distinguished vanishing cycles $a_1, \dots, a_m$ admit representatives depositing pairwise disjoint arcs on $\tilde D$, and the complement $\tilde D \setminus \{a_1, \dots, a_m\}$ is a union of disks. 
    \end{lemma}

\begin{proof}
    Let $\tau \subset \CP^1$ be a path.
    In light of \Cref{lemma:tacnodelocal}, to prove the claim, it suffices to exhibit paths $\tau_1, \dots, \tau_m \subset \CP^1$ that are pairwise disjoint on their interiors, such that the distinguished lifts of $\{\tau_i\}$ decompose $\tilde D$ into a union of disks. 
    This follows {\em exactly} as in the proof of \Cref{prop:simplebraidmonodromy}, taking $\{\tau_i\}$ to be some set of mutually disjoint paths from $t_0$ to each of the branch points.
\end{proof}

\section{Proof of \Cref{mainthm:ambientX}, base cases}\label{section:basecase}

Recall that \Cref{lemma:smallgenustable} gives a complete enumeration of those multidegrees $\bd$ for which $r(\bd) \le 1$. Here we study the monodromy of such families as the base case in our inductive argument.

\begin{proposition}\label{prop:basecases}
    For any multidegree $\bd = (d_1, \dots, d_{n-1})$ such that $r(\bd) \le 1$ and any smooth complete intersection surface $X$ of multidegree $(d_1, \dots, d_{n-2})$, the topological monodromy group $\Gamma_{X,\bd}$ is computed as
    \[
    \Gamma_{X,\bd} = \Mod(\Sigma_{g(\bd)}).
    \]
\end{proposition}

We first show that \Cref{prop:basecases} follows from an apparently weaker statement, where we consider all smooth complete intersection curves, not just those embedded in a given $X$.

\begin{lemma}\label{lemma:bigsmall}
    With $X, \bd$ as in \Cref{prop:basecases}, let $\Gamma_\bd$ denote the topological monodromy group of the family $U_\bd$ of all smooth complete intersection curves of multidegree $\bd$. Then there is an equality
    \[
    \Gamma_{X,\bd} = \Gamma_\bd.
    \]
\end{lemma}
\begin{proof}
    The base space of the family $U_\bd$ of smooth complete intersection curves of multidegree $\bd$ is defined as the complement of the discriminant divisor $\Sigma_{\bd}$ in the product space
\[
V_\bd := \prod_{i =1}^{n-1} H^0(\CP^n; \cO(d_i)).
\]

Let $X$ be a smooth complete intersection surface of multidegree $(d_1, \dots, d_{n-2})$ defined by 
\[
X = Z(f_1, \dots, f_{n-2}).
\]
Then a pencil of curves in $X$ of multidegree $\bd$ can be constructed by taking $sf_{n-1}+tg_{n-1}$ for $s,t \in \C$ and $f_{n-1}, g_{n-1} \in H^0(\CP^n;\cO(d_{n-1})$; this determines an affine subspace $L \subset V_\bd$ of dimension $2$. As is well-known, $f_{n-1}$ and $g_{n-1}$ can be chosen so that $L$ intersects $\Sigma_\bd$ transversely. By the Lefschetz hyperplane theorem, $\pi_1(L \setminus \Sigma_\bd)$ then surjects onto $\pi_1(V_\bd \setminus \Sigma_\bd)$, showing $\Gamma_\bd = \Gamma_{X,\bd}$.
\end{proof}

\Cref{prop:basecases} thus will follow from \Cref{prop:basecasesbig} stated below.

\begin{proposition}\label{prop:basecasesbig}
    For any multidegree $\bd$ such that $r(\bd) \le 1$, there is an equality
    \[
    \Gamma_\bd = \Mod(\Sigma_{g(\bd)}).
    \]
\end{proposition}

\begin{proof}[Proof of \Cref{prop:basecasesbig}]
As discussed in the introduction, the cases of $r(\bd) < 0$ are trivial, as the associated curve is of genus 0 and has trivial mapping class group. Consider next the the cases when $r(\bd) =0$; by \Cref{lemma:smallgenustable}, these are the cases $\bd = (3)$ and $\bd = (2,2)$. In both of these cases the corresponding curves are of genus $1$ and the mapping class group is simply $\SL_2 (\Z)$.
\begin{proposition}
    If $\bd = (3)$ or $(2,2)$, then $\Gamma_{\bd} = \SL_2(\Z).$
\end{proposition}
\begin{proof}
This is a result of Dolgachev--Libgober - see \cite[Section 4]{DolgLib}.
\end{proof}

Finally we consider the cases where $r(\bd) = 1$. This will follow immediately from the next two statements (\Cref{result1,result2}).

\begin{lemma}\label{result1}
    Let $U \subseteq \cM_g$ be a Zariski-open suborbifold. Then the induced map $\pi_1^{orb}(U) \to \pi_1^{orb}(\cM_g)$ is surjective.
\end{lemma}
\begin{proof}
    Let $\tilde\cM$ denote a finite connected $G$ cover of $\cM_g$ that is a variety. Let $\tilde U$ denote the pullback of this cover to $U$. We note that $\tilde U$ is connected and that $\pi_1(\tilde U) \to \pi_1(\tilde \cM)$ is surjective, since $\tilde U$ is Zariski-open in $\tilde \cM$. We now have the following diagram where both rows are exact:

\[\xymatrix{
1 \ar[r] & \pi_1(\tilde U) \ar[r] \ar[d]  & \pi_1^{orb}(U) \ar[r] \ar[d]  & G \ar[r] \ar[d] & 1\\
1 \ar[r] & \pi_1(\tilde \cM) \ar[r]       & \pi_1^{orb}(\cM_g) \ar[r]      & G \ar[r] & 1.
}\]
 Since the first and third columns are surjective, so is the middle one.
\end{proof}

\begin{lemma}\label{result2}
    Let $\bd$ be one of $4, (3,2)$ and $(2,2,2)$. Let $\rho: U_{\bd} \to \cM_g$ denote the natural map defining the family of complete intersection curves where $g = 3,4,5$ respectively. Let $\cU$ denote the image of $\rho$. Then $\cU$ contains a Zariski-open set and $\rho_*: \pi_1(U_{\bd}) \to \pi_1^{orb}(\cU)$ is surjective. 
\end{lemma}
\begin{proof}
    The fact that for these particular multidegrees, the image $\cU$ is Zariski open is classical. We also note that it is well known that curves arising as complete intersections of these particular multidegrees are \emph{canonically} embedded in $\P^n$. Let $Y$ denote the parameter space of canonically embedded smooth curves in $\P^{g-1}$. This is known to be a smooth projective variety that embeds into the Hilbert scheme of curves in $\P^{g-1}$.
    
    To understand the map $\rho_*$ we will factorize $\rho$ into two fiber bundles with connected fibers (in the case when $\bd =4$ the first fiber bundle will be trivial). We note that a curve defined by an element of $U_{\bd}$ naturally lies in $\P^{g-1}$, and we have a natural map $\phi: U_{\bd} \to Y.$ Let $Y'$ denote the image of $\phi$.
    
      We also have a quotient map $\pi: Y' \to \cM_g$ whose image is $\cU$. Clearly $\rho = \pi \circ \phi.$
    Let us first understand  the map $\pi.$ It is well known that a general curve of genus $3,4,5$ 
    has a canonical model that is a complete intersection of multidegree $\bd$, and furthermore, this model is unique up to automorphisms of the ambient projective space. As a result one has that the map $\pi: Y' \to \cU$ is a quotient map under the action of $\PGL_{g}(\C)$ and that $\cU$ is the orbifold quotient $Y' / \PGL_{g}(\C).$ Orbifold quotients satisfy the associated long exact sequence in homotopy groups and so there is an exact sequence 
    \[
    \pi_1(Y') \to \pi_1^{orb}(\cU) \to \pi_0(\PGL_{g}(\C)).
    \]
    Since $\PGL_{g}(\C)$ is connected, the map $\pi_1(Y') \to \pi_1^{orb}(\cU)$ is surjective.

    We now consider the map $\phi: U_{\bd} \to Y'$. Our arguments will differ slightly based on what the value of $\bd$ is.
   If $\bd = 4$, the map $\phi$ is a quotient map by $\C^*$, since a plane curve in $\P^2$ is uniquely determined by its equation up to scaling. This makes the map $\phi$ a $\C^*$ bundle.
    
    Now let $\bd =(3,2).$ Let $(f,g) \in U_{\bd}$, and let $C$ be the common zero locus of $f$ and $g$. Define 
    \[
    U' = \{(f,g)\in H^0(\P^3, \O(3)) \times H^0(\P^3, \O(2)) \mid g\neq 0, f \not \in g H^0(\P^3, \O(1)) \}.
    \]
    Also set
    \[
    Y'' = \{(V_1, V_2) \in Gr_5 (H^0(\P^3, \O(3))) \times \P H^0(\P^3, \O(2)) \mid V_2 H^0(\P^3, \O(1)) \subseteq V_1\}.
    \]
    Let $\bar \phi : U' \to Y''$ be defined by 
    \[
    \bar \phi(f,g) = ( \Span f + g H^0(\P^3, \O(1)),\Span g).
    \]
    Then $\bar \phi$ is a fiber bundle with fiber $\C^* \times \C^*  \times \C^4$. 
    
     There is an injective map $i: Y' \to Y''$ defined by $i(C) = (H^0(\P^3, I_C(3)),H^0(\P^3, I_C(2)))$. 
   We now have a Cartesian square

 \[
    \xymatrix{
    U \ar[r]\ar[d] & U' \ar[d]\\
    Y' \ar[r] & Y''
    }
    \]
   Thus $\phi$ is a fiber bundle with connected fibers.

    Finally, let $\bd =(2,2,2)$. Let $(f_1, f_2, f_3) \in U_{\bd}$, and let $C = Z(f_1) \cap Z(f_2) \cap Z(f_3)$. Let $W = H^0(\P^4,\O(2) )$, and define
    \[
    \bar U = \{(f,g,h) \in W^3 \mid (f,g,h) \textrm{ are linearly independent}\}.
    \]
    Let $\bar \phi : \bar U \to Gr_3(W)$ be defined by $\bar \phi(f,g,h) = \Span {f,g,h}$. Clearly $\bar \phi$ is a fiber bundle with fibers $\GL_{3}(\C)$. Furthermore, $Y'$ can be identified with a subspace of $Gr_3( W)$, via the map 
    \begin{align*}
    i : Y' &\to Gr_3(W)\\
    C &\mapsto H^0(\P^4, I_C(2)).
    \end{align*}
    We now have a Cartesian diagram:
    \[
    \xymatrix{
    U \ar[r]\ar[d] & \bar U \ar[d]\\
    Y' \ar[r] & Gr_3(W)
    }
    \]

    Thus the map $\phi$ is a fiber bundle with connected fibers in all cases- therefore $\phi$ induces a surjective map at the level of $\pi_1$ and hence $\rho$ does as well.
\end{proof}
This completes the proof of \Cref{prop:basecasesbig}.
\end{proof}

\section{Proof of main theorem, inductive step}\label{section:inductionstep}
We formulate the following inductive hypothesis:\\

IH($\bd$): {\em For any smooth complete intersection surface $X$ of multidegree $(d_1, \dots, d_{n-2})$, the equality $\Gamma_{X,\bd} = \Mod(\Sigma_{g(\bd)})[\phi_{\bd}]$ of \Cref{mainthm:ambientX} holds for the multidegree $\bd$.\\}

As usual, if $n = 2$, then $X = \CP^2$ by convention. If $\bd$ is any multidegree for which $r(\bd^+) \ge 2$, then $r(\bd) = r(\bd^+)-1 \ge 1$.
Also note that IH($d_1, \dots, d_{n-1}$) trivially implies IH($d_1, \dots, d_{n-1},1)$.
Taken along with the base cases established in \Cref{prop:basecases}, these observations show that to establish \Cref{mainthm:ambientX} (and hence \Cref{mainthm:ambientPn}), it suffices to establish IH($\bd^+$) assuming IH($\bd$) under the additional assumption that $r(\bd) \ge 1$, or equivalently that $g(\bd) \ge 3$.

\para{Proof outline} Suppose that IH($\bd$) holds. Let $D \subset X$ be a smooth complete intersection curve of multidegree $\bd' = (d_1, \dots, d_{n-2}, 1)$. By \Cref{lemma:maxgenexists}, there is a pencil $C_t$ of curves in $X$ of multidegree $\bd$ that is maximally generic rel $D$; let $C := C_{t_0}$ be a smooth member of the pencil. As usual, we take $E$ to be a smoothing of $C \cup D$, chosen in accordance with the basepoint conventions of \Cref{convention:basepoint}.

Ultimately we will establish IH($\bd^+$) by constructing an assemblage (recall \Cref{def:assemblage}) of vanishing cycles. Using the results of \Cref{section:genspinmcg}, the associated Dehn twists will generate a {\em framed} mapping class group on the subsurface with boundary $E^\circ \subset E$ determined by the assemblage. To descend to an $r$-spin mapping class group, we appeal to \Cref{lemma:framedontospin}. 

The argument proceeds in four steps. In Step 1, we will begin building our assemblage by finding a suitable configuration of vanishing cycles on a subsurface $\tilde C' \subset \tilde C$ of genus $g(C)$ but with one boundary component. In Step 2, we will extend the assemblage to encompass all of $\tilde C$. In Step 3, we use the technique of tacnodal degeneration (\Cref{lemma:lotsoftacnodeVCs}) to extend the assemblage onto $\tilde D$. At this stage, we will have constructed a framed subsurface $E^\circ \subset E$ with $E \setminus E^\circ$ a union of disks, and exhibited sufficient vanishing cycles supported on $E^\circ$ to generate its framed mapping class group. In Step 4, we show that the inclusion $E^\circ \into E$ induces a surjection onto the $r$-spin mapping class group of $E$ induced from $\cO(1)$.

\para{Step 1: Lifting from $C$} By IH($\bd$) and \Cref{lemma:Earbexists}, since $g:=g(C) \ge 3$ by assumption, there is an $E$-arboreal configuration $\bar{\cC_C} = \{\bar{c_1}, \dots, \bar{c_{2g}}\}$ of homologically-independent simple closed curves on $C$ that are all admissible with respect to the $r(\bd)$-spin structure $\phi_{\bd}$ on $C$. This takes one of the two forms shown in \Cref{fig:liftingvcs}, depending on the parity of an Arf invariant, should one be present. Thus each $\bar{c_i}$ is the vanishing cycle for some nodal degeneration $C_{t_i}$ of $C$. Again since we assume $g(C) \ge 3$, by \Cref{lemma:liftconfig}, this configuration lifts to a configuration $\cC_C = \{c_1, \dots, c_{2g}\}$ of vanishing cycles on $E$. 

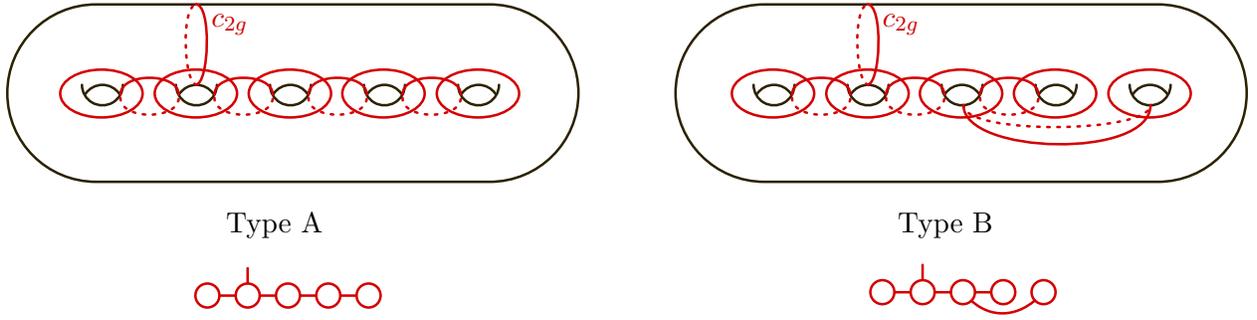
\begin{figure}
\centering
		\labellist
        \small 
        \pinlabel $\red{c_{2g}}$ at 100 130
        \pinlabel $\red{c_{2g}}$ at 400 130
        \pinlabel {Type A} at 120 40
        \pinlabel {Type B} at 420 40
		\endlabellist
\includegraphics[width=\textwidth]{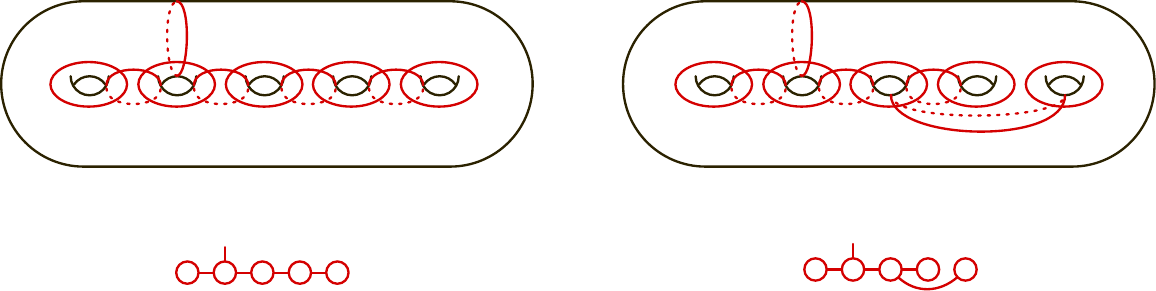}
\caption{The two configurations of vanishing cycles on $C$ invoked in Step 1 (top row), and a more efficient schematic rendering (bottom row), depicted in the case $g = 5$. A type A configuration can be found on $C$ whenever $r(\bd)$ is odd, or for one value of the Arf invariant when $r(\bd)$ is even; Type B is present when $r(\bd)$ is even for the other value of Arf. With reference to the diagrams at bottom, these differ only in whether the rightmost circle is attached to the circle to its immediate left, or two over.}
\label{fig:liftingvcs}
\end{figure}

\para{Step 2: Extending to $\tilde{C}$} 
Recall that by \Cref{lemma:tildeDframed}, $D$ endows $\tilde C$ with a framing $\phi_{C,D}$ that descends to the $r(\bd)$-spin structure $\phi_{\bd}$ on $C$. 
As before, we denote the boundary components of $\tilde C$ by $\beta_1, \dots, \beta_N$, noting that $N = \Pi(\bd) = d_1\dots d_{n-1}$.

Suppose that $a_{1}, \dots, a_{r+2} \subset \tilde C$ are pairwise disjoint simple closed curves such that $a_1 \cup \dots \cup a_{r+2} \cup \beta_j$ bounds an $r(\bd)+3$-holed sphere, and such that $\phi(a_1) = \dots = \phi(a_{r+1}) = 0$. 
By \Cref{lemma:tildeDframed}, $\phi_{C,D}(\beta_j) = -r(\bd)-1$. 
Homological coherence then implies that $\phi(a_{r+2}) = 0$ as well. 

We apply this reasoning to the curves $c_{2g+1}, \dots, c_{2g+N-1} \subset \tilde C$ described as follows. Represent the configuration $\cC_C$ schematically as in the bottom row of \Cref{fig:liftingvcs}. 
There are $g(\bd) = \tfrac{1}{2} \Pi(\bd) r(\bd) + 1$ circles; number these from left to right as $0, 1, \dots, g(\bd)-1$. 
For notational convenience, let $c_{2g} \in \cC_C$ be the exceptional curve depicted as the segment coming out of the top of circle $1$ (again see \Cref{fig:liftingvcs}). 
Define $c_{2g+1}$ to intersect only $c_{r(\bd)+1} \in \cC_C$, and such that $c_{2g}, c_{2g+1}, \beta_1$, and the $r(\bd)$ curves in $\cC_C$ connecting circles $1, \dots, r(\bd)+1$ together bound an $r(\bd)+3$-holed sphere.
Schematically, this is indicated by a segment attached to the top of circle $r(\bd)+1$.
By the discussion above, $\phi(c_{2g+1})$ is admissible.

Continue in this way, defining $c_{2g+k}$ schematically as attached to the top of circle $kr(\bd)+1$, so long as $kr(\bd)+1 \le g(\bd) - 1$. 
Then $c_{2g+k-1}, c_{2g+k}, \beta_k$, and $r(\bd)$ curves in $\cC_C$ together bound an $r(\bd)+3$-holed sphere, implying by the above discussion that each $c_{2g+k}$ constructed in this way is admissible. 
The remaining curves are constructed by attaching along the {\em bottoms} of circles $kr(\bd) - 1$, for $1 \le k \le \tfrac1 2 \Pi(\bd)$. 
Special care must be taken to analyze what happens in the case $r(\bd) = 2$; there is an exceptional case when $g(\bd)$ is odd in type B that requires a modified construction shown in \Cref{fig:specialcaser2B}.
We leave it to the interested reader to assure themselves that there are no further unexpected subtleties here.

\begin{figure}
\centering
		\labellist
        \small 
		\endlabellist
\includegraphics[width=\textwidth]{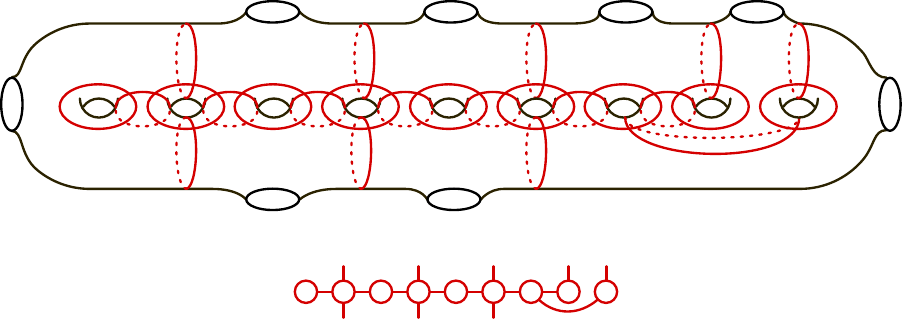}
\caption{When $r(\bd) = 2, g(\bd)$ is odd, and the configuration is of type B, a special modification has to be made to the construction of the curves $c_{2g+1}, \dots c_{2g+N-1}$, as shown here for $\bd = (4,2)$.}
\label{fig:specialcaser2B}
\end{figure}

The curves $c_{2g+k}$ of this latter sort are likewise admissible. For the most part, the same logic applies, although at the left (and sometimes right) ends of the surface, one obtains a surface of genus one bounded by some $\beta_j$, some $c_{2g+k}$ not known to be admissible, and $r(\bd)-1$ curves in $\cC_C$, giving an Euler characteristic of $-r(\bd)-1$. Arguing as in the case of an $r(\bd)+3$-holed sphere, it again follows that $c_{2g+k}$ is admissible. 

Though there will be variable numbers of curves of the first and second sort, depending on the parities of $\Pi(\bd)$ and $r(\bd)$, an analysis of the possibilities shows that this always produces $N-1 = \Pi(\bd) - 1$ additional admissible curves as claimed.
Their images $\bar{c_{2g+1}}, \dots, \bar{c_{2g+N-1}}$ are likewise admissible on $C$. 
By the inductive hypothesis IH($\bd$), they are moreover vanishing cycles on $C$. 
By \Cref{lemma:liftVC}, it follows that $c_{2g+1}, \dots, c_{2g+N-1}$ are vanishing cycles on $\tilde C$ as required.

\para{Step 3: Extending across $\tilde D$} We now turn our attention to $\tilde D$. Using the maximally generic pencil $C_t$ given by \Cref{lemma:maxgenexists}, we represent $D$ as a branched cover of $\CP^1$. Exactly as in the proof of \Cref{prop:simplebraidmonodromy}, we exhaust $\tilde D$ with a sequence $\alpha_1, \dots, \alpha_m$ of distinguished arcs lifted from pairwise-disjoint paths $\tau_i$ on $\CP^1$ from the basepoint to each of the branch points. By \Cref{lemma:lotsoftacnodeVCs}, each such path $\tau_i$ gives rise to a tacnodal degeneration of $E$ for which one of the vanishing cycles is a curve $a_i \subset E$, with $a_i \cap \tilde D$ in the isotopy class of $\alpha_i$. 

As in the proof of \Cref{prop:simplebraidmonodromy}, adding $a_{i+1}$ to the subsurface assembled from $\tilde C \cup a_1 \cup \dots \cup a_i$ decreases the Euler characteristic by $1$. As the configuration $\{c_1, \dots, c_{2g+N-1}\}$ constructed in Steps 1 and 2 is $E$-arboreal, it follows that $\{c_1,\dots,c_{2g+N-1}, a_1, \dots, a_m\}$ is an $h$-assemblage of type $E$ of genus $h = g(C)$. Let the subsurface assembled from these curves be denoted $E^\circ$; note that $E \setminus E^\circ$ is a union of disks. By \Cref{lemma:framingextend}, the natural framing $\phi_{C,D}$ of $\tilde C$ induced by $D$ extends to a framing $\phi^+$ of $E^\circ$ by the stipulation that each $a_i$ be admissible. 

We first consider the case $g(C) \ge 5$. In this setting, we apply \Cref{theorem:assemblagegenset} to see that the twists about $c_i$ and $a_j$ generate the framed mapping class group $\Mod(E^\circ)[\phi^+]$. 
In the case $g(C) \le 4$, we will appeal to \Cref{prop:gencriterion}. According to \Cref{lemma:smallgenustable}, there are exactly two possibilities for $\bd$ with $3 \le g(\bd) \le 4$: either $\bd = 4$ or $\bd = (3,2)$. In each of these cases, $r(\bd) = 1$, so that by \Cref{lemma:tildeDframed}, $\tilde C$ has constant signature $-2$. Also note that by construction, the curves $a_1, \dots, a_m$ enter and exit $\tilde C$ exactly once. 

In order to apply \Cref{prop:gencriterion}, it remains to see that $\Gamma_{\bd^+}$ contains the admissible subgroup $\cT_{\tilde C}$. Let $a \subset \tilde C$ be admissible for $\phi_{C,D}$. By \Cref{prop:basecases}, the monodromy $\Gamma_{\bd}$ for $C$ is the full mapping class group $\Mod(C)$, so that every nonseparating simple closed curve is a vanishing cycle. Then by \Cref{lemma:liftVC}, $a$ is a vanishing cycle, and hence $T_a \in \Gamma_{\bd^+}$.

\para{Step 4: From $\mathbf{E^\circ}$ to $\mathbf{E}$} By \Cref{lemma:framedontospin}, the inclusion $E^\circ \into E$ induces a surjection
\[
\Mod(E^\circ)[\phi^+] \onto \Mod(E)[\bar{\phi^+}],
\]
where $\bar{\phi^+}$ is the mod-$\rho$ reduction of $\phi^+$, and $\rho$ is computed as follows:
\[
\rho = \gcd(\phi^+(d_1)+1, \dots, \phi^+(d_p)+1),
\]
where $d_1, \dots, d_p$ are the boundary components of $E^\circ$, oriented with $E^\circ$ to the left. 

Since we have exhibited a generating set for $\Mod(E^\circ)[\phi^+]$ consisting of Dehn twists about vanishing cycles for $E$, it follows that the image of $\Mod(E^\circ)[\phi^+]$ in $\Mod(E)$ is contained in $\Gamma_{\bd^+}$, and {\em a fortiori} in $\Mod(E)[\phi_{\bd^+}]$. Thus we have a containment
\[
\Mod(E)[\bar{\phi^+}] \le \Mod(E)[\phi_{\bd^+}].
\]
By \Cref{lemma:rspincontainment}, it follows that $r(\bd^+)$ divides $\rho$, and that $\phi_{\bd^+}$ is the mod-$r(\bd^+)$ reduction of $\bar{\phi^+}$, so moreover is the mod-$r(\bd^+)$ reduction of $\phi^+$. To complete the argument, we must show $\Gamma_{\bd^+}$ contains $\Mod(E)[\phi_{\bd^+}]$. According to \Cref{lemma:rspingencrit}, for this it suffices to exhibit a simple closed curve $c \subset E^\circ$ with $\phi^+(c) = \pm r(\bd^+)$ and with $T_c \in \Gamma_{\bd^+}$. Any boundary component $c$ of $\tilde C$ will do: by \Cref{lemma:tildeDframed}, $\phi^+(c) = \pm(r(\bd)+1) = \pm r(\bd^+)$, and each such $c$ is a vanishing cycle by \Cref{lemma:boundaryVC}, so that $T_c \in \Gamma_{\bd^+}$ as required. \qed

\bibliographystyle{alpha}
\bibliography{bibliography}

@article {walker,
    AUTHOR = {Walker, K. C.},
     TITLE = {Quotient groups of the fundamental groups of certain strata of
              the moduli space of quadratic differentials},
   JOURNAL = {Geom. Topol.},
  FJOURNAL = {Geometry \& Topology},
    VOLUME = {14},
      YEAR = {2010},
    NUMBER = {2},
     PAGES = {1129--1164},
      ISSN = {1465-3060,1364-0380},
   MRCLASS = {14D20 (14F35 30F30)},
  MRNUMBER = {2651550},
MRREVIEWER = {Mirroslav\ Tzanov\ Yotov},
       DOI = {10.2140/gt.2010.14.1129},
       URL = {https://doi.org/10.2140/gt.2010.14.1129},
}

@article {salterquintic,
    AUTHOR = {Salter, N.},
     TITLE = {On the monodromy group of the family of smooth quintic plane
              curves},
   JOURNAL = {Glasg. Math. J.},
  FJOURNAL = {Glasgow Mathematical Journal},
    VOLUME = {67},
      YEAR = {2025},
    NUMBER = {2},
     PAGES = {163--184},
      ISSN = {0017-0895,1469-509X},
   MRCLASS = {14D05 (57K20)},
  MRNUMBER = {4883969},
MRREVIEWER = {Igor\ V.\ Nikolaev},
       DOI = {10.1017/S0017089524000284},
       URL = {https://doi.org/10.1017/S0017089524000284},
}

@incollection {shimadapres,
    AUTHOR = {Shimada, I.},
     TITLE = {Fundamental groups of complements to hypersurfaces},
      NOTE = {Singularities and complex analytic geometry (Japanese) (Kyoto,
              1997)},
   JOURNAL = {S\=urikaisekikenky\=usho K\B oky\=uroku},
  FJOURNAL = {S\=urikaisekikenky\=usho K\B oky\=uroku},
    NUMBER = {1033},
      YEAR = {1998},
     PAGES = {27--33},
   MRCLASS = {14E20},
  MRNUMBER = {1660627},
MRREVIEWER = {Caryn\ Werner},
}

@article {shimada,
    AUTHOR = {Shimada, I.},
     TITLE = {Fundamental groups of algebraic fiber spaces},
   JOURNAL = {Comment. Math. Helv.},
  FJOURNAL = {Commentarii Mathematici Helvetici},
    VOLUME = {78},
      YEAR = {2003},
    NUMBER = {2},
     PAGES = {335--362},
      ISSN = {0010-2571,1420-8946},
   MRCLASS = {14F35},
  MRNUMBER = {1988200},
       DOI = {10.1007/s000140300014},
       URL = {https://doi.org/10.1007/s000140300014},
}

@article {QZ,
    AUTHOR = {Qiu, Y. and Zhou, Y.},
     TITLE = {Finite presentations for spherical/braid twist groups from
              decorated marked surfaces},
   JOURNAL = {J. Topol.},
  FJOURNAL = {Journal of Topology},
    VOLUME = {13},
      YEAR = {2020},
    NUMBER = {2},
     PAGES = {501--538},
      ISSN = {1753-8416,1753-8424},
   MRCLASS = {20F36 (13F60 16G20 20F05)},
  MRNUMBER = {4092774},
MRREVIEWER = {Thomas\ Koberda},
       DOI = {10.1112/topo.12135},
       URL = {https://doi.org/10.1112/topo.12135},
}

@article {strata3,
    AUTHOR = {Calderon, A. and Salter, N.},
     TITLE = {Framed mapping class groups and the monodromy of strata of
              abelian differentials},
   JOURNAL = {J. Eur. Math. Soc. (JEMS)},
  FJOURNAL = {Journal of the European Mathematical Society (JEMS)},
    VOLUME = {25},
      YEAR = {2023},
    NUMBER = {12},
     PAGES = {4719--4790},
      ISSN = {1435-9855,1435-9863},
   MRCLASS = {57K20 (30F30)},
  MRNUMBER = {4662301},
       DOI = {10.4171/jems/1290},
       URL = {https://doi.org/10.4171/jems/1290},
}

@article {strata2,
    AUTHOR = {Calderon, A. and Salter, N.},
     TITLE = {Higher spin mapping class groups and strata of abelian
              differentials over {T}eichm\"uller space},
   JOURNAL = {Adv. Math.},
  FJOURNAL = {Advances in Mathematics},
    VOLUME = {389},
      YEAR = {2021},
     PAGES = {Paper No. 107926, 56},
      ISSN = {0001-8708,1090-2082},
   MRCLASS = {57K20 (14H15 14M25 20F38 30F60)},
  MRNUMBER = {4289049},
MRREVIEWER = {Andrea\ Tamburelli},
       DOI = {10.1016/j.aim.2021.107926},
       URL = {https://doi.org/10.1016/j.aim.2021.107926},
}

@article {saltertoric,
    AUTHOR = {Salter, N.},
     TITLE = {Monodromy and vanishing cycles in toric surfaces},
   JOURNAL = {Invent. Math.},
  FJOURNAL = {Inventiones Mathematicae},
    VOLUME = {216},
      YEAR = {2019},
    NUMBER = {1},
     PAGES = {153--213},
      ISSN = {0020-9910,1432-1297},
   MRCLASS = {14D05 (14C20 14M25)},
  MRNUMBER = {3935040},
MRREVIEWER = {Wenfei\ Liu},
       DOI = {10.1007/s00222-018-0845-6},
       URL = {https://doi.org/10.1007/s00222-018-0845-6},
}

@article {CL1,
    AUTHOR = {Cr\'etois, R. and Lang, L.},
     TITLE = {The vanishing cycles of curves in toric surfaces {I}},
   JOURNAL = {Compos. Math.},
  FJOURNAL = {Compositio Mathematica},
    VOLUME = {154},
      YEAR = {2018},
    NUMBER = {8},
     PAGES = {1659--1697},
      ISSN = {0010-437X,1570-5846},
   MRCLASS = {14M25 (14T05 32S30)},
  MRNUMBER = {3830549},
MRREVIEWER = {Frank\ Sottile},
       DOI = {10.1112/s0010437x18007200},
       URL = {https://doi.org/10.1112/s0010437x18007200},
}

@article {CL2,
    AUTHOR = {Cr\'etois, R. and Lang, L.},
     TITLE = {The vanishing cycles of curves in toric surfaces {II}},
   JOURNAL = {J. Topol. Anal.},
  FJOURNAL = {Journal of Topology and Analysis},
    VOLUME = {11},
      YEAR = {2019},
    NUMBER = {4},
     PAGES = {909--927},
      ISSN = {1793-5253,1793-7167},
   MRCLASS = {14M25 (20F38 32S30)},
  MRNUMBER = {4040016},
MRREVIEWER = {Jongbaek\ Song},
       DOI = {10.1142/s1793525319500353},
       URL = {https://doi.org/10.1142/s1793525319500353},
}

@article{DolgLib,
AUTHOR = {Dolgachev, I. and Libgober, A. },
TITLE= {On the fundamental group of the complement to a
discriminant variety}, 
JOURNAL = { Lecture Notes in Mathematics },
VOLUME= {862},
YEAR = {1981},
DOI = {https://doi.org/10.1007/BFb0090888}
}

@incollection {donaldson,
    AUTHOR = {Donaldson, S. K.},
     TITLE = {Polynomials, vanishing cycles and {F}loer homology},
 BOOKTITLE = {Mathematics: frontiers and perspectives},
     PAGES = {55--64},
 PUBLISHER = {Amer. Math. Soc., Providence, RI},
      YEAR = {2000},
      ISBN = {0-8218-2070-2},
   MRCLASS = {57R17 (53D12 53D40 57R57 57R58)},
  MRNUMBER = {1754767},
MRREVIEWER = {Nikolai\ N.\ Saveliev},
}

@article {ishanpi1,
    AUTHOR = {Banerjee, I.},
     TITLE = {A {$\pi_1$} obstruction to having finite index monodromy and
              an unusual subgroup of infinite index in
              {$\text{Mod}(\Sigma_g)$}},
   JOURNAL = {Math. Ann.},
  FJOURNAL = {Mathematische Annalen},
    VOLUME = {392},
      YEAR = {2025},
    NUMBER = {4},
     PAGES = {5045--5064},
      ISSN = {0025-5831,1432-1807},
   MRCLASS = {14D05},
  MRNUMBER = {4958496},
       DOI = {10.1007/s00208-025-03213-7},
       URL = {https://doi.org/10.1007/s00208-025-03213-7},
}

@article {HJ,
    AUTHOR = {Humphries, S. P. and Johnson, D.},
     TITLE = {A generalization of winding number functions on surfaces},
   JOURNAL = {Proc. London Math. Soc. (3)},
  FJOURNAL = {Proceedings of the London Mathematical Society. Third Series},
    VOLUME = {58},
      YEAR = {1989},
    NUMBER = {2},
     PAGES = {366--386},
      ISSN = {0024-6115,1460-244X},
   MRCLASS = {57N05 (20F34 55M25)},
  MRNUMBER = {977482},
MRREVIEWER = {Ruth\ Charney},
       DOI = {10.1112/plms/s3-58.2.366},
       URL = {https://doi.org/10.1112/plms/s3-58.2.366},
}

@book{GKZ,
    AUTHOR = {Gelfand, Israel M. Kapranov, Mikhail M. and Zelevinsky, Igor} ,
    title = {Discriminants, Resultants, and Multidimensional Determinants },
    publisher ={Springer} ,
    year = {1994}
}

@book {FM,
    AUTHOR = {Farb, B. and Margalit, D.},
     TITLE = {A primer on mapping class groups},
    SERIES = {Princeton Mathematical Series},
    VOLUME = {49},
 PUBLISHER = {Princeton University Press, Princeton, NJ},
      YEAR = {2012},
     PAGES = {xiv+472},
      ISBN = {978-0-691-14794-9},
   MRCLASS = {57M50 (20F36 20F65 57M07 57N05)},
  MRNUMBER = {2850125},
MRREVIEWER = {Stephen\ P.\ Humphries},
}

@article {bellingerigodelle,
    AUTHOR = {Bellingeri, P. and Godelle, E.},
     TITLE = {Positive presentations of surface braid groups},
   JOURNAL = {J. Knot Theory Ramifications},
  FJOURNAL = {Journal of Knot Theory and its Ramifications},
    VOLUME = {16},
      YEAR = {2007},
    NUMBER = {9},
     PAGES = {1219--1233},
      ISSN = {0218-2165,1793-6527},
   MRCLASS = {20F36 (57M07)},
  MRNUMBER = {2375822},
MRREVIEWER = {Luis\ Paris},
       DOI = {10.1142/S0218216507005762},
       URL = {https://doi.org/10.1142/S0218216507005762},
}

@article {putmantrick,
    AUTHOR = {Putman, A.},
     TITLE = {A note on the connectivity of certain complexes associated to
              surfaces},
   JOURNAL = {Enseign. Math. (2)},
  FJOURNAL = {L'Enseignement Math\'ematique. Revue Internationale. 2e
              S\'erie},
    VOLUME = {54},
      YEAR = {2008},
    NUMBER = {3-4},
     PAGES = {287--301},
      ISSN = {0013-8584},
   MRCLASS = {57M50 (57M07)},
  MRNUMBER = {2478089},
MRREVIEWER = {Richard\ Weidmann},
}

@article {LV,
    AUTHOR = {Lawrence, B. and Venkatesh, A.},
     TITLE = {Diophantine problems and {$p$}-adic period mappings},
   JOURNAL = {Invent. Math.},
  FJOURNAL = {Inventiones Mathematicae},
    VOLUME = {221},
      YEAR = {2020},
    NUMBER = {3},
     PAGES = {893--999},
      ISSN = {0020-9910,1432-1297},
   MRCLASS = {11G35 (11F80 11J89 14G05)},
  MRNUMBER = {4132959},
MRREVIEWER = {Carlos\ de Vera-Piquero},
       DOI = {10.1007/s00222-020-00966-7},
       URL = {https://doi.org/10.1007/s00222-020-00966-7},
}

\end{document}